\documentclass[11pt,oneside, reqno]{amsart}

\usepackage{algorithm, algorithmic, amsfonts, amssymb, amsthm, amsmath, array, bm, braket, caption, changepage, comment, dsfont, enumitem, esint, fullpage, graphicx, mathrsfs, mathtools, microtype, multicol, textcomp, tikz, tikz-cd}
\usepackage[hidelinks]{hyperref}
\usepackage[foot]{amsaddr}
\usetikzlibrary{positioning}
\usetikzlibrary{shapes.geometric}

\usepackage{cleveref}

\DeclareMathOperator{\1}{\mathds{1}}

\DeclareMathOperator{\curl}{curl}
\DeclareMathOperator{\diverg}{div}
\DeclareMathOperator{\Leb}{Leb}
\DeclareMathOperator{\Lie}{Lie}

\newcommand{\pc}[1]{\left(#1\right)}

\newcommand{\pg}[1]{\left\{#1\right\}}

\newcommand{\eqdef}{\coloneqq}
\newcommand{\lref}[1]{Lemma~\ref{#1}}
\newcommand{\eref}[1]{(\ref{#1})}
\newcommand{\Nth}{N\nobreak\hspace{-.1em}\text{th}}

\def\ie{i.e., }

\def\theta{\vartheta}

\def\ZN{\mathbb Z_N^2}
\def\ZNp{\mathbb Z_N^2\times \{+,-\}}

\renewcommand{\epsilon}{\varepsilon}

\usepackage{caption}
\usepackage{subcaption}

\usepackage{mhequ}
\usepackage[normalem]{ulem}
\usepackage{cite}

\renewcommand{\aa}[1]{ #1}
\newcommand{\jcm}[1]{ #1}
\newcommand{\om}[1]{ #1}
\newcommand{\st}[1]{ }

\usepackage{todonotes}
\newcommand{\fref}[1]{Fig.~\ref{#1}}
\newcommand{\tref}[1]{Theorem~\ref{#1}}

\newcommand{\rref}[1]{Remark~\ref{#1}}
\renewcommand{\cref}[1]{Corollary~\ref{#1}}

\newcommand{\pref}[1]{Proposition~\ref{#1}}

\def\phi{\varphi}

\def\II{\mathcal I}
\newcommand{\pit}[1]{\phi_{#1}^\iota}
\newcommand{\Pit}[1]{\Phi_{#1}^\iota}
\def\QQ{\mathcal Q}
\def\EE{\mathcal E}
\def\AA{\mathcal A}
\def\AAc{\AA}

\def\ll{{\bm \ell}}
\def\jj{{\bm j}}
\def\kk{{\bm k}}
\def\llb{{\bm l}}

\def\jkl{{\bm j \bm k \bm \ell}}

\newcommand{\qa}[1]{a_{#1}}
\newcommand{\qb}[1]{b_{#1}}

\newtheoremstyle{}{}{}{}{}{}{}{ }{}

\numberwithin{equation}{section}

\theoremstyle{plain}
\newtheorem{theorem}{Theorem}[section]

\newtheorem{assumption}{Assumption}
\newtheorem{cor}[theorem]{Corollary}
\newtheorem{corollary}[theorem]{Corollary}
\newtheorem{definition}[theorem]{Definition}
\newtheorem{proposition}[theorem]{Proposition}

\newtheorem{lemma}[theorem]{Lemma}

\newtheorem{prop}[theorem]{Proposition}
\newtheorem{remark}[theorem]{Remark}

\newcommand{\QED}{\tag*{\qed}}

\title{Random splitting of fluid models:\\ Ergodicity and convergence}

\author{Andrea Agazzi$^\star$}
\author{Jonathan~C.~Mattingly$^{\star\diamond\dagger }$}
\author{Omar Melikechi$^\star$}

\email{jonm@math.duke.edu}

\address[$\star$]{Department of Mathematics, Duke  University, Durham (NC), USA.}
\address[$\dagger$]{Department of Statistical Science, Duke
  University, Durham (NC), USA.}
\address[$\diamond$]{Institute for Advanced Study, Princeton (NJ), USA.}

\begin{document}

\begin{abstract}
We introduce a  family of stochastic models motivated by the study of nonequilibrium steady states of fluid equations.
These models decompose the deterministic dynamics of interest into fundamental building blocks, \ie minimal
vector fields preserving some fundamental aspects of the original dynamics. Randomness is injected by sequentially following each vector field for a random amount of time. We show under general conditions that these random
dynamics possess a unique\jcm{, ergodic}  invariant measure and converge almost surely
 to the original, deterministic model in the small noise limit.
 We apply our construction to the Lorenz-96
 equations, often used in studies of chaos and data assimilation, and Galerkin approximations of the
 2D Euler and Navier-Stokes equations. An interesting feature of the
 models developed is that they apply directly to the conservative
 dynamics and not just those with excitation and dissipation.
\end{abstract}

\maketitle




\section{Introduction}


This paper studies the long time dynamics of fluid-like
equations that are kept out of equilibrium. Among the simplest examples of
fluid models displaying interesting out-of-equilibrium behavior (such as fluxes
across scales) are the two-dimensional Euler and incompressible Navier-Stokes
equations.
On the $2$-dimensional torus $\mathbb{T}$, \ie $\mathbb{T}\coloneqq
[0,2\pi]^2$ with periodic boundary conditions, the Navier-Stokes equations,
which model the flow of an incompressible fluid, are
\begin{align}\label{eq:2DNS}
	\begin{cases}
		\partial_tu+(u\cdot\nabla)u = -\nabla p +F + \nu
                \Delta u\,,\\
		\diverg(u) \coloneqq \nabla\cdot u = 0\,,
	\end{cases}
\end{align}
where $u:\mathbb{T}\times\mathbb{R}\to\mathbb{R}^2$ is the fluid velocity, $p:\mathbb{T}\times\mathbb{R}\to\mathbb{R}$ the fluid pressure,
\begin{align*}
	(u\cdot\nabla)u &= (u_1\partial_1u_1+u_2\partial_2u_1,
                          u_1\partial_1u_2+u_2\partial_2u_2),
                          \quad\text{and}\quad  \Delta u=
                          \partial_1^2 u_1 + \partial_2^2 u_2\,.
\end{align*}
Here $u=(u_1,u_2)$ and $\partial_j\eqdef  \partial_{x_j}$.  The viscosity
$\nu>0$ measures the strength of the dissipation introduced by the
Laplacian $\Delta$, and $F(x,t)$ is an external driving force whose role is to keep the
system from relaxing to the trivial state $u\equiv 0$.

By balancing the dissipative effect of
$\Delta u$, the forcing term allows the system to establish an out-of-equilibrium
steady state. Such statistical equilibria often develop fluxes across scales, a phenomenon whose study is an active area of research.
Often $F$ is taken to live
on only a few scales so that the flux out of those scales can be
studied \cite{EMattingly,mattingly03,hairermattingly06,KuksinNersesyanShirikyan20}. In practice, the
 forcing $F(x,t)$ is usually taken to be stochastic in space and time
for some stationary distribution which is typically white in time
\cite{FlandoliMaslowski95,EMattingly,DaPratoDebussche02,hairermattingly06}.
A common choice in the literature is
$F(x,t)=\sum \psi_k(x) \dot W_k(t)$ where each $\psi_k(x)$ is a fixed
spatial forcing \jcm{and $\{ \dot W_k(t) \} $ are a collection of  mutually
independent  white in time noise terms written
here as the formal derivative of a Brownian motion.} Stochastic
forcing serves multiple
 purposes in these settings. On one hand, as already
mentioned, it provides the energetic excitation which keeps the system
out of equilibrium and allows for the establishment of a nontrivial
statistical steady state. On the other hand, it
provides local agitation which, \jcm{modulo certain constraints}, ensures the existence of a unique statistical steady state to which the system converges \jcm{for most} initial conditions. In other words, it guarantees the forcing is sufficiently
varied and generic to ensure convergence to a single long time statistical
behavior of the system\jcm{, largely} independent of the system's initial configuration.

 This paper studies a class of processes, introduced in the next
  section, injecting randomness into the fluid models of interest
  while separating in a simple way the various roles served by
  noise in previous works. In particular, the randomness is used
primarily to ensure that when the dynamics is sufficiently generic, unique ergodicity\footnote{\om{See \Cref{def:uniquely_ergodic}.}} holds for a broad class of initial
conditions. This will free one to use a much less disruptive class of
forcing to keep the system out of equilibrium.
More specifically, the class of models introduced below have a number of desirable properties:
\begin{enumerate}
  \item They allow one to separate the effect of forcing, which keeps the system out of equilibrium, and stochastic agitation, which ensures the system has a unique long time statistical behavior.
	\item The stochastic agitation is strongly non-reversible since it is constructed from dynamics which only flow in the
          directions the original dynamics could already move.
\item The stochastic agitation preserves the conserved quantities of
  the original dynamics. This allows the properties of the
  (stochastic) conservative dynamics to be studied directly \jcm{rather than only as a limit of the
    forced-dissipated dynamics.}
\item The model dynamics will be
constructed as the composition of simple dynamics, isolating particular nonlinear
interactions
which are relatively intuitive and can be explicitly analyzed.
\end{enumerate}
By balancing between preservation of fundamental macroscopic properties of the original dynamics as in (3) and simplicity of the fundamental building blocks in our model dynamics as in (4), we expect the stochastic models introduced in this paper will provide meaningful physical and dynamical insight into nonequilibrium steady states of
 models such as \eref{eq:2DNS}.

Our decomposition \om{into fundamental building blocks} is partially motivated by the
classical stylized models of dynamics studied in depth at the dawn of
the theory of dynamical systems. \jcm{Examples include the doubling
  map, quadratic maps,
the Henon map, the Smale horseshoe, and extended systems like
coupled map lattices (see \cite{Devaney,Katok}
and references therein). The form of the decomposition is also motivated by the recent progress in proving ergodic properties of piecewise deterministic Markov processes (PDMPs) and their success as modeling and sampling tools. See for example
 \cite{Dufour11,Bakhtin,Benaim,Lawley14,Hurth15,Lawley15,Monmarche15,Benaim_Hurth_Strickler_2018,Hurth18,Bierkens19,Li_Liu_Cui_2017,Lawley21,Guillin21,Debussche21}.
 }


\subsection{A class of stochastic models}


We now introduce the general idea underlying the class of stochastic models, called \textit{random splitting}, that we study in this paper.  A more systematic definition of these models is deferred to Section~\ref{sec:randomSplittingGeneral}. Consider an ordinary differential equation (ODE)
\begin{align}\label{eq:originalODE}
  	\dot{x} &= V(x)
  		= \sum_{k=1}^n V_k(x)\,,
\end{align}
where $n\in\mathbb{N}$ and $V$ and $\{V_k\}_{k=1}^n$ are vector fields on $\mathbb{R}^d$. In what follows, we choose the $V_k$ so that the dynamics
\begin{align}\label{eq:splitODE}
	\dot{x} &= V_k(x)
\end{align}
are in some sense simpler than the dynamics corresponding
to~\eqref{eq:originalODE}. We then approximate the solution $\Psi_t
\colon x(0) \mapsto x(t)$ of~\eqref{eq:originalODE} with compositions
of the solution maps $\varphi^{(k)}_t \colon x(0) \mapsto x(t)$
of~\eqref{eq:splitODE}. This procedure is known as operator splitting in the
numerical analysis literature and is often used in numerical
simulations of various ordinary, partial, and stochastic differential
equations
\cite{MR1403645,MR1403562,MR1377257,Strang,Kusuoka_2004,Ninomiya_Victoir_2008,Ninomiya_Ninomiya_2009,Childs_Ostrander_Su_2019,Childs_Su_Tran_Wiebe_Zhu_2021}. Typically,
the goal is to leverage the fact that each of the dynamics
in~\eqref{eq:splitODE} is more computationally tractable
than~\eqref{eq:originalODE} to construct an efficient and accurate
numerical method. A variant of these models was also explored in the thesis
\cite{williamson19}.

Here our goal is related but slightly different. Specifically, instead of evolving each $\varphi^{(k)}$ for a fixed time $h$ as in traditional operator splitting methods, we evolve each of the $\varphi^{(k)}$ for a random time with mean $h$. Repeated composition then produces dynamics on $\mathcal{O}(1)$ times. The evolution times for
each $\varphi^{(k)}$, and over each cycle, will be identically
distributed and mutually independent, \om{which implies our models are Markovian}. As in the numerical analysis
context, we hope to leverage the simplified nature of each
$\varphi^{(k)}$, obtained from~\eqref{eq:splitODE}, to gain insight
into the complex dynamics of the composition of maps. We will also see
that as the mean evolution time $h \rightarrow 0$, the random
splitting associated to~\eqref{eq:originalODE} will almost surely converge to the
deterministic dynamics $\Psi_t$ on finite time
intervals. However, we
are most interested in studying the random splitting in its own right
and not \jcm{strictly} as a  approximation of~\eqref{eq:originalODE}. We will be
particularly interested in its long time behavior and qualitative
understanding of the stationary dynamics the random splitting produces
when $h>0$. \aa{More specifically, the property of the system we aim to establish is codified in }\om{ the following standard definition from the theory of Markov processes; the supporting definition of \textit{invariant measure} is given in the first paragraph of \Cref{sec:ergodicity} after the transition kernel of random splitting is explicitly introduced.

\begin{definition}\label{def:uniquely_ergodic}
A Markov process on a manifold $\mathcal{X}$ is {\normalfont uniquely ergodic} on $\mathcal{X}$ if its transition kernel admits
exactly one invariant probability measure  on $\mathcal{X}$.
\end{definition}

\noindent We note that the definition of the set $\mathcal X$ where the above property holds can be quite delicate. While in general there might not exist a $d$-dimensional manifold $\mathcal X$ in $\mathbb R^d$ on which the random splitting is uniquely ergodic (see for example \Cref{rem:uniqueErgodic}), in the examples below we will identify a family of manifolds of lowest co-dimension where the above definition applies.

\begin{remark}
The set of invariant probability measures for a Markov transition kernel is convex, and the extremal points of this set are precisely the ergodic invariant measures \cite{CornfeldFominSinai,HairerConvergence}. In particular, if the transition kernel admits exactly one invariant measure, then it is necessarily extremal and therefore ergodic. This explains the use of the term {\normalfont ergodic} in \Cref{def:uniquely_ergodic}.
\end{remark}
}


\subsection{Two motivating examples}\label{sec:MotivatingExamples}


In this paper, we consider two motivating examples: A conservative version of
the Lorenz-96 model and Galerkin
approximations of the vorticity
formulation of the 2D Euler equations. We then use these analyses to study the full Lorenz-96 model and Galerkin
approximations of the vorticity
formulation of 2D Navier-Stokes.


\subsubsection*{Lorenz-96}


Fix $n\geq 4$ and let $\{e_k\}_{k=1}^n$ denote the standard basis of $\mathbb{R}^n$. The Lorenz-96 model is
\begin{align}\label{lorenz96full}
	\dot{x} &= \sum_{k=1}^n\big((x_{k+1}-x_{k-2})x_{k-1}- \nu x_k+F_k\big)e_k
\end{align}
for $x\in\mathbb{R}^n$, $\nu>0$, and nonnegative constants $F_k$, where the indices are periodized via the identities $x_{-1}\coloneqq x_{n-1}$, $x_0\coloneqq x_n$, and $x_{n+1}\coloneqq
x_1$. The $-\nu x_k$ term in~\eqref{lorenz96full} represents dissipation in the
$k$th coordinate and $F_k$ is a forcing constant. Initially, we
study a variant of Lorenz-96, called \textit{conservative
  Lorenz-96}, obtained by removing the dissipation  and forcing
terms from Lorenz-96. That is,
\begin{align}\label{5.2}
	\dot{x} &= V(x)
		\coloneqq \sum_{k=1}^n(x_{k+1}-x_{k-2})x_{k-1}e_k\,.
\end{align}
We sometimes refer to the original Lorenz-96 model as the
\textit{forced  Lorenz-96} model to emphasize the forcing (though
the dissipation is equally important). For conservative Lorenz-96, we will decompose $V$ into a collection of simple rotations by observing that
\begin{align}\label{5.3}
	V(x) &= \sum_{k=1}^n V_k(x)
\end{align}
where $V_k(x)\coloneqq (x_{k+1}e_k-x_ke_{k+1})x_{k-1}$. The dynamics given by $\dot{x}=V_k(x)$ are easy to understand on
their own; any complex behavior comes from interactions of the
rotations. Importantly,  each $V_k$ is chosen to conserve, like $V$, the
system's \textit{energy}, which for Lorenz-96 is defined to be the square of the usual Euclidean norm, $\lVert x\rVert^2\coloneqq\sum_{k=1}^n x_k^2$.


\subsubsection*{2{D} Euler}


Returning to~\eqref{eq:2DNS},
we begin by defining the scalar vorticity  $q(x,t)=\curl u(x,t)$ of the
velocity field $u(x,t)$. Initially, we will
consider the Euler equations which are obtained from
\eqref{eq:2DNS} by  taking
$\nu=F=0$. Writing the equation for the $j$th
Fourier mode $q_j \in \mathbb{C}$, defined by $q(x,t)=\sum_j q_j(t) e_j(x)$ for
$e_j(x)\eqdef e^{ix\cdot j}$, and
$j \in\{ j\in \mathbb{Z}^2: |j| < N, j \neq 0\}$, we have
\begin{align}
  \label{eq:1}
  \dot{q_j} =  -\sum_{j+k+\ell=0} C_{k\ell} \bar q_k\bar q_\ell
\end{align}
for a constant $C_{k\ell}$ defined in
Section~\ref{sec:ConstructingSplitting}. We will see that this system
has two conserved quantities, the \textit{enstrophy}, $\sum_j |q_j|^2$, and the
\textit{energy}, $\sum_j |j|^{-2} |q_j|^2$. Notice that the definition of energy
differs between this equation and the Lorenz-96 model.

As in the  Lorenz-96 model, we introduce the simpler dynamics $\dot q=
V_{j k \ell}(q)$ where $V_{j k \ell}(q) =- C_{k\ell} \bar q_k \bar q_\ell e_j- C_{j\ell} \bar q_j\bar q_\ell e_k-
C_{jk} \bar q_j \bar q_k e_\ell$ and observe that
\begin{align*}
  V(q)=\sum_{j+k+\ell=0} V_{jk\ell}(q)\,.
\end{align*}
We will see in Section~\ref{sec:Euler} that with this choice of splitting the
dynamics $\dot{q}= V_{jk\ell}(q)$, like the original system $V(q)$,  preserves the important
physical quantities of enstrophy and energy.
\begin{remark}
In Section~\ref{sec:Euler}, we further simplify these complex-valued dynamics by
projecting onto a real basis. The current choice is sufficient for an
introductory discussion.
\end{remark}

\begin{remark}
Our results do not focus on establishing minimal hypoellipticity assumptions \om{for our models of Lorenz-96, 2D Euler and 2D Navier-Stokes}; the stochastic agitation we use is more global than the minimal hypoellipticity forcing considered in \cite{EMattingly,hairermattingly06}. We hope this will allow us to progress further than with previous models while preserving much of the physically interesting dynamics.
\end{remark}

\begin{remark}
  It is important to emphasize that, with regard to unique ergodicity,
  the main role of the forcing\jcm{, when included,} is only to
  destroy the fixed points and other low-dimensional invariant structures of the original flows and not to provide
  the stochastic mixing which ensures \jcm{the existence of a
  unique, ergodic measure to which the system's statistics converge}.
The \jcm{stochastic mixing is largely} provided by
  the random splitting and is in contrast to the results in
  \cite{EMattingly,KuksinShirikyan00,EWeinanMattingly01,KuksinShirikyan03,
    hairermattingly06,KuksinNersesyanShirikyan20,Bedrossian21}.
\end{remark}

\jcm{\begin{remark}\label{rem:uniqueErgodic}
  When considering conservative versions of our split dynamics (those without any explicit dissipation or body forcing),
 we cannot expect there to be a unique invariant measure for
 the system. In particular, since the dynamics will be constrained
 to level sets of the conserved quantities, there will be at
 least one invariant measure per level set. Furthermore,
 we will see that even on such constraint level sets there can be
 multiple ergodic invariant measures. Most will correspond to
fixed points of the original dynamics and other lower-dimensional
 invariant structures. However we will see, in the two examples considered, that when our family of switched
 vector fields is sufficiently rich, there will be a unique ergodic invariant
 measure which is absolutely continuous with respect to the volume
 measure on the level set.
 This implies that in these examples, there
 is a unique ergodic
 invariant measure concentrated on a set of full measure inside
 each constraint level set. In this sense, we will demonstrate a form of uniqueness which aligns with the form of unique
   ergodicity often proven in the smooth deterministic dynamics setting, i.e., that there is only one invariant measure absolutely continuous
 with respect to the setting's natural Lebesgue measure.
\end{remark}}


\subsection*{Organization of paper}


In Section~\ref{sec:randomSplittingGeneral}, we introduce random
splitting \om{and its state spaces, called \textit{$\mathcal{V}$-orbits}. In Section~\ref{sec:ergodicity}, we give conditions for random splitting to be uniquely ergodic on a $\mathcal{V}$-orbit.}
In Section~\ref{sec:Convergence}, we show under general conditions that random splitting converges
to its deterministic counterpart \eqref{eq:originalODE} on finite time
intervals both in terms of its transition kernel and almost surely as
the average time step $h$ goes to zero. In Sections~\ref{sec:Lorenz}
and~\ref{sec:Euler}, we construct random splittings of conservative
Lorenz-96 and Galerkin approximations of 2D Euler and apply the
preceding results to show these splittings are uniquely ergodic
and converge
\jcm{on finite time intervals as $h \to 0$.}
  \aa{In doing so we show each system has a }\jcm{
  unique invariant
measure that is absolutely continuous (with respect to the volume measure) on
the set defined by a given choice of the conserved quantities.} In Section~\ref{sec:ForceAndDissip}, we
consider the Lorenz-96 and Euler models when fixed forcing and
dissipation are added. \jcm{When appropriate dissipation is chosen,  the
latter model corresponds to a random splitting of Galerkin approximations of 2D
Navier-Stokes.} We again construct random splittings of these
models, prove convergence, and show that if the forcing is not aligned
with the equations' invariant structures (such as fixed points) then
both randomly split Lorenz-96 and Galerkin approximations of 2D
Navier-Stokes \om{have a unique invariant measure and the
  distribution starting from any initial condition converges exponentially to this
  measure.}


\section{Random splitting in a general setting}\label{sec:randomSplittingGeneral}


Let $\om{\mathcal{V}\coloneqq}\{V_k\}_{k=1}^n$ be \om{a family of
  complete}\footnote{\jcm{A vector field is \textit{complete} if its flow curve starting from any point exists for all time.
  }}, $\mathcal{C}^2$ vector fields\footnote{We use calligraphic $\mathcal{C}^k$ for $k$-times continuously differentiable maps throughout to avoid confusion with constants which are often denoted by normal script $C$ (for example, the constants $C_{jk}$ in 2D Euler).}  on $\mathbb{R}^d$ and set
\begin{align}\label{2.1}
	V &\coloneqq \sum_{k=1}^n V_k\,.
\end{align}
Denote the flow of $\dot{x}=V(x)$ by $\Psi$ and the flow of
$\dot{x}=V_k(x)$ by $\varphi^{(k)}$. $\Psi$ is the \textit{true dynamics}. To construct a random dynamics approximating $\Psi$, fix $h>0$, let $\tau=(\tau_k)_{k=1}^\infty$ be a sequence of independent exponential random variables with mean $1$, and set $h\tau\coloneqq (h\tau_k)_{k=1}^\infty$. The approximating dynamics, henceforth referred to as the \textit{random splitting associated to \om{$\mathcal{V}$}} or just \textit{random splitting} for short, is the Markov chain $\{\Phi^m_{h\tau}\}_{m=0}^\infty$ defined by $\Phi^0_{h\tau}\coloneqq I$ and, for $m>0$,
\begin{align}\label{eq:Phi}
	\Phi^m_{h\tau} &\coloneqq \varphi^{(n)}_{h\tau_{mn}}\circ\cdots\circ\varphi^{(1)}_{h\tau_{(m-1)n+1}}(\Phi^{m-1}_{h\tau}),
\end{align}
where $I$ is the identity on $\mathbb{R}^d$, $\Phi\coloneqq\varphi^{(n)}\circ\cdots\circ\varphi^{(1)}$, and $\Phi^m$ is the $m$-fold composition of $\Phi$. Note that $h\tau_k\overset{\scriptscriptstyle{iid}}{\sim}\text{Exp}(1/h)$. Therefore, starting from the current step, the next step of the chain is obtained by following each $V_k$ for $\text{Exp}(1/h)$ time in order from $k=1$ to $n$. The chain is Markovian because the random times are independent. Its transition kernel $P_h$ acts on measurable functions $f:\mathbb{R}^d\to\mathbb{R}$ via
\begin{align}\label{2.2}
	P_hf(x) &= \mathbb{E}\big(f(\Phi_{h\tau}(x))\big)
		= \int_{\mathbb{R}^n_+}f(\Phi_{ht}(x))e^{-\sum_{k=1}^n t_k} dt
\end{align}
where $\mathbb{R}_+\coloneqq(0,\infty)$, $t=(t_1,\dots,t_n)$, and $dt=dt_1\cdots dt_n$.

\begin{remark}\label{rmk2.1}
Throughout this paper the superscripts $k$ in $\varphi^{(k)}$ and subscripts $k$ in $V_k$ are understood to be taken modulo $n$ if $k\ \text{mod}\ n\neq 0$ and to be $n$ otherwise. For example, if $n=3$,
\begin{align*}
	\varphi^{(6)}\circ\varphi^{(5)}\circ\varphi^{(4)}\circ\varphi^{(3)}\circ\varphi^{(2)}\circ\varphi^{(1)} &= \varphi^{(3)}\circ\varphi^{(2)}\circ\varphi^{(1)}\circ\varphi^{(3)}\circ\varphi^{(2)}\circ\varphi^{(1)}.
\end{align*}
Also, the $t$ in $\Phi^m_t$ is always a sequence $t=(t_1,\dots, t_{mn})$ or, more generally, $t=(t_k)_{k=1}^\infty$, so that
\begin{align*}
	\Phi^m_t(x) &= \varphi^{(n)}_{t_{mn}}\circ\cdots\circ\varphi^{(1)}_{t_1}(x)\,.
\end{align*}
Note that the above is a composition of $mn$ flows, as in~\eqref{eq:Phi}.
\end{remark}

\begin{remark}\label{rmk2.2}
All results in this paper remain true if at each step we
randomly permute indices in the composition $\Phi$. That is, given a
current state $x$, the next step is
$\varphi^{(\sigma(n))}_{h\tau_n}\circ\cdots\circ\varphi^{(\sigma(1))}_{h\tau_1}(x)$
where $\sigma$ is a random permutation of $\{1,\dots,n\}$. This yields both additional randomness and an avenue to higher order
approximations of the true dynamics
\cite{Kusuoka_2004,Ninomiya_Victoir_2008,Ninomiya_Ninomiya_2009,Childs_Ostrander_Su_2019,Childs_Su_Tran_Wiebe_Zhu_2021}. We forgo this more general setting however
to keep exposition more approachable and notationally light.
\end{remark}

\begin{remark}\label{r:exponential}
  The times are assumed exponentially distributed for convenience. All results extend to any
  distribution on $[0,\infty)$ with positive density \aa{on $(0,\epsilon)$ for some $\epsilon >0$} and
  exponential tails. \aa{The second condition\jcm{, which is not
      sharp,} guarantees sufficient concentration of averages of
    random flow times $\tau_i$ in Lemmas~\ref{lemA1} and \ref{lemA2}
    and is required for the convergence results as $h \to 0$ in
    Section~\ref{sec:Convergence}. The first condition is used in
    Sections~\ref{sec:Lorenz} and \ref{sec:Euler} to guarantee
    sufficient flexibility in the trajectories of the split
    systems of interest to establish \jcm{the global irreducibility needed}
    for ergodicity.}
\end{remark}


\subsection{$\mathcal{V}$-Orbits}\label{sec:orbits}


\om{Throughout this paper we often restrict attention to certain subsets of $\mathbb{R}^d$ affiliated with the family of vector fields $\mathcal{V}$. Specifically, for each $x$ in $\mathbb{R}^d$ define the \textit{$\mathcal{V}$-orbit of $x$} by}
\begin{align}\label{eq:X}
	\mathcal{X}(x) &\coloneqq \big\{ \Phi^m_t(x) : m\geq 0, t\in\mathbb{R}^{mn}\big\}.
\end{align}
This is the set of points in $\mathbb{R}^d$ that can be reached by the split dynamics starting from $x$ in any finite number of steps and over arbitrary \om{positive and negative} times. \om{$\mathcal{X}(x)$ is well-defined since the $V_k$ are complete. Furthermore, since the time vectors $t$ in \eqref{eq:X} admit coordinates that are 0,
\begin{align*}
	\mathcal{X}(x) &= \big\{\varphi^{(i_m)}_{t_{i_m}}\circ\cdots\circ\varphi^{(i_1)}_{t_{i_1}}(x):m\in\mathbb{N}, 1\leq i_j\leq n, t_{i_j}\in\mathbb{R}\big\}.
\end{align*}
Hence \eqref{eq:X} agrees with the definition of $\mathcal{V}$-orbits from control theory \cite{jurdjevic, sussmann}. Note the collection $\{\mathcal{X}(x):x\in\mathbb{R}^d\}$ partitions $\mathbb{R}^d$ and if the random splitting associated to $\mathcal{V}$ starts in $\mathcal{X}(x)$ then it stays in $\mathcal{X}(x)$ for all time. Therefore the random splitting $\{\Phi^m_{h\tau}\}$ previously defined on $\mathbb{R}^d$ also defines a Markov chain on $\mathcal{X}(x)$ whenever it starts in $\mathcal{X}(x)$, and its transition kernel $P_h$ acts on measurable functions $f:\mathcal{X}(x)\to\mathbb{R}$ as in \eqref{2.2}. When $x$ is arbitrary or clear from context, we denote $\mathcal{X}(x)$ by $\mathcal{X}$. A classic result from geometric control theory, sometimes called \textit{the orbit theorem}, says if every $V_k$ in $\mathcal{V}$ is $\mathcal{C}^r$ for some $1\leq r\leq \infty$ (respectively, analytic\footnote{Throughout this paper \textit{analytic} means \textit{real-analytic}.}), then every $\mathcal{X}$ is a $\mathcal{C}^r$ (respectively, analytic) submanifold of $\mathbb{R}^d$ \cite{jurdjevic}. In particular, each $\mathcal{X}$ has a Riemannian structure induced by the Euclidean structure on $\mathbb{R}^d$ and an associated volume form, henceforth denoted $\lambda$, sometimes called \textit{Hausdorff} or \textit{Lebesgue measure on $\mathcal{X}$}, which serves as our reference measure on $\mathcal{X}$.
}


\section{Ergodicity}\label{sec:ergodicity}


\om{Let $\mathcal{V}\coloneqq\{V_k\}_{k=1}^n$ be a family of complete, $\mathcal{C}^2$ vector fields on $\mathbb{R}^d$ as before and fix a $p$-dimensional $\mathcal{V}$-orbit $\mathcal{X}$. Also fix $h>0$ and let $P_h$ be the transition kernel of the associated random splitting on $\mathcal{X}$.} A measure $\mu$ on $\mathcal{X}$ is \textit{$P_h$-invariant} if $\mu P_h=\mu$ where $\mu P_h$ is defined by
\begin{align}\label{eq:action}
	\mu P_h f &\coloneqq \int_\mathcal{X} P_hf(x)\mu(dx)
\end{align}
for all bounded, measurable functions
$f:\mathcal{X}\to\mathbb{R}$. The main result of this section is

\begin{theorem}\label{thrm:Ergodicity}
If there exists $x_*$ in $\mathcal{X}$ such that for all $x$ in
$\mathcal{X}$ there is an $m$ in $\mathbb N$ and $t$ in $\mathbb R_+^{mn}$ with $\Phi^m(x,t)=x_*$ and
$D_t\Phi^m(x,t):T_{t}\mathbb{R}^{mn}_+\to T_{x_*}\mathcal{X}$
surjective, then $P_h$ has at most one invariant measure on
$\mathcal{X}$. \om{Moreover, if such a measure exists, it is
  absolutely continuous with respect to the volume form \jcm{on $\mathcal{X}$}.}
\end{theorem}

\noindent Here $T_{x_*}\mathcal{X}$ is the
tangent space of $\mathcal{X}$ at $x_*$. The proof of
Theorem~\ref{thrm:Ergodicity}  follows from the classical minorization
condition \cite{Nummelin,MeynTweedie,MattinglyStuart,yetMH} given by
the following result, which appears in \cite[Lemma~6.3]{Benaim}.

\begin{lemma}\label{lem4.1}
\om{Let $p\leq m$ and let $F:\mathcal{X}\times U\to\mathcal{X}$ be $\mathcal{C}^1$, where $U$ is an open subset of $\mathbb{R}^m$.} Suppose $\tau$ is a $U$-valued random variable with continuous density $\rho$. If for some $(x,t)$ in $\mathcal{X}\times U$ the map $D_tF(x,t)$ is surjective and $\rho$ is bounded below by $c_0>0$ on a neighborhood of $t$, then there exists a constant $c>0$ and neighborhoods $U_x$ of $x$ and $U_*$ of $x_*\coloneqq F(x,t)$ such that
\begin{align}\label{4.1}
	\mathbb{P}\big(F(y,\tau)\in B\big) &\geq c\lambda(B\cap U_*)
\end{align}
for all \om{$y$ in $U_x$ and $B$ in the Borel $\sigma$-algebra $\mathcal{B}(\mathcal{X})$ of $\mathcal{X}$ (recall $\lambda$ is the volume form on $\mathcal{X}$).}
\end{lemma}

\begin{remark}\label{remark4.1}
In our setting, $U=\mathbb{R}^{mn}_+$, $F=\Phi^m:\mathcal{X}\times\mathbb{R}^{mn}_+\to\mathcal{X}$, and $\tau=(\tau_1,\dots,\tau_{mn})$ with the $\tau_k$ independent exponential random variables with mean $h$. In this case, if $x_*=\Phi^m(x,t)$ for some $t$ with $D_t\Phi^m(x,t)$ surjective, then Lemma~\ref{lem4.1} guarantees the existence of a constant $c>0$ and neighborhoods $U_x$ of $x$ and $U_*$ of $x_*$ such that, for all $y$ in $U_x$ and $B$ in $\mathcal{B}(\mathcal{X})$,
\begin{align}\label{4.2}
	P^m(y,B) &\geq c\lambda(B\cap U_*)\,.
\end{align}
\end{remark}

\noindent\begin{proof}[Proof of Theorem~\ref{thrm:Ergodicity}]
The proof is by contradiction. Suppose $\mu_1$ and $\mu_2$ are
distinct $P_h$-invariant probability measures. Assume without loss of
generality both $\mu_i$ are ergodic and therefore mutually singular \cite{CornfeldFominSinai,Kifer86}. Then there exist disjoint measurable sets $A_1$ and $A_2$ partitioning $\mathcal{X}$ such that $\mu_i(B)=\mu_i(B\cap A_i)$ for all $B$ in $\mathcal{B}(\mathcal{X})$. Fix $x_i$ in the support of $\mu_i$ so, by definition, $\mu_i$ gives positive measure to every neighborhood of $x_i$. By hypothesis and Remark~\ref{remark4.1} there exist $c_i>0$, $m_i\in\mathbb{N}$, and neighborhoods $U_i$ of $x_i$ and $U_*$ of $x_*$ such that $P_h^{m_i}(x,\cdot)\geq c_i\lambda(\cdot\cap U_*)$ for all $x$ in $U_i$. So,
\begin{align}\label{eq:abs_cnts}
	\mu_i(B) &= \mu_iP_h^{m_i}(B)
		\geq \int_{U_i} P_h^{m_i}(x,B)\mu_i(dx)
		\geq c_i\lambda(B\cap U_*)\mu_i(U_i)
\end{align}
for all $B$ in $\mathcal{B}(\mathcal{X})$. In particular, $\mu_i(B)=0$ implies $\lambda(B\cap U_*)=0$ since $c_i$ and $\mu_i(U_i)$ are strictly positive. But $\mu_1(A_2\cap U_*)=\mu_2(A_1\cap U_*)=0$ and hence
\begin{align*}
	0 &< \lambda(U_*)
		= \lambda(A_1\cap U_*)+\lambda(A_2\cap U_*)
		= 0,
\end{align*}
\jcm{which is} a contradiction. \om{Absolute continuity of the $P_h$-invariant measure $\mu$, provided it exists, follows from
  uniqueness together with the fact that the absolutely
  continuous part, $\mu_{ac}$, and singular part, $\mu_s$, of $\mu$
  are $P_h$-invariant whenever $\mu$ is \cite[Proposition
  2.7]{Benaim}. Specifically, since $\mu_{ac}$ and $\mu_s$ are
  $P_h$-invariant and there can be at most one $P_h$-invariant
  probability measure, either $\mu_{ac}$ or $\mu_s$ is identically
  zero. Since $\mu_{ac}$ is nonzero by \eqref{eq:abs_cnts}, it follows
  that $\mu_s=0$ and therefore $\mu=\mu_{ac}$.}
\end{proof}

\jcm{
  \begin{remark}\label{rem:invariant}
    The invariant measure $\mu$, which we
    defined as a
    fixed point of the left action of the Markov semigroup $P$, is often called a
    stationary measure. This is since the sequence of random variables
    generated by the Markov process starting from an initial condition
    distributed according to $\mu$ will be stationary. This helps distinguish from the invariant measure of the skew flow
     $(x, \tau)\mapsto (\Psi_{h\tau}(x), \theta \tau)$ where
    the shift $\theta$ is defined by
    $\theta \tau: \tau= (\tau_1,\tau_2,\cdots) \mapsto
    (\tau_{n+1},\tau_{n+2},\cdots)$. The skew perspective captures more information
    about the dynamics and is preferred for many
    questions. However, we will not pursue it here as it complicates
    the simple picture we explore in this note.
  \end{remark}
  }

\om{

\subsection{The Lie bracket condition}\label{sec:lie}


Let $\mathfrak{X}(\mathcal{X})$ be the Lie algebra of smooth vector fields on $\mathcal{X}$ and assume throughout this subsection the vector fields in $\mathcal{V}$ are smooth. Then the smallest subalgebra $\Lie(\mathcal{V})$ of $\mathfrak{X}(\mathcal{X})$ containing $\mathcal{V}$ is well-defined, and for each $x$ in $\mathcal{X}$ the collection $\Lie_x(\mathcal{V})\coloneqq\{V(x):V\in\Lie(\mathcal{V})\}$ is a subspace of the tangent space $T_x\mathcal{X}$ at $x$.

\begin{definition}
The {\normalfont Lie bracket condition} holds at $x$ in $\mathcal{X}$ if $\Lie_x(\mathcal{V})=T_x\mathcal{X}$.
\end{definition}

\noindent The Lie bracket condition is called the \textit{weak bracket condition} in \cite{Benaim} and \textit{Condition B} in \cite{Bakhtin}. Both papers also consider a \textit{strong bracket condition} (\textit{Condition A}) which is used for results about continuous time Markov processes and is therefore not needed here. The Lie bracket condition has the following important consequence. Note $\mathbb{R}_+\coloneqq(0,\infty)$ throughout this paper.

\begin{theorem}\label{thrm:submersion}
If the Lie bracket condition holds at a point $x_*$ in $\mathcal{X}$ then for every neighborhood $U$ of $x_*$ in $\mathcal{X}$ and every $T>0$ there exists an $x$ in $U$, an $m$, and a $t$ in $\mathbb{R}^{mn}_+$ such that $\sum_{k=1}^{mn} t_k\leq T$ and $t\mapsto \Phi^m(x_*,t)=x$ is a submersion at $t$, i.e. $D_t\Phi^m(x_*,t):T_t\mathbb{R}^{mn}\to T_x\mathcal{X}$ is surjective.
\end{theorem}

A version of \Cref{thrm:submersion} appears as Theorem 3.1 in \cite{jurdjevic}; the equivalent version given here is better suited to random splitting and other classes of piecewise deterministic Markov processes. See Theorem 5 in \cite{Bakhtin} and Theorem 4.4 in \cite{Benaim} and their corresponding discussions for details. Intuitively, Theorem \ref{thrm:submersion} says that if the Lie bracket condition holds at $x_*$ then, as a consequence of surjectivity, the random splitting can move in any infinitesimal direction from $x_*$ in arbitrarily small positive times. The next result is an immediate consequence of Theorems \ref{thrm:Ergodicity} and \ref{thrm:submersion}.

\begin{cor}\label{cor:ergodic}
Suppose there is an $x_*$ in $\mathcal{X}$ at which the Lie bracket condition holds and such that for every $x$ in $\mathcal{X}$ there is an $m \in \mathbb N$ and a $t \in \mathbb{R}^{mn}_+$ satisfying $\Phi^m(x,t)=x_*$. Then $P_h$ has at most one invariant measure on $\mathcal{X}$. Furthermore, if such a measure exists, it is absolutely continuous with respect to the volume form $\lambda$.
\end{cor}

\noindent One benefit of \Cref{cor:ergodic} is that it replaces the need to check the surjectivity assumption of \Cref{thrm:Ergodicity}, which can be challenging in practice, with the verification of the Lie bracket condition. The next result provides a further convenience in the analytic setting which will be used in the specific examples considered below. See \cite{jurdjevic, Nagano} for further discussion and proof.

\begin{theorem}[Nagano]\label{thrm:nagano}
Suppose the vector fields in $\mathcal{V}$ are analytic. If the Lie bracket condition holds at any point in $\mathcal{X}$, then it holds at every point in $\mathcal{X}$.
\end{theorem}

\begin{cor}\label{cor:nagano}
Suppose the vector fields in $\mathcal{V}$ are analytic and there is a point $x_*$ in $\mathcal{X}$ such that for every $x$ in $\mathcal{X}$ there is an $m \in \mathbb N$ and a $t \in \mathbb{R}^{mn}_+$ satisfying $\Phi^m(x,t)=x_*$. If the Lie bracket condition holds at any point in $\mathcal{X}$, then $P_h$ has at most one invariant measure on $\mathcal{X}$. Furthermore, if such a measure exists, it is absolutely continuous with respect to the volume form $\lambda$.
\end{cor}

\begin{proof}
Since the Lie bracket condition holds at one point in $\mathcal{X}$, it also holds at $x_*$ by Nagano's theorem. The result follows by \Cref{cor:ergodic}.
\end{proof}
}


\section{Convergence as mean time step goes to zero}\label{sec:Convergence}


A well-known result in the operator splitting literature is that the error incurred in approximating $\Psi$ by the deterministic splitting scheme $\Phi_h=\varphi^{(n)}_h\circ\cdots\circ\varphi^{(1)}_h$ is $\mathcal{O}(h)$ \cite{Strang}. That is,  \aa{$\Phi_h$} converges to the true dynamics $\Psi$ at worst linearly in $h$ as $h\to 0$. In this section we give analogous results for random splitting; the pluralized ``results" reflects that with randomness comes several different notions of convergence. Specifically, we give two main results. First, as in the deterministic case, the transition kernel $P_h$ of random splitting converges to the true dynamics linearly in $h$ as $h\to 0$. Second, random splitting converges almost-surely to the true dynamics
as $h\to 0$.  \jcm{Each case requires a slightly different notion of
$\mathcal{O}(h)$.} These statements are made precise in Theorems~\ref{thrm3.1} and~\ref{thrm3.2}, respectively, but to make sense of them we first introduce the appropriate setting.

The following assumption on $\mathcal{V}$-orbits is used throughout this section.

\begin{assumption}\label{assump:Bounded}
$\mathcal{X}(x)$ is bounded for each $x$ in $\mathbb{R}^d$.
\end{assumption}

\noindent Since the vector fields $V_k$ are assumed $\mathcal{C}^2$, Assumption~\ref{assump:Bounded} implies the $V_k$ are bounded with bounded first and second derivatives on every $\mathcal{X}$. In particular,
\begin{align}\label{eq:Bounded}
	C_*(x_0) &\coloneqq \sup_{x\in\mathcal{X}(x_0)}\left\{\lVert V_k(x)\rVert, \lVert DV_k(x)\rVert, \lVert D^2V_k(x)\rVert : 1\leq k\leq n\right\}
		< \infty\,,
\end{align}
where $\lVert V_k(x)\rVert$ is the usual Euclidean norm, $\lVert
DV_k(x)\rVert$ is the operator norm of the linear map
$DV_k(x):\mathbb{R}^d\to\mathbb{R}^d$, and $\lVert D^2V_k(x)\rVert$ is
the operator norm of the bilinear map
$D^2V_k(x):\mathbb{R}^d\times\mathbb{R}^d\to\mathbb{R}^d$.

For a positive integer $k$ let $\mathcal{C}^k(\mathcal{X})$ be the space of $k$-times continuously differentiable functions $f:\mathcal{X}\to\mathbb{R}$. For $f$ in $\mathcal{C}^k(\mathcal{X})$ and $\ell\leq k$, the $\ell$th derivative $D^\ell f(x)$ of $f$ at $x$ is a multilinear operator from $\otimes_1^\ell T_x\mathcal{X}$ to $\mathbb{R}$. The operator norm of $D^\ell f(x)$ is then
\begin{align*}
	\lVert D^\ell f(x)\rVert &\coloneqq \sup_{\lVert\eta\rVert=1}\left\{\lvert D^\ell f(x)\eta\rvert\right\}\,,
\end{align*}
where $\eta\in \otimes_1^\ell T_x\mathcal{X}$. Defining $D^0 f(x) \coloneqq f(x)$, this in turn induces a norm on $\mathcal{C}^k(\mathcal{X})$ given by
\begin{align*}
	\lVert f\rVert_k &\coloneqq \sup_{x\in\mathcal{X}}\left\{\lVert D^\ell f(x)\rVert : 0\leq \ell\leq k\right\}.
\end{align*}
The corresponding operator norm is denoted $\lVert\cdot\rVert_{k\to k}$. More generally, for any $k$ and $\ell$ define a norm $\lVert\cdot\rVert_{k\to \ell}$ on the space of linear operators $L:\mathcal{C}^k(\mathcal{X})\to\mathcal{C}^\ell(\mathcal{X})$ by
\begin{align*}
	\lVert L\rVert_{k\to \ell} &\coloneqq \sup_{\lVert f\rVert_k=1} \lVert Lf\rVert_\ell\,.
\end{align*}
We make frequent use of the \textit{submultiplicity} of $\lVert\cdot\rVert_{k\to\ell}$. Namely, if $A$ and $B$ are bounded linear operators from $\mathcal{C}^j(\mathcal{X})$ to $\mathcal{C}^k(\mathcal{X})$ and from $\mathcal{C}^k(\mathcal{X})$ to $\mathcal{C}^\ell(\mathcal{X})$, respectively, then
\begin{align*}
	\lVert BA\rVert_{j\to\ell} &\leq \lVert B\rVert_{k\to\ell}\lVert A\rVert_{j\to k}\,.
\end{align*}

The results below are stated in terms of semigroups of the flows $\Psi$ and $\varphi^{(j)}$, which are $\mathcal{C}^2$ by assumption. Hence for all $k\leq 2$ the semigroup $\{S_t\}_{t\geq 0}$ corresponding to $\Psi$ acts on $f\in\mathcal{C}^k(\mathcal{X})$ via
\begin{align}\label{3.20}
	S_tf(x) &= e^{tV}f(x)
		= f(\Psi_t(x))
\end{align}
and, similarly, the semigroup $\{\widetilde{S}^{(j)}_t\}_{t\geq 0}$ corresponding to $\varphi^{(j)}$ is given by
\begin{align}\label{3.21}
	\widetilde{S}^{(j)}_tf(x) &= e^{tV_j}f(x)
		= f(\varphi^{(j)}_t(x))\,.
\end{align}
In particular, $m$ steps of random splitting corresponds to $\widetilde{S}^m_{h\tau} \coloneqq \widetilde{S}^{(1)}_{h\tau_1}\cdots\widetilde{S}^{(mn)}_{h\tau_{mn}}$ with superscripts taken as in Remark~\ref{rmk2.2}. The transition kernel $P^m_h$ and semigroup composition $\widetilde{S}^m_{h\tau}$ are related via
\begin{align}\label{3.2}
P^m_hf=\mathbb{E}(f(\Phi^m_{h\tau}))=\mathbb{E}(\widetilde{S}^m_{h\tau}f)\,.
\end{align}
With the above notation we now present the two main results of this section, Theorems \ref{thrm3.1} and \ref{thrm3.2}, which follow from Lemmas \ref{lem3.1} and \ref{lem3.2}, respectively. The full proofs of both lemmas are given in the Appendix, but we discuss the general idea behind each at the end of this section.

\begin{theorem}\label{thrm3.1}
Suppose Assumption~\ref{assump:Bounded} holds and fix $t>0$. For all $h$ sufficiently small and satisfying $mh=t$ for some $m\in\mathbb{N}$, there exists a constant $C(t)$ depending on $t$ but not on $h$ such that
\begin{align}\label{3.1}
	\lVert P_h^m-S_t\rVert_{2\to 0} &\leq C(t)h.
\end{align}
\end{theorem}

\begin{lemma}\label{lem3.1}
If Assumption~\ref{assump:Bounded} holds then there exists a constant $C$ such that
\begin{align}\label{3.3}
	\lVert P_h-S_h\rVert_{2\to 0} &\leq Ch^2
\end{align}
for all $h$ sufficiently small.
\end{lemma}

\jcm{
\noindent Recalling from~\eqref{3.2} that
$P_h=\mathbb{E}(\widetilde{S}_{h\tau}^1)$, informally Lemma~\ref{lem3.1}
states
that the average difference between one step of random splitting and
the true dynamics is $\mathcal{O}(h^2)$ for sufficiently small
$h$. For any finite time interval $[0,t]$ we can leverage this result
to approximate $S_t$ by successive steps of $P_h$. Specifically,
choose $h$ sufficiently small so that~\eqref{3.3} holds and there
exists an integer $m$ with $mh=t$. Then the composition $P^m_h$
corresponds to $\mathcal{O}(1/h)$ steps of $P_h$. Consequently, since
the difference between $P_h$ and $S_h$ is $\mathcal{O}(h^2)$, the
difference between $P^m_h$ and $S_t$ is $\mathcal{O}(h)$.

\aa{A possible interpretation} of $\mathcal{O}(h^p)$ is given in \Cref{thrm3.1} and \Cref{3.1}. This choice matched the particular results being proved.
In  \Cref{thrm3.2} and
\Cref{lem3.2} below, we chose to quantify the error in another fashion,
though the same order of magnitude statements hold true. The same
basic reasoning can be used to prove \aa{the following.}
}
\begin{remark}
  As we have made minimal assumptions on the splitting, we will only
  be able to deduce that $P_h-S_h= \mathcal{O}(h^2)$. In specific examples, it is often
  possible to  arrange the splitting so that $P_h-S_h=
  \mathcal{O}(h^p)$ with $p > 2$. An example of a higher order
  splitting  is Strang splitting \cite{Strang}. Alternatively, higher order can also
  be obtained by fully randomizing the order \cite{Childs_Ostrander_Su_2019} or randomly choosing
  between one ordering and its reverse \cite{Kusuoka_2004,Ninomiya_Victoir_2008,Ninomiya_Ninomiya_2009}.
\end{remark}

\noindent\begin{proof}[Proof of Theorem~\ref{thrm3.1}] Let $h$ be sufficiently small that~\eqref{3.3} holds and such that $mh=t$ for some $m\in\mathbb{N}$. The quantity of interest can be written as the following telescoping sum:
\begin{align}\label{eq:Telescope}
	P^m_h-S_t &= \sum_{k=1}^m P^{k-1}_h(P_h-S_h)S_{h(m-k)}\,.
\end{align}
For any $k$ and continuous function $f$ with $\lVert f\rVert_0=1$,
\begin{align*}
	\lVert P^k_h f\rVert_0 &\leq \mathbb{E}\left(\big\lVert f\big(\Phi^k_{h\tau}\big)\big\rVert_0\right)
		= 1\,.
\end{align*}
Hence $\lVert P^k_h\rVert_{0\to 0}=1$. Similarly, since $mh=t$ implies $h(m-k)\leq t$ for $k\geq 0$ and $\mathcal{X}$ is bounded by Assumption~\ref{assump:Bounded} (so $\Psi$ and its first and second derivatives are bounded on $\mathcal{X}$, uniformly on $[0,t]$),
\begin{align*}
	\lVert S_{h(m-k)}\rVert_{2\to 2} &\leq K(t)
\end{align*}
for some $K(t)$ depending on $t$ but not $h$. Hence, by submultiplicity, \eqref{eq:Telescope}, and \Cref{lem3.1}, we have
\begin{align*}
	\lVert P^m_h-S_t\rVert_{2\to 0} &\leq \sum_{k=1}^m \lVert P^{k-1}_h\rVert_{0\to 0}\lVert P_h-S_h\rVert_{2\to 0}\lVert S_{h(m-k)}\rVert_{2\to 2}
		\leq K(t)\sum_{k=1}^m Ch^2
		= C(t)h\,,
\end{align*}
where $C(t) \coloneqq K(t)C$, with $C$ the constant from~\eqref{3.3} in Lemma~\ref{lem3.1}.
\end{proof}

\begin{remark}
Theorem \ref{thrm3.1} had the relation $h=t/m$, while in the almost-sure results below we will take $h=t/m^2$ (note we explicitly write $t/m^2$, making no reference to the variable $h$). The reason, loosely speaking, is that the transition kernel depends only on the expectation of the randomness, while the almost-sure results additionally depend on fluctuations of the randomness about its mean. For example, Lemma \ref{lem3.2} prepares for an application of the Borel-Cantelli lemma by establishing the summability of probabilities of ``large'' fluctuations over sets of $\mathcal{O}(m)=\mathcal{O}(1/\sqrt h)$ cycles. This is discussed in more detail at the end of this section and worked out in full in the Appendix.
\end{remark}

\begin{theorem}\label{thrm3.2}
Suppose Assumption~\ref{assump:Bounded} holds and fix $t>0$. Then for any $\epsilon > 0$,
\begin{align}\label{3.4}
\mathbb P \left( \limsup_{m \to \infty}	\lVert \widetilde{S}^{m^2}_{t\tau/m^2}-S_t\rVert_{2\to 0} > \epsilon\right) = 0\,.
\end{align}
\end{theorem}

\begin{lemma}\label{lem3.2}
Suppose Assumption~\ref{assump:Bounded} holds and fix $t>0$. Then for any $\epsilon > 0$,
\begin{align}\label{3.5}
	\sum_{m=1}^\infty \mathbb{P}\left(\lVert \widetilde{S}^m_{t\tau/m^2}-S_{t/m}\rVert_{2\to 0} > \tfrac{\epsilon}{m}\right) &< \infty\,.
\end{align}
\end{lemma}

\jcm{  \begin{remark}
   There is a relationship between \Cref{thrm3.1,thrm3.2} and the averaging results from
    Wentzell-Freidlin theory, e.g.,
    \cite[Theorem 2.1, Chapter 7]{WentzellFreidlin}.  This theorem builds on local results like \Cref{lem3.1,lem3.2}. Since our averaging is
    that of a deterministic, cyclic process, the calculations can be
    more explicit and more precise. We are able to prove using simple
    calculations that the local error is  \aa{$\mathcal O(h^2)$} which leads to $\mathcal O(h)$ error over order one times. Typical soft averaging
    results prove a local error of $o(h)$\,\footnote{\jcm{$f(h)$ is
      $o(g(h))$ when $h \to 0$  if $\lim  f(h)/g(h)=0$ as $h \to
      0$. $f(h)$ is
      $O(g(h))$ when $h \to 0$ if $\lim |f(h)/g(h)|  \in (0,\infty)$. }} and then simply conclude
    that the order one error goes to 0. Of course, more  careful
    calculations are possible in the averaging setting. However, the
    simple structure of our problems, where the only randomness is in
    the switching times and not the orderings, allows for the direct,
    straightforward proofs we have presented.
  \end{remark}}

\begin{proof}[Proof of Theorem~\ref{thrm3.2}] By the Borel-Cantelli Lemma it suffices to show
\begin{align*}
	\sum_{m=1}^\infty \mathbb{P}\left(\lVert \widetilde{S}^{m^2}_{t\tau/m^2}-S_t\rVert_{2\to 0} > \epsilon\right) &< \infty\,.
\end{align*}
Consider the telescoping sum
\begin{align}\label{eq:Telescope2}
	\widetilde{S}^{m^2}_{t\tau/m^2}-S_t &= \sum_{k=1}^m \widetilde{S}^{(k-1)}_{t\tau/m^2}\left(\widetilde{S}^m_{t\tau/m^2}-S_{t/m}\right)S_{(m-k)t/m}\,.
\end{align}
For any $k$ and continuous function $f$ with $\lVert f\rVert_0=1$,
\begin{align*}
	\big\lVert \widetilde{S}^k_{t\tau/m^2}f \big\rVert_0 &= \big\lVert f\big(\Phi^k_{h\tau}\big)\big\rVert_0
		= 1\,.
\end{align*}
Hence $\lVert \widetilde{S}^{(k-1)}_{t\tau/m^2}\rVert_{0\to 0}=1$. Similarly, since $(m-k)t/m\leq t$ for $k\geq 0$ and $\mathcal{X}$ is bounded by Assumption~\ref{assump:Bounded} (so $\Psi$ and its first and second derivatives are bounded on $\mathcal{X}$, uniformly on $[0,t]$),
\begin{align*}
	\lVert S_{(m-k)t/m}\rVert_{2\to 2} &\leq K(t)
\end{align*}
for some $K(t)$ depending on $t$ but not $h$. Hence, by submultiplicity, \eqref{eq:Telescope2}, and \Cref{lem3.2}, we have
\begin{align*}
	\big\lVert\widetilde{S}^{m^2}_{t\tau/m^2}-S_t\big\rVert_{2\to 0} &\leq K(t)\sum_{k=1}^m \big\lVert\widetilde{S}^m_{t\tau/m^2}-S_{t/m}\big\rVert_{2\to 0}
		= K(t)m\big\lVert\widetilde{S}^m_{t\tau/m^2}-S_{t/m}\big\rVert_{2\to 0}\,,
\end{align*}
and hence by Lemma~\ref{lem3.2},
\begin{align*}
	\sum_{m=1}^\infty \mathbb{P}\left(\big\lVert \widetilde{S}^{m^2}_{t\tau/m^2}-S_t\big\rVert_{2\to 0} > \epsilon\right) &\leq \sum_{m=1}^\infty \mathbb{P}\left(\big\lVert\widetilde{S}^m_{t\tau/m^2}-S_{t/m}\big\rVert_{2\to 0} > \tfrac{\epsilon}{K(t)m}\right)
		< \infty. \qedhere
\end{align*}
\end{proof}

We conclude this section by sketching the proofs of Lemmas \ref{lem3.1} and \ref{lem3.2}, which are inspired by ideas from \cite{Childs_Ostrander_Su_2019,Childs_Su_Tran_Wiebe_Zhu_2021} and given in full detail in the Appendix. \om{In what follows we set $\widetilde{S}_{h\tau}\coloneqq \widetilde{S}^1_{h\tau}$ and define $\widetilde{S}^{(i,j)}_{h\tau}\coloneqq \widetilde{S}^{(i)}_{h\tau}\cdots\widetilde{S}^{(j)}_{h\tau}$.} Consider first Lemma \ref{lem3.1}. Differentiating $\widetilde{S}_{h\tau}$ in $h$ gives
\begin{align*}
	\partial_h\widetilde{S}_{h\tau} &= \sum_{k=1}^n \tau_k e^{h\tau_1}\cdots e^{h\tau_{k-1}}V_k e^{h\tau_k}\cdots e^{h\tau_n}
		= \sum_{k=1}^n \tau_k \widetilde{S}^{(1,k-1)}_{h\tau}V_k\widetilde{S}^{(k,n)}_{h\tau}.
\end{align*}
Next, commute $\widetilde{S}^{(1,k-1)}_{h\tau}$ and $V_k$ via the Lie bracket $[\widetilde{S}^{(1,k-1)}_{h\tau}, V_k]\coloneqq\widetilde{S}^{(1,k-1)}_{h\tau}V_k-V_k\widetilde{S}^{(1,k-1)}_{h\tau}$ to get
\begin{align*}
		\partial_h\widetilde{S}_{h\tau} &= \sum_{k=1}^n \tau_kV_k\widetilde{S}_{h\tau}+\sum_{k=1}^n \tau_k[\widetilde{S}^{(1,k-1)}_{h\tau}, V_k]\widetilde{S}^{(k,n)}_{h\tau}
		= V\widetilde{S}_{h\tau}+(V_\tau-V)\widetilde{S}_{h\tau}+E_{h\tau}
\end{align*}
where $V_\tau\coloneqq \sum_{k=1}^n \tau_kV_k$ and $E_{h\tau}\coloneqq \sum_{k=1}^n \tau_k[\widetilde{S}^{(1,k-1)}_{h\tau}, V_k]\widetilde{S}^{(k,n)}_{h\tau}$. So, by variation of constants,
\begin{align}\label{4.10}
	\widetilde{S}_{h\tau}-S_h &= \int_0^h S_{h-r}(V_\tau-V)\widetilde{S}_{r\tau} dr+\int_0^hS_{h-r}E_{r\tau} dr.
\end{align}
Loosely speaking, the first integrand is $\mathcal{O}(h)$ because
\begin{equ}\label{e:exp=0}
	\mathbb{E}(V_\tau-V) = \sum_{k=1}^n \mathbb{E}(\tau_k-1)V_k
		= 0
\end{equ}
cancels first order terms from the full expression, $S_{h-r}(V_\tau-V)\widetilde{S}_{r\tau}$. On the other hand the second integrand is $\mathcal{O}(h)$ because the bracket terms in $E_{h\tau}$ also cancel first order terms (most of the work of the proof in the Appendix is making these two statements precise). Thus, integrating these $\mathcal{O}(h)$ terms over the interval $(0,h)$, the difference on the right side of \eqref{4.10} is $\mathcal{O}(h^2)$ as claimed.

\aa{The proof of Lemma \ref{lem3.2} is structurally similar to the one sketched above in that it again begins with an application of variation of constants. However, in this case our analysis aims to establish a concentration estimate and can therefore not rely solely on the vanishing first moment in as in \eref{e:exp=0}. Instead, morally speaking, we expect the desired estimate to hold because of the averaging of iid flow times $\tau_i$ in the homologue of \eref{e:exp=0}. In order to capture such averaging, we cannot limit our analysis to one cycle, but have to consider a variation of constants estimate on $m\gg 1$ such cycles:
\begin{align}
	\widetilde{S}_{h\tau}^m-S_{mh} &= \int_0^h S_{m(h-r)} (V_\tau-V)\widetilde{S}_{r\tau}^m dr+\int_0^hS_{m(h-r)}^m E_{r\tau}^{(m)} dr
\end{align}\
where now $E_{r\tau}^{(m)} \coloneqq \sum_{k=1}^{mn} \tau_k[\widetilde{S}^{(1,k-1)}_{h\tau}, V_k]\widetilde{S}^{(k,n)}_{h\tau}$. Note that the second term contains $\mathcal O(m^2)$ commutators, each contributing $\mathcal O(h^2)$ as in the previous analysis. On the other hand, once integrated, the difference in the first integral, $\sum_{k=1}^{mn} (\tau_k-1)V_k$, scales as $\mathcal O(\sqrt{m}h )$ by the central limit theorem for iid random variables. In order to have both terms decay faster than $\mathcal O(1/m)$ we choose $m\sim \mathcal O(1/\sqrt{h})$, whence the relation $h=t/m^2$.}


\section{Conservative Lorenz-96}\label{sec:Lorenz}


In this section, we apply results of the previous sections to the conservative Lorenz-96 model introduced in Section~\ref{sec:MotivatingExamples}. There we noted the vector field $V$ in~\eqref{5.2} splits as~\eqref{5.3} where the flow of each $V_k$ is a rotation. Specifically, each flow $\varphi^{(k)}$ of the splitting vector fields
\begin{align}\label{eq:LorenzSplittingVF}
	V_k(x)	&= (x_{k+1}e_k-x_ke_{k+1})x_{k-1}
\end{align}
is a rotation in the $(x_k,x_{k+1})$-plane with angular velocity $x_{k-1}$ and therefore preserves Euclidean norm, which we refer to as the \textit{energy} of the system. Throughout this section $\mathcal{V}$ denotes the family of splitting vector fields corresponding to \eqref{eq:LorenzSplittingVF}. By the preceding remarks \om{every $\mathcal{V}$-orbit lies on a sphere} centered at the origin in $\mathbb{R}^n$. In particular, we have

\begin{prop}\label{prop:LorenzConvergence}
\jcm{All the finite time convergence results} of Section~\ref{sec:Convergence} apply to the random splitting \eqref{5.3} of conservative Lorenz-96 starting
from any initial condition.
\end{prop}

\begin{proof}
The splitting vector fields are smooth and Assumption \ref{assump:Bounded} is satisfied since every $\mathcal{V}$-orbit lies on a sphere, so the conclusions of Theorems \ref{thrm3.1} and \ref{thrm3.2} both hold.
\end{proof}


\subsection{Ergodicity}\label{sec:ErgodicityLorenz-96}


A complicating feature of the conservative Lorenz-96 dynamics is that it has fixed points. Specifically, a point $x$ in $\mathbb{R}^n$ is a fixed point of~\eqref{5.2} if and only if $\sum_{k=1}^n (x_k^2+x_{k+1}^2)x_{k-1}^2=0$. For a 2-sphere embedded in $\mathbb{R}^3$ these are precisely the 6 points of intersection of the sphere with the standard coordinate axes. In higher dimensions, these fixed points lie on submanifolds that in general have dimension greater than 0 and in particular are no longer isolated. Nevertheless, nonfixed points cannot reach fixed points in finite time; \om{in fact, the following result shows there is precisely one $\mathcal{V}$-orbit on each sphere that contains all the nonfixed points on that sphere.}

\begin{prop}\label{prop5.1}
If $x$ is a nonfixed point of the conservative Lorenz-96 equations, then
\begin{align}\label{eq:X_L96}
	\mathcal{X}(x) &= \mathcal{X}
		\coloneqq \bigg\{y\in\mathbb{R}^n: \lVert y\rVert = R\ \text{and}\ \sum_{k=1}^n (y_k^2+y_{k+1}^2)y_{k-1}^2\neq 0\bigg\},
\end{align}
where $R=\lVert x\rVert$. Furthermore \jcm{the random splitting of conservative Lorenz-96 is uniquely ergodic on $\mathcal{X}$:} for all $h>0$
the volume form $\lambda$ is the unique
$P_h$-invariant probability measure on $\mathcal{X}$.
\end{prop}

\begin{corollary}\label{cor:lorenz_sphere}
For all $h>0$ the volume form $\lambda$ on $S^{n-1}(R)\coloneqq\{x\in\mathbb{R}^n:\lVert x\rVert=R\}$ is the unique ergodic $P_h$-invariant probability measure on $S^{n-1}(R)$ that is absolutely continuous with respect to $\lambda$.
\end{corollary}

\begin{proof}
$\mathcal{X}$ in \eqref{eq:X_L96} is the complement of a closed, measure zero subset of $S^{n-1}(R)$. Thus $\lambda$ on $\mathcal{X}$ agrees with the volume form, also denoted $\lambda$, on $S^{n-1}(R)$. In particular, $\lambda$ is an ergodic invariant measure on $S^{n-1}(R)$ by \Cref{prop5.1}. Since ergodic invariant measures are mutually singular, see e.g. \cite{HairerConvergence}, any other ergodic invariant measure on $S^{n-1}(R)$ must be singular with respect to $\lambda$.
\end{proof}

\begin{proof}[Proof of \Cref{prop5.1}]
Let $x$ be a nonfixed point with $\lVert x\rVert=R$. We first prove $x$ can be mapped via the split dynamics to $x_*\coloneqq(R/\sqrt{n},\dots, R/\sqrt{n})$. Since $x$ is a nonfixed point, i.e. $\sum_{k=1}^n (x_k^2+x_{k+1}^2)x_{k-1}^2\neq 0$, there exists $k$ such that $x_{k-1}\neq 0$ and $x_k$ or $x_{k+1}$ is nonzero.
Now, since $\varphi^{(k)}$ is a rotation in the $(x_k,x_{k+1})$-plane with angular velocity $x_{k-1}$, there is a $t_k$ such that both $k$ and $k+1$ coordinates of $\varphi^{(k)}(x,t_k)$ are nonzero. By the same argument there is a $t_{k+1}$ such that the $k$, $k+1$, and $k+2$ coordinates of $x^{(k+1)}=\varphi^{(k+1)}(\varphi^{(k)}(x,t_k),t_{k+1})$ are nonzero. Continuing this way, we see $x$ can be made to have nonzero coordinates in a finite number of steps.

Now since $\lVert x\rVert=R$, there exists an index $k$ such that
$\lvert x_k\rvert\geq R/\sqrt{n}$. If $k=n$, rotate in the
$(n-1,n)$-plane so that the $n$th coordinate of $x$ becomes
$R/\sqrt{n}$. If $k<n$, rotate in the $(k,k+1)$-plane so that the
$k+1$ coordinate of $x$ becomes $R/\sqrt{n}$, then rotate in the
$(k+1,k+2)$-plane so that the $k+2$ coordinate of $x$ becomes
$R/\sqrt{n}$, and so on until the $n$th-coordinate of $x$ becomes
$R/\sqrt{n}$. Such rotations are always possible because all
coordinates of $x$ are nonzero by the preceding argument. Thus,
whether $k=n$ or $k<n$ we can evolve $x$ via the split dynamics so
that its last coordinate, $x_n$, is $R/\sqrt{n}$. In particular, there
now must exist an index $k<n$ such that
$\lvert x_k\rvert\geq R/\sqrt{n}$. By the same procedure, and without
disturbing the last coordinate, we can use rotations to make the $n-1$
coordinate of $x$ equal $R/\sqrt{n}$. Iterating this process maps $x$
to $x_*$ in a finite number of steps. Since $x$ was arbitrary it
follows that every nonfixed point with norm $R$ belongs to the same
orbit, which is precisely the set $\mathcal{X}$ defined in
\eqref{eq:X_L96}.

Next we prove there is at most one $P_h$-invariant measure on $\mathcal{X}$. First note that since the split dynamics are all rotations, the above procedure mapping any arbitrary $x$ in $\mathcal{X}$ to $x_*$ can be done using strictly positive times. Furthermore, by direct observation, the matrix of splitting vector fields
\begin{align*}
	\begin{pmatrix}
	 	\vline & \vline & 	& \vline \\
	 	V_1(x) & V_2(x) & \cdots & V_{n-1}(x) \\
	 	\vline & \vline & & \vline
	 \end{pmatrix} &=
	\begin{pmatrix}
	 	\phantom{-}x_2x_n & \phantom{-}0 & \hdots & \phantom{-}0 \\
	 	-x_1x_n & \phantom{-}x_3x_1 & \hdots & \phantom{-}0 \\
	 	\phantom{-}0 & -x_2x_1 & \cdots & \phantom{-}\vdots \\
	 	\phantom{-}\vdots & \phantom{-}\vdots & \ddots & \phantom{-}x_nx_{n-2} \\
	 	\phantom{-}0 & \phantom{-}0 &  & -x_{n-1}x_{n-2}
	 \end{pmatrix}
\end{align*}
has rank $n-1$ whenever all $x_k$ are nonzero. In particular, since $\mathcal{X}$ is an open subset of the sphere of radius $R$ and therefore itself an $n-1$-dimensional manifold, the splitting vector fields $V_k$ span $T_{x_*}\mathcal{X}$. Hence $\Lie_{x_*}(\mathcal{V})=T_{x_*}\mathcal{X}$. By \Cref{cor:ergodic}, $P_h$ has at most one invariant measure on $\mathcal{X}$.

We next show Lebesgue measure, $\Leb$, in $\mathbb R^n$ is
$P_h$-invariant. Let $S^{n-1}(R)$ denote the sphere of radius $R$ in
$\mathbb{R}^n$ and let $\Leb^{(k)}_t\coloneqq
(\varphi^{(k)}_t)_\#\Leb$ be the pushforward of $\Leb$ by
$\varphi^{(k)}_t$. Since the $V_k$ in~\eqref{eq:LorenzSplittingVF} are
divergence free, the continuity equation, \aa{intended in the
  weak sense\footnote{\jcm{This equation should be interpreted as an
    equation on measures or, equivalently, as holding in the weak
    sense. In other words, the left and right side are equal when
    integrated against any
    compactly supported, smooth test function.}}}, becomes
\begin{equs}
	0 &= \partial_t\Leb^{(k)}_t + \diverg\left(V_k\Leb^{(k)}_t\right)
		= \partial_t\Leb^{(k)}_t + \nabla\Leb^{(k)}_t\cdot V_k\,.
\end{equs}
The latter is a transport equation with constant initial condition
$\Leb^{(k)}_0\equiv 1$ and hence $\Leb^{(k)}_t=\Leb$ for all $t$.
Because the trajectories of all $V_k$ conserve the energy $\|x\|$, we
fiber $\mathbb R^n$ using spherical coordinates
$(r, \theta) \in \mathbb R_+ \times S^{n-1}(R)$. In these coordinates,
we have that
$V_k(r, \theta) = 0\, \partial_r + r v_k(\theta) \nabla_\theta$ and by
a change of coordinates of the divergence operator the stationarity
equation becomes
\begin{equ}
  0 = \diverg\left(V_k(x)\lambda(x)\right) = u(r) w(\theta)
  \diverg_\theta (\lambda(r, \theta) v_k(\theta))\,,
\end{equ}
where $\diverg_\theta$ denotes the angular terms of the divergence in spherical coordinates, and $u(r), w(\theta)$ result from the change of variables.
Hence, we can factor the solution $\lambda(r, \theta) = \bar
\lambda(\theta|r) \cdot \mu_R(d r) = \bar \lambda(\theta) \cdot
\mu_R(d r)$, where $ \bar \lambda(\theta|r)$ is the conditional
density of Lebesgue measure on a fiber. The measure $\bar \lambda$
solves $w(\theta) \diverg_\theta (\bar \lambda(\theta) v_k(\theta) ) =
0$ and is therefore invariant under the flows $\phi_t^{(k)}$. By
rotational symmetry of $\Leb$, we must have that $\bar
\lambda(\theta)$ is the volume form on $S^{n-1}(R)$. And since
$\mathcal{X}$ is a full-measure open subset of $S^{n-1}(R)$, the
volume form $\lambda$ on $\mathcal{X}$ is just the restriction of
$\bar\lambda$ to $\mathcal{X}$. Thus $\lambda$ is also invariant under
the flows and is therefore the unique $P_h$-invariant measure on
$\mathcal{X}$.
\end{proof}


\section{Galerkin approximations of {2D} Euler}\label{sec:Euler}


The 2D Euler equations on the torus $\mathbb{T}$ are obtained from the 2D Navier-Stokes equations~\eqref{eq:2DNS} by dropping the dissipative and forcing terms:
\begin{align}\label{6.1}
	\begin{cases}
		\partial_tu+(u\cdot\nabla)u = -\nabla p \\
		\diverg(u) \coloneqq \nabla\cdot u = 0
	\end{cases}
\end{align}
where, as before, $u:\mathbb{T}\times\mathbb{R}\to\mathbb{R}^2$ is the fluid velocity, $p:\mathbb{T}\times\mathbb{R}\to\mathbb{R}$ the fluid pressure, and
\begin{align*}
	(u\cdot\nabla)u &= (u_1\partial_1u_1+u_2\partial_2u_1, u_1\partial_1u_2+u_2\partial_2u_2)\,.
\end{align*}
In this section we construct a convenient random splitting of~\eqref{6.1}. To do so we first write~\eqref{6.1} in vorticity form and apply the Fourier transform. This yields an infinite system of ODEs which we truncate to systems of arbitrary finite size, referred to throughout as Galerkin approximations. Finally, we split these Galerkin approximations and apply the results of Sections~\ref{sec:ergodicity} and~\ref{sec:Convergence} to the associated random splitting.


\subsection{Constructing the splitting}\label{sec:ConstructingSplitting}


The vorticity formulation of~\eqref{6.1} is obtained by taking the curl of velocity. Specifically, setting $q\coloneqq\curl(u)\coloneqq \partial_2u_1-\partial_1u_2$, equation~\eqref{6.1} becomes
\begin{align}\label{6.2}
	\begin{cases}
		\partial_tq+(\mathcal{K}q\cdot\nabla)q = 0 \,,\\
		\diverg(q) = 0\,,
	\end{cases}
\end{align}
where $\mathcal{K}\coloneqq \nabla^\perp(-\Delta)^{-1}$ with $\nabla^\perp\coloneqq (\partial_2,-\partial_1)$.
To express~\eqref{6.2} in Fourier space, set $\mathbb{Z}^2_\infty\coloneqq \mathbb{Z}^2\setminus\{(0,0)\}$ and let $\{e_j\}_{j\in\mathbb{Z}^2_\infty}$ be the orthonormal basis of $L^2(\mathbb{T},\mathbb{R})$ given by $e_j(x)\coloneqq (2\pi)^{-1}\exp(ix\cdot j)$. Then $q(x,t)=\sum_{j\in\mathbb{Z}^2_\infty}q_j(t)e_j(x)$ where
\begin{align*}
	q_j(t) &\coloneqq \langle q,e_j\rangle_{L^2}
		=\int_{\mathbb{T}} q(x,t)\overline{e}_j(x) dx
\end{align*}
is the $j$th Fourier mode of $q$. Here $\langle \cdot,\cdot\rangle_{L^2}$ is the standard inner product on $L^2(\mathbb{T},\mathbb{R})$ with $\overline{e}_j$ denoting the complex conjugate of $e_j$. The $j$th Fourier mode of $(\mathcal{K}q\cdot\nabla)q$ is
\begin{align*}
	\langle (\mathcal{K}q\cdot\nabla)q,e_j\rangle_{L^2} &= \sum_{k+\ell=j} C_{k\ell}q_kq_\ell
\end{align*}
where
\begin{equ}\label{e:c}
	C_{k\ell} \coloneqq \frac{\langle k,\ell^\perp\rangle}{4\pi}\bigg(\frac{1}{\lvert k\rvert^2}-\frac{1}{\lvert\ell\rvert^2}\bigg)
\end{equ}
with $\langle\cdot,\cdot\rangle$ the standard inner product in $\mathbb R^2$, $\ell^\perp\coloneqq (\ell_2,-\ell_1)$, and $\lvert\ell\rvert^2\coloneqq\ell_1^2+\ell_2^2$. Therefore
\begin{align*}
	\sum_j \dot{q}_je_j &= \partial_t q
		= -(\mathcal{K}q\cdot\nabla)q
		= -\sum_j\bigg(\sum_{k+\ell=j} C_{k\ell}q_kq_\ell\bigg)e_j
\end{align*}
and hence $\dot{q}_j=-\sum_{k+\ell=j} C_{k\ell}q_kq_\ell$. Moreover, since $q$ is real-valued,
\begin{align*}
	\sum_j q_je_j &= q
		= \overline{q}
		= \sum_j \overline{q}_je_{-j}
\end{align*}
which gives $q_j=\overline{q}_{-j}$. In particular,
\begin{align*}
	\dot{q}_j &= \dot{\overline{q}}_{-j}
		= -\sum_{j+k+\ell=0} C_{k\ell}\overline{q}_k\overline{q}_\ell\,.
\end{align*}
Writing each Fourier mode $q_j=a_j+ib_j$ in terms of real and imaginary parts then gives
\begin{align*}
	\dot{a}_j+i\dot{b}_j = \dot{q}_j
		&= -\sum_{j+k+\ell=0} C_{k\ell}(a_k-ib_k)(a_\ell-ib_\ell) \\
		&= \sum_{j+k+\ell=0} C_{k\ell}(b_kb_\ell-a_ka_\ell)+i\sum_{j+k+\ell=0} C_{k\ell}(a_kb_\ell+a_\ell b_k)\,.
\end{align*}
Thus the Fourier modes of solutions to the Euler equation in vorticity form satisfy
\begin{equation}
 \label{6.3}\left\{
  \begin{aligned}
		\dot{a}_j &= \sum_{j+k+\ell=0}^{~} C_{k\ell}(b_kb_\ell-a_ka_\ell) \\
		\dot{b}_j &= \sum_{j+k+\ell=0} C_{k\ell}(a_kb_\ell+a_\ell b_k)
\end{aligned}\right.
\end{equation}
for all $j\in\mathbb{Z}^2_\infty$. While~\eqref{6.3} could be studied as is, notice the constraint $q_{-j}=\overline{q}_j$ implies $a_{-j}=a_j$ and $b_{-j}=-b_j$, which introduces redundancy in~\eqref{6.3}. Therefore we restrict to the subset
\begin{align*}
	\mathbb{Z}^2_+ &\coloneqq \{j\in\mathbb{Z}^2 : j_2>0\}\cup\{j\in\mathbb{Z}^2 : j_2=0\ \text{and}\ j_1>0\}\,.
\end{align*}
Specifically, by straightforward computation together with the identities $a_{-j}=a_j$, $b_{-j}=-b_j$, and $C_{k\ell}=C_{-k,-\ell}=-C_{-k,\ell}=-C_{k,-\ell}$, the system~\eqref{6.3} can be re-expressed as
\begin{equation}\label{6.4}
\left\{	\begin{aligned}
		\dot{a}_j =& \sum_{j+k-\ell=0} C_{k\ell}(a_ka_\ell+b_kb_\ell)+\sum_{j-k-\ell=0} C_{k\ell}(b_kb_\ell-a_ka_\ell) \\
		\dot{b}_j =& \sum_{j+k-\ell=0} C_{k\ell}(a_k b_\ell-b_k a_\ell)-\sum_{j-k-\ell=0} C_{k\ell}(a_kb_\ell+b_k a_\ell)
	\end{aligned}\right.
\end{equation}
for all $j\in\mathbb{Z}^2_+$ with each sum running over all pairs $k,\ell\in\mathbb{Z}^2_+$ satisfying the specified identity. To split~\eqref{6.4} note that for any $j,k,\ell\in\mathbb{Z}^2_+$ satisfying $j+k-\ell=0$ (and hence $\ell-j-k=0$) we can isolate from the above sums exactly $6$ equations involving only these indices:
\begin{equation}\label{6.5}
\begin{aligned}
	\dot{a}_j &= C_{k\ell}(a_ka_\ell+b_kb_\ell)\,, \qquad \dot{a}_k = C_{j\ell}(a_ja_\ell+b_jb_\ell)\,, \qquad \dot{a}_\ell = C_{jk}(b_jb_k-a_ja_k)\,, \\
	\dot{b}_j &= C_{k\ell}(a_kb_\ell-b_k a_\ell)\,, \qquad \dot{b}_k = C_{j\ell}(a_jb_\ell-b_j a_\ell)\,, \qquad\, \dot{b}_\ell = -C_{jk}(a_jb_k+b_j a_k)\,.
\end{aligned}
\end{equation}
For reasons to be made clear shortly, we recombine~\eqref{6.5} into $4$ groups of $3$ equations:
\begin{align}\label{e:vf}
	\begin{cases}
		\dot{a}_j = C_{k\ell}a_ka_\ell \\
		\dot{a}_k = C_{j\ell}a_ja_\ell \\
		\dot{a}_\ell = -C_{jk}a_ja_k
	\end{cases}
	\begin{cases}
		\dot{a}_j = C_{k\ell}b_kb_\ell \\
		\dot{b}_k = C_{j\ell}a_jb_\ell \\
		\dot{b}_\ell = -C_{jk}a_jb_k
	\end{cases}
	\begin{cases}
		\dot{b}_j = C_{k\ell}a_kb_\ell \\
		\dot{a}_k = C_{j\ell}b_jb_\ell \\
		\dot{b}_\ell = -C_{jk}b_ja_k
	\end{cases}
	\begin{cases}
		\dot{b}_j = -C_{k\ell}b_ka_\ell \\
		\dot{b}_k = -C_{j\ell}b_ja_\ell \\
		\dot{a}_\ell = C_{jk}b_jb_k
	\end{cases}\,.
\end{align}
Let $V_{a_ja_ka_\ell}$, $V_{a_jb_kb_\ell}$, $V_{b_ja_kb_\ell}$, and $V_{b_jb_ka_\ell}$ be the vector fields associated to the equations of~\eqref{e:vf} from left to right. For example, $V_{a_ja_ka_\ell}$ is the vector field on $\mathbb{R}^\infty$ mapping the $a_j$ coordinate to $-C_{k\ell}a_ka_\ell$, the $a_k$ coordinate to $-C_{j\ell}a_ja_\ell$, the $a_\ell$ coordinate to $-C_{jk}a_ja_k$, and all other coordinates to $0$. These are the \textit{splitting vector fields}. Our sought-after splitting is
\begin{align}\label{6.7}
	V &= \sum_{j+k-\ell=0}V_{a_ja_ka_\ell}+V_{a_jb_kb_\ell}+V_{b_ja_kb_\ell}+V_{b_jb_ka_\ell}\,,
\end{align}
where $V$ is the vector field associated to~\eqref{6.4}. As noted earlier, our focus will be on finite truncations of the infinite-dimensional system~\eqref{6.4}. Thus we fix an integer $N\geq 2$ and define the \textit{$\Nth$ Galerkin approximation} of~\eqref{6.4} to be~\eqref{6.4} with indices restricted to the set
\begin{align*}
	\mathbb{Z}^2_N &\coloneqq \big\{j\in\mathbb{Z}^2_+ : \max\{\lvert j_1\rvert, \lvert j_2\rvert\}\leq N\big\}\,.
\end{align*}
The splitting~\eqref{6.7} remains valid in this finite-dimensional setting, bearing in mind that now all indices lie in $\mathbb{Z}^2_N$. By a slight abuse of notation, we denote the finite-dimensional counterpart of $V$ by $V$ and similarly for the splitting vector fields. Thus our family of splitting vector fields is
\begin{align}\label{eq:euler_splitting}
	\mathcal{V} &=\left\{V_{a_ja_ka_\ell}, V_{a_jb_kb_\ell}, V_{b_ja_kb_\ell}, V_{b_jb_ka_\ell} : j,k,\ell\in\mathbb{Z}^2_N \text{ and } j+k-\ell=0\right\}.
\end{align}
Since $\mathbb{Z}^2_N$ has cardinality $2N(N+1)$ and each index $j\in\mathbb{Z}^2_N$ has an associated $a_j$ and $b_j$ coordinate, these are all vector fields on $\mathbb{R}^n$, where throughout this section we set $n\coloneqq 4N(N+1)$. We also abuse notation by conflating elements $j$ in $\mathbb{Z}^2_N$ with elements $j$ in $\{1,\dots,n/2\}$, which can be formalized via any bijection between the two sets. Moreover, we denote elements of $\mathbb{R}^n$ by $q=(a_j,b_j)_{j=1}^{n/2}$. This reflects that the $a_j$ and $b_j$ coordinates of $q$ in $\mathbb{R}^n$ are in one-to-one correspondence with the real and imaginary parts of the $j$th mode of $q$.

\begin{remark}\label{rem:diffSplit}
  There are many possible splittings of a given
  equation. For the Euler equations, we made the particular choice we
  have so that both energy and enstrophy are conserved but the
  dynamics of each splitting are still relatively easily understood. We could have
  further decomposed the three-dimensional dynamics in the above
  splitting into a number of two-dimensional dynamics, similar in spirit to the decomposition into rotations used in Lorenz-96. However, that
  would have necessitated only conserving either the energy or the
  enstrophy.
\end{remark}


\subsection{Conservation and convergence}


The conservative Lorenz-96 dynamics discussed in Section~\ref{sec:Lorenz} conserves Euclidean norm (energy in that case) and therefore remains on whichever sphere it starts on. So too do the flows of each of the splitting vector fields \eqref{eq:LorenzSplittingVF}. We now show a similar thing is true for Galerkin approximations of 2D Euler. Define the \textit{energy} and \textit{enstrophy} of $q=(a_j,b_j)_{j=1}^{n/2}$ by
\begin{equ}\label{e:uandv}
	E(q) \coloneqq \sum_{j\in\mathbb{Z}^2_N}\frac{a_j^2+b_j^2}{\lvert j\rvert^2}
	\qquad\text{and}\qquad
	\mathcal{E}(q) \coloneqq \sum_{j\in\mathbb{Z}^2_N} a_j^2+b_j^2\,,
\end{equ}
respectively (note the aforementioned conflation of $j$ in $\mathbb{Z}^2_N$ and $j\in\{1,\dots,n/2\}$ in the summations). Straightforward computation shows that for all $j,k,\ell\in\mathbb{Z}^2_N$ with $j+k-\ell=0$,
\begin{align*}
	C_{k\ell}+C_{j\ell}-C_{jk} &= \frac{C_{k\ell}}{\lvert j \rvert^2}+\frac{C_{j\ell}}{\lvert k \rvert^2}-\frac{C_{jk}}{\lvert \ell \rvert^2}
		= 0\,,
\end{align*}
which in turn implies that under the dynamics~\eqref{6.4},
\begin{align*}
	\partial_tE(q) &= \partial_t\mathcal{E}(q)
		= 0
\end{align*}
for all $q\in\mathbb{R}^n$. That is, both energy and enstrophy are conserved by the true dynamics and the set
\begin{align}\label{eq:EngEst}
	\mathcal{Q}_0(E,\mathcal{E}) \coloneqq \big\{q\in\mathbb{R}^n : E(q)=E,\
     	\mathcal{E}(q)=\mathcal{E} \big\}\,.
\end{align}
is invariant under \eref{6.4}. This is a well-established property of the 2D Euler equations. Moreover, if we flow by $V_{a_ja_ka_\ell}$ starting from $q$ for any $j,k,\ell\in\mathbb{Z}^2_N$ with $j+k-\ell = 0$, then
\begin{align*}
	\tfrac{1}{2}\partial_tE(q) &= \frac{a_j\dot{a}_j}{\lvert j\rvert^2}+\frac{a_k\dot{a}_k}{\lvert k\rvert^2}+\frac{a_\ell\dot{a}_\ell}{\lvert \ell\rvert^2}
		= \bigg(\frac{C_{k\ell}}{\lvert j \rvert^2}+\frac{C_{j\ell}}{\lvert k \rvert^2}-\frac{C_{jk}}{\lvert \ell \rvert^2}\bigg)a_ja_ka_\ell
		= 0\,,
\end{align*}
and similarly $\partial_t\mathcal{E}(q)=0$. The same computation shows energy and enstrophy are conserved by \textit{all} of the splitting vector fields in $\mathcal{V}$, which provides the motivation for recombining~\eqref{6.5} as~\eqref{e:vf} in the first place. In particular, we have

\begin{prop}\label{prop:EulerConvergence}
\jcm{All of the finite time convergence results} of Section~\ref{sec:Convergence} apply to the random splitting \eqref{6.7} of every Galerkin approximation of 2D Euler starting
from any initial condition.
\end{prop}

\begin{proof}
The splitting vector fields are smooth and Assumption \ref{assump:Bounded} is satisfied since every $\mathcal{V}$-orbit lies on a sphere, so the conclusions of Theorems \ref{thrm3.1} and \ref{thrm3.2} both hold.
\end{proof}


\subsection{Ergodicity}\label{s:eulergo}


Fix energy and enstrophy values $E$ and $\mathcal{E}$ and set $\mathcal{Q}_0\coloneqq\mathcal{Q}_0(E,\mathcal{E})$. \om{$\mathcal{Q}_0$ is an $n-2$-dimensional submanifold of $\mathbb{R}^n$ where, recall, $n\coloneqq 4N(N+1)$; denote its volume form by $\lambda$}. As with conservative Lorenz-96, the $\Nth$ Galerkin approximation of 2D
Euler has points $q$ in $\mathcal Q_0$ whose $\mathcal{V}$-orbits are not dense in $\mathcal Q_0$. For example, any $q$ with exactly one nonzero coordinate is a fixed point of~\eqref{6.4} and of all the
equations~\eqref{e:vf}. In this subsection we characterize these
points and prove there is exactly one $\mathcal{V}$-orbit
  $\mathcal{Q}$ on $\mathcal{Q}_0$ such that
  $\lambda(\mathcal{Q})=1$. By a slight abuse of notation we denote the restriction of $\lambda$ to $\mathcal{Q}$ by $\lambda$ as well.
  We then show \aa{there exists a}
  unique $P_h$-invariant measure on $\mathcal{Q}$ \aa{-- and hence on $\mathcal{Q}_0$ --  that is absolutely continuous with respect to $\lambda$ on $\mathcal{Q}_0$.}

To make the above statements precise, we begin by enumerating the coordinates of $q \in \mathbb R^n$ by extending the indices $j\in \ZN$ with an element $\chi \in \{+, -\}$ which denotes the real ($+$) or imaginary ($-$) part of the corresponding mode. Then, for ${\bf j} = (j, \chi) \in \ZNp$, we define the \emph{type}
of such coordinates via the function $\mathrm{T}({\bf j}) = \chi$ so that $q_{{\bf j}}$ is identified with $a_j$ if $\mathrm{T}({\bf j}) = +$ and with $b_j$ if $\mathrm{T}({\bf j}) = -$.
For $q\in \mathbb R^n$, denote by
\begin{equ}
	\AAc(q)\eqdef\big\{{\bf j} \in \ZNp~:~q_{{\bf j}}\neq 0\big\}
	\end{equ}
	the set of ``active'' coordinates.
	To streamline our analysis, we define the following operation to expand the set $\mathcal A$:
	\begin{equ}\label{e:+}
		\AA \oplus {\bm \ell}  \coloneqq \begin{cases} \AA \cup \{{\bm \ell}\}\qquad&\text{if }  \ell \in \{j+k,j-k\}\cap \ZN \text{ for }{\bf j}, {\bf k} \in \AA, C_{jk}\neq 0, \mathrm{T}({\bf j})\cdot  \mathrm{T}({\bf k}) =  \mathrm{T}({\bm \ell})\,,\\ \AA & \text{else}\,, \end{cases}
  \end{equ}
where $\mathrm{T}({\bf j})\cdot  \mathrm{T}({\bf k})$ is $+$ if $\mathrm{T}({\bf j})=  \mathrm{T}({\bf k})$  and $-$ if $\mathrm{T}({\bf j})\neq  \mathrm{T}({\bf k})$. This operation corresponds to extending the nonzero coordinates of $q$ from $\jj,\kk$ to $\ll$ by letting a triple $\iota = \jkl$ interact.

We assume that the initial condition is sufficiently nondegenerate, as stated in the following assumption similar to the one made in \cite[Thm.~2.1]{hairermattingly06}.
\om{
\begin{definition}[Nondegenerate point]\label{def:nondegenerate}
A point $q$ in $\mathcal{Q}_0$ is {\normalfont nondegenerate} if there exists $M \in \mathbb N$, $j^*\in \ZN$ with $|j^*|^2>1$, and an ordered set of indices $(\ll_i)_{i=1}^M$ in $\ZNp$ such that
	\begin{equ}\label{eq:nondegenerate}
		\big\{(1,0,+),(0,1,+),  (j^*,-)\big\}\subseteq \big((\AAc(q)\oplus \bm\ell_1)\oplus \bm\ell_2\big) \dots \oplus \bm\ell_M.
	\end{equ}
\end{definition}

\begin{definition}[Generic point]\label{def:generic}
A point in $\mathbb{R}^n$ is {\normalfont generic} if all of its coordinates are nonzero.
\end{definition}

\begin{remark}\label{rmk:nondegenerate}
Every point with all coordinates nonzero is a nonfixed point of conservative Lorenz-96;
similarly, every generic point in $\mathcal{Q}_0$ is
nondegenerate. However, comparing~\eqref{eq:nondegenerate}
with~\eqref{eq:X_L96}, we see the conditions defining nondegenerate
points in $\mathcal{Q}_0$ are more complicated than the easily characterized nonfixed
points of conservative Lorenz-96. The difference is that, unlike
spheres in conservative Lorenz-96, there are proper subspaces of
$\mathcal{Q}_0$ which are invariant for our splitting of the Euler
dynamics but are not fixed points. One such subspace is the collection
of purely real points; another is the purely imaginary
points.
\end{remark}

\noindent The following analogs of \Cref{prop5.1} and \Cref{cor:lorenz_sphere} are the main results of this subsection.

\begin{prop}\label{prop6.1}
Every nondegenerate point in $\mathcal{Q}_0$ belongs to the same
$\mathcal{V}$-orbit, $\mathcal{Q}$, and for all $h>0$ \aa{there exists a} unique $P_h$-invariant probability measure on
$\mathcal{Q}$. Furthermore, this unique invariant measure is absolutely continuous with respect to the volume form on $\mathcal{Q}$.
\end{prop}

\begin{proof}
By \pref{p:conteuler} there is a $q^*$ in $\mathcal{Q}_0$ such that
every nondegenerate point in $\mathcal{Q}_0$ belongs to the
$\mathcal{V}$-orbit $\mathcal{Q}\coloneqq \mathcal{Q}(q^*)$, and for
every $q$ in $\mathcal{Q}$ there is an $m  \in \mathbb N$ and a $t \in \mathbb{R}^{mn}_+$ satisfying $\Phi^m(q,t)=q^*$. By \Cref{lem6.1} the
splitting vector fields span the tangent space of $\mathcal{Q}$ at
generic points; in particular, the Lie bracket condition holds at
every generic point. Thus, since the vector fields in $\mathcal{V}$
are analytic, \Cref{cor:nagano} implies $P_h$ has at most one
invariant probability measure on $\mathcal{Q}$, which is necessarily
the \aa{one identified} by \Cref{lem:EulerInvariant}.
\end{proof}
}

\begin{corollary}\label{cor:euler_shell}
For all $h>0$ the measure from \Cref{prop6.1} is the unique $P_h$-invariant ergodic probability measure on $\mathcal{Q}_0$ that is absolutely continuous with respect to the volume form on $\mathcal{Q}_0$.
\end{corollary}

\begin{proof}
Let $\lambda$ denote volume form on $\mathcal{Q}_0$. Since $\mathcal{Q}$ contains all generic points in $\mathcal{Q}_0$, it is an open subset of $\mathcal{Q}_0$ satisfying $\lambda(\mathcal{Q})=1$. In particular, the unique invariant measure on $\mathcal{Q}$ from \Cref{prop6.1} is an ergodic invariant measure on $\mathcal{Q}_0$. Since ergodic invariant measures are mutually singular, see e.g. \cite{HairerConvergence}, any other ergodic invariant measure on $\mathcal{Q}_0$ must be singular with respect to $\lambda$.
\end{proof}

\begin{remark}\label{rm:randomSplit}
  Continuing in the spirit of Remark~\ref{rem:diffSplit}, we observe the
  splitting in~\eqref{e:vf} splits $q_j$ into its real and
\jcm  imaginary parts. We could have chosen another basis of $\mathbb{C}$
  and even randomized over this choice for each evolution of an
  interacting triple $(j,k,\ell)$. More explicitly, if we define
  $e(\theta)=cos(\theta)+i \sin(\theta)$ then $e(\theta)$ and
  $e(\theta+\frac\pi2)$ form an orthonormal basis of $\mathbb{C}$ for
  any $\theta$. Then we can drive a system analogous to~\eqref{e:vf}
  by setting $q_\ell = a_\ell^\theta e(\theta) + b_\ell^\theta
  e(\theta+\frac\pi2)$. As the form is similar to~\eqref{e:vf},
  the results of the paper extend to this system. In particular, by randomizing the choice of $\theta$ for each such triple $(j,k,\ell)$,
  we can relax the characterization of nondegenerate points in \Cref{def:nondegenerate} by destroying some of the
  invariant structures discussed in Remark~\ref{rmk:nondegenerate} which obstruct controllability starting from
  some initial conditions.
\end{remark}


\subsubsection{Controllability}\label{sec:EulerControl}


In this section, we prove controllability of the dynamics \eref{e:vf}. By \om{conservation of energy and enstrophy}, the \om{$\mathcal{V}$-}orbit of an initial condition $q^{(0)}$ \om{in $\mathcal{Q}_0$ is contained in $\mathcal{Q}_0$}. Recalling the definition of extended indices in Section~\ref{s:eulergo}, we define the set of \emph{interacting coordinate triples}
\begin{equ}
	\II \eqdef \big\{({\bf j},{\bf k},{\bm \ell}) \in (\ZNp)^3~:~
  j+k=\ell,~(C_{jk}, C_{j \ell}, C_{k\ell})\neq (0,0,0),~\mathrm{T}({\bf j})\cdot\mathrm{T}({\bf k})=\mathrm{T}({\bm \ell})\big\}\,.
\end{equ}
Then, for any such triple of interacting indices $\iota \in \II$
we denote by $\pit{t}~:~\QQ_0 \to \QQ_0$ the flow of the ODEs \eref{e:vf} evolving the corresponding coordinates. The dynamics we consider is then obtained by cycling through the set $\II$ in a fixed or random order. For any $\iota \in \II$ we denote by $\Pit{t}~:~\QQ_0\to\QQ_0$ the flow of \eref{e:vf} after one such full cycle where the flow times are chosen as
\begin{equ}\label{e:tauxis}
\tau^\xi =\begin{cases}t\qquad &\text{if } \xi = \iota\,,\\0 & \text{else}\,, \end{cases}
\end{equ}
so that for any $q \in \QQ_0$, $\Pit{t}(q) = \pit{t}(q)$.

\om{Let $q^* = (a_j^*,b_j^*)_{j=1}^{n/2}$ be the point in $\mathcal{Q}_0$ defined as follows:}
 \begin{equ}\label{e:qstar}
   q_{(1,0)}^* = q_{(0,1)}^* = (a^*,0) \,,\qquad q_{(N,N)}^*  = (0,b^*)\,,
 \end{equ}
 for $a^*, b^*\geq 0$ and $q_{j}^*  = (0,0)$ for all other $j\in  \ZN$. We show below that for any \om{nondegenerate initial condition $q^{(0)} \in \QQ_0$} the system can be driven to this unique point $q^*$\,.

\begin{proposition}\label{p:conteuler}
For any \om{nondegenerate point $q^{(0)} = (a_j^{(0)},b_j^{(0)})_{j=1}^{n/2}$ in $\mathcal{Q}_0$} there exists $M \in \mathbb N$ and a joint sequence of transition times and coordinate triples $\{(\iota(m),\tau(m))\}_{m = 1}^M$  such that
\begin{equ}\label{e:conteuler}
	\Phi_{\tau(M)}^{\iota(M)}\circ\dots\circ \Phi_{\tau(1)}^{\iota(1)} (q^{(0)}) = q^*\,.
\end{equ}
\om{Thus every nondegenerate point belongs to the same orbit, $\mathcal{Q}\coloneqq\mathcal{Q}(q^*)$. Furthermore, for every $q$ in $\mathcal{Q}$ there is an $m \in \mathbb N$ and a $t \in \mathbb{R}^{mn}_+$ such that $\Phi^m(q,t)=q^*$.}
\end{proposition}
\aa{ Recall from \rref{r:exponential} that the only property of the exponential distribution used in this proof is the fact that it has a density around $0$, allowing to choose the flow of some of the split vector fields to be the identity as, e.g., in \eref{e:tauxis}. This comment also applies to the proof of Proposition~\ref{prop5.1} in the previous section. We further note that, since the trajectories of each of the $\phi^{\iota(m)}$ in the above theorem are periodic (see \lref{l:midmode} and \lref{l:samenorm}), each of these transformations can be inverted by choosing complementary transition times to $\tau(m)$. Inverting the order of the transformations yields the converse statement:

\begin{corollary}\label{c:reverse}
For any nondegenerate point $q^{(0)} = (a_j^{(0)},b_j^{(0)})_{j=1}^{n/2}$ in $\mathcal{Q}_0$ there exists $M \in \mathbb N$ and a joint sequence of transition times and coordinate triples $\{(\tilde \iota(m),\tilde \tau_0(m))\}_{m = 1}^M$  such that
	\begin{equ}
		\Phi_{\tilde \tau(M)}^{\tilde \iota(M)}\circ\dots\circ \Phi_{\tilde \tau(1)}^{\tilde \iota(1)} (q^*) = q^{(0)}\,.
	\end{equ}
\end{corollary}
\noindent While the \Cref{c:reverse} will not be used in the remainder of the paper, it offers an alternative to \tref{thrm:nagano} in proving that, when applying \Cref{cor:ergodic}, it is sufficient to verify that Lie bracket condition holds at \emph{any} point in $\mathcal Q$, not necessarily at $q^*$.
}

\begin{proof}[Proof of \Cref{p:conteuler}]
We prove the \om{first statement} by first evolving the initial condition $q^{(0)}$ into a sufficiently nondegenerate state $q^{(1)}$, and then by sequentially shrinking the set of active components of the coordinate vector $q$ to the ones listed in \eref{e:qstar}. We realize this program by following, in order, the sequence of steps described below, represented schematically in \fref{f:steps}:\\[-5pt]

\begin{enumerate}
	\setcounter{enumi}{-1}
\item If it is not the case at initialization, \lref{l:contr1} shows that we can ``prepare'' our state by evolving $q^{(0)}$ into $q^{(1)}$ such that
\begin{equ}\label{e:q0}
	a_{(1,0)}^{(1)}, b_{(1,0)}^{(1)},a_{(0,1)}^{(1)}, b_{(0,1)}^{(1)},a_{(1,1)}^{(1)},b_{(1,1)}^{(1)}\neq 0\,,
	\end{equ}
as	represented in \fref{f:steps1}.

\item As shown in \lref{l:contr2}, we can then transform $q^{(1)}$ into $q^{(2)}$
with the property
\begin{equ}\label{e:q20}
q_{j}^{(2)}=(0,0)\qquad \text{for all } j\in \ZN \setminus \{(0,1),(1,0),(1,1),(N,N),(-N,N)\}\,,
\end{equ}
as represented in \fref{f:steps2}, and
\begin{equ}\label{e:q21}
	a_{(1,0)}^{(2)}, b_{(1,0)}^{(2)},	a_{(0,1)}^{(2)}, b_{(0,1)}^{(2)},a_{(1,1)}^{(2)},b_{(1,1)}^{(2)} \neq 0\,.
	\end{equ}

		\item \lref{l:contr3} shows that we can then ``transfer'' the amplitude from modes $a_{(-N,N)}$, $b_{(-N,N)}$, $a_{(N,N)}$ to mode $b_{(N,N)}$ \ie we can reach a state $q^{(3)}$ that satisfies
		\begin{equs}
		q_{j}^{(3)}=(0,0)\qquad \qquad \, &\text{for all } j\in \ZN\setminus \{(0,1),(1,0),(1,1),(N,N)\}\,,\label{e:q30}\\
		q_{(N,N)}^{(3)}=(0,b_{(N,N)}^{(3)})\quad&\text{with }b_{(N,N)}^{(3)} \geq 0\,.\label{e:q31}
		\end{equs}
		This state is represented in \fref{f:steps3}.

  \item  Finally, \lref{l:contr4} shows that we can ``transfer'' the amplitude from modes $a_{(1,1)}$, $b_{(1,1)}$, $b_{(0,1)}$ and $b_{(1,0)}$ to modes $a_{(0,1)}, a_{(1,0)}, b_{(N,N)}$ so that, after the transfer, $a_{(0,1)} = a_{(1,0)}$ and $a_{(0,1)},a_{(1,0)}, b_{(N,N)}>0$ \ie we reach the unique state $q^*$ from \eref{e:qstar} (represented in \fref{f:steps4}).
\end{enumerate}

\noindent\om{This proves the first part of \Cref{p:conteuler}, which immediately implies nondegenerate points in $\mathcal{Q}_0$ belong to $\mathcal{Q}=\mathcal{Q}(q^*)$. Let $q$ be \textit{any} point in $\mathcal{Q}$. By definition there exist $m$ and $t$ in $\mathbb{R}^{mn}$ such that
\begin{align*}
	\Phi^m(q,t) &= \varphi^{(n)}_{t_{mn}}\circ\cdots \varphi^{(1)}_{t_1}(q)
		= q^*.
\end{align*}
Note that the times $t_i$ may be negative; however, by \Cref{l:midmode} each $\varphi^{(i)}$ is periodic. Thus for every $t_i\leq 0$ there exists a $t_i'>0$ such that $\varphi^{(i)}_{t_i}(q') = \varphi^{(i)}_{t_i'}(q')$ for all $q'$ in $\mathcal{Q}$. Let $t'$ be $t$ with all $t_i\leq 0$ replaced by $t_i'$. Then $t'$ is in $\mathbb{R}^{mn}_+$ and $\Phi^m(q,t')=\Phi^m(q,t)=q^*$.}
\end{proof}

\begin{figure}
	\centering
  \begin{subfigure}{.4\textwidth}
  \centering
  \includegraphics[width=\linewidth]{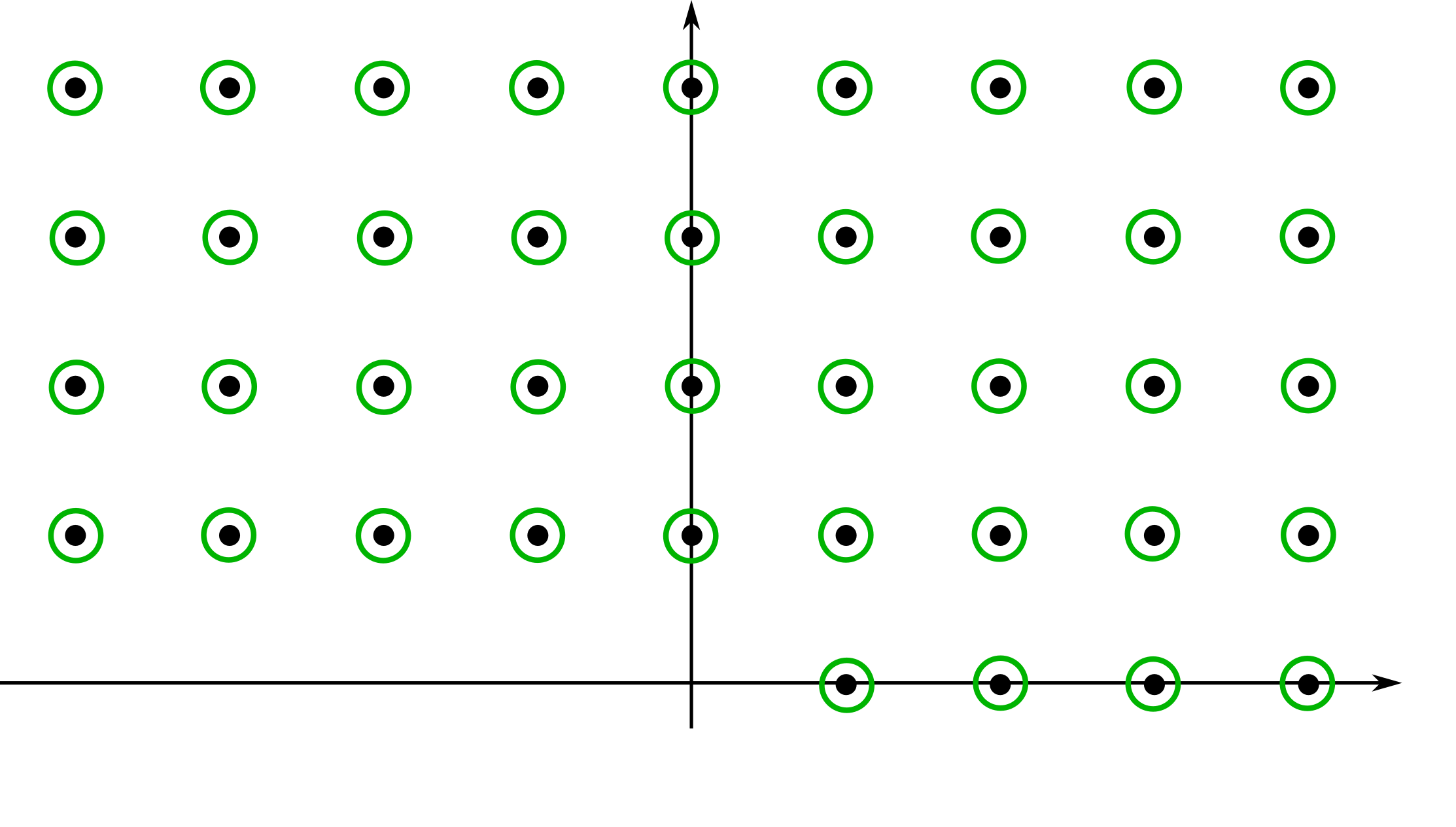}
  \caption{ }
  \label{f:steps1}
  \end{subfigure}%
  \,
  \begin{subfigure}{.4\textwidth}
    \centering
    \includegraphics[width=\linewidth]{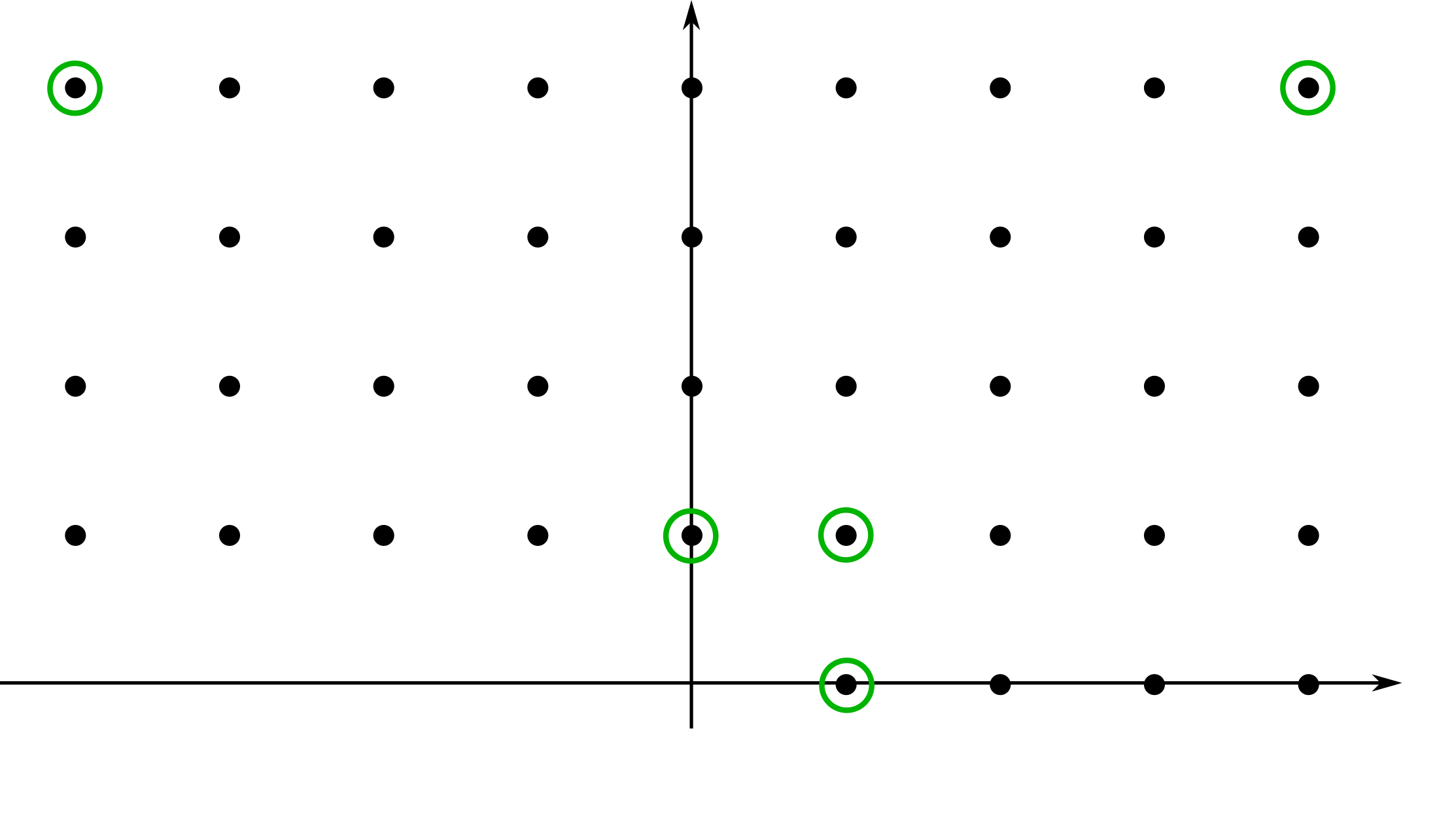}
    \caption{ }
    \label{f:steps2}
    \end{subfigure}%
    \,
    \begin{subfigure}{.4\textwidth}
      \centering
      \includegraphics[width=\linewidth]{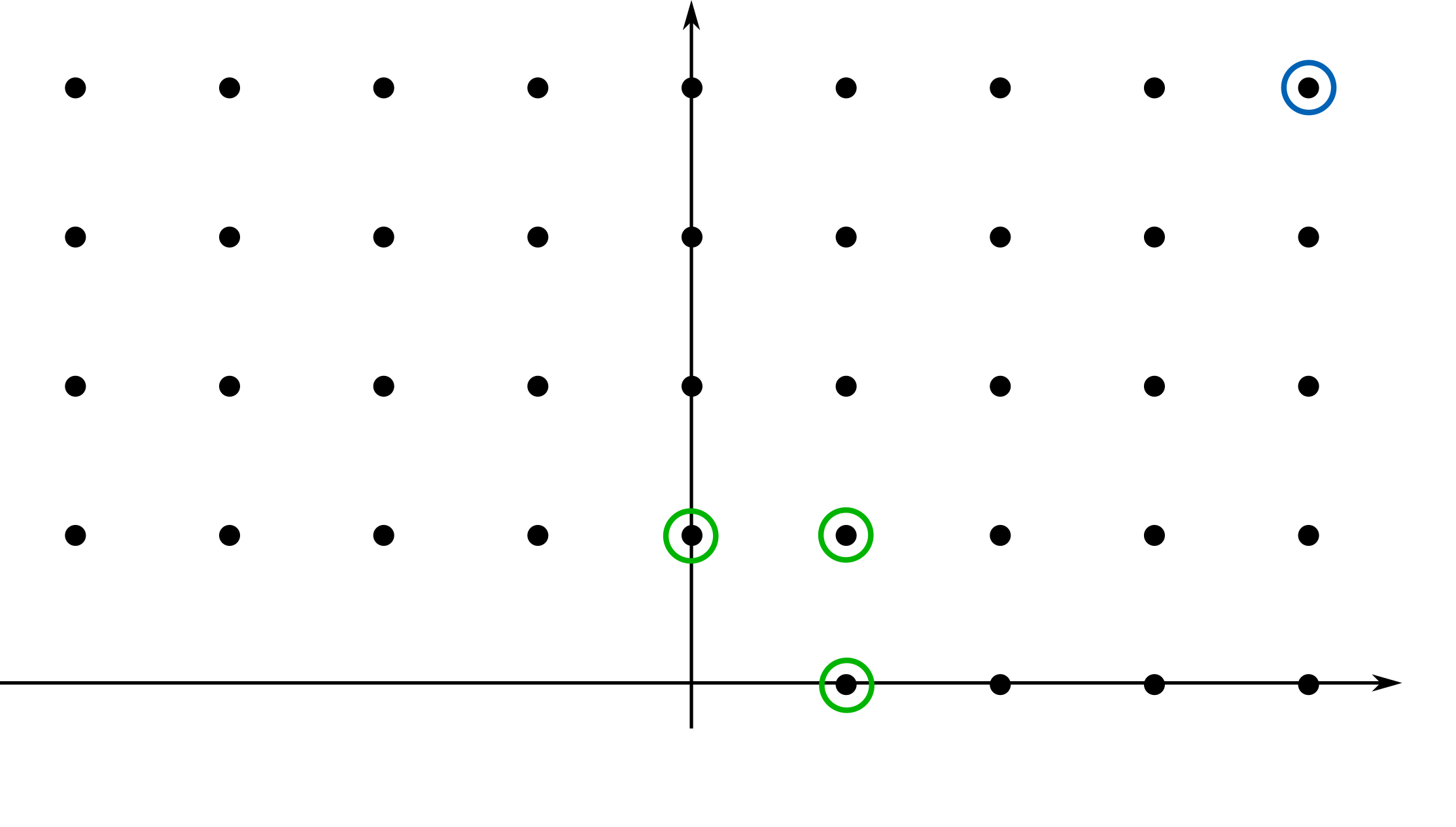}
      \caption{ }
      \label{f:steps3}
      \end{subfigure}%
			\,
			\begin{subfigure}{.4\textwidth}
				\centering
				\includegraphics[width=\linewidth]{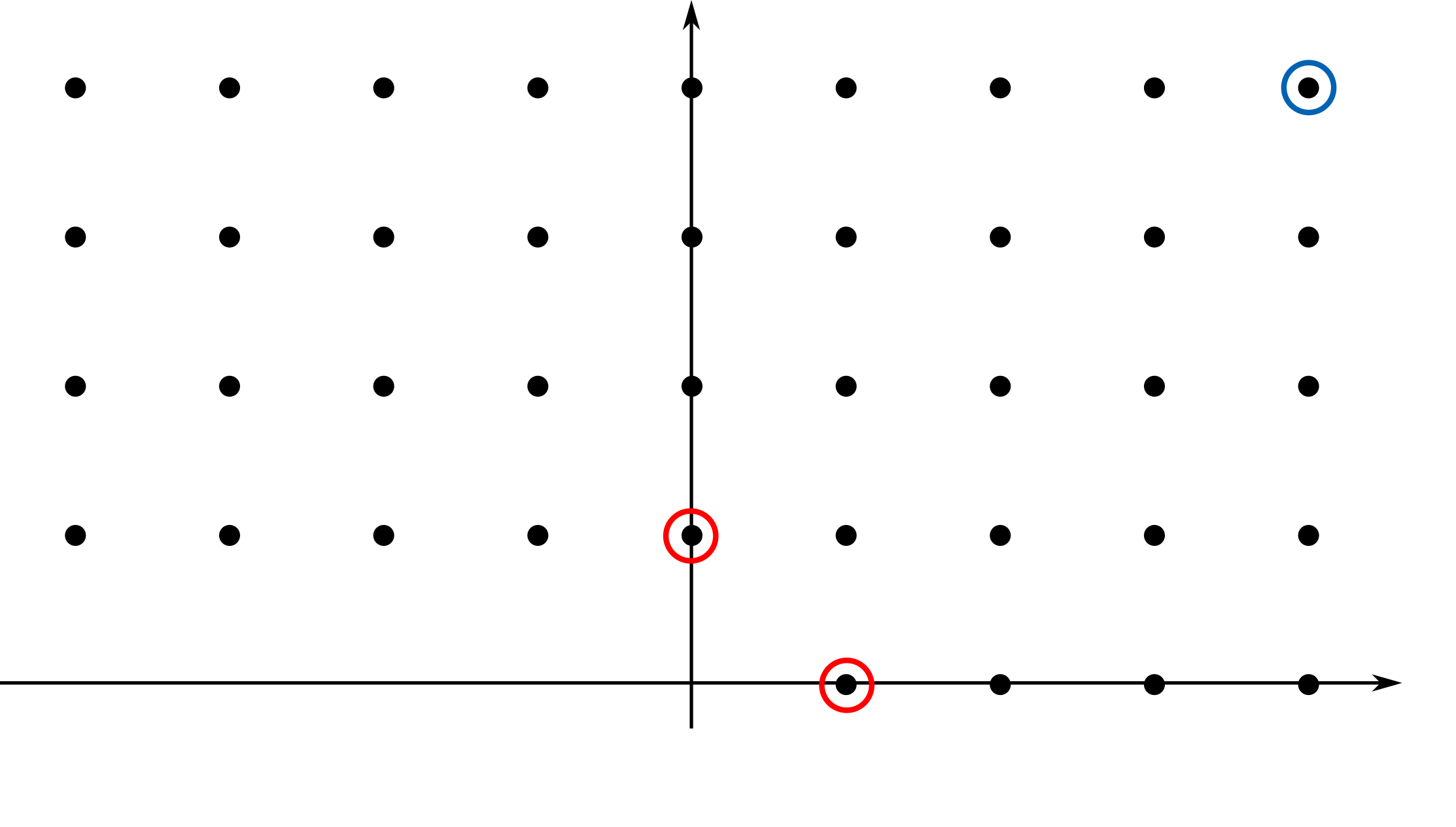}
				\caption{ }
				\label{f:steps4}
				\end{subfigure}%
    \caption{Representation of the state of the network in a generic initial state (a), after step 1 of the procedure in the proof of \pref{p:conteuler} (b), and after step 2 (c) and after step 3 (d) of the same procedure. In the above pictures, each point corresponds to a mode, \ie an element of $\mathbb{Z}^2_N$ while the color of each circle represents the real/complex value of the corresponding mode: zero (white, no circle), purely imaginary (red), purely real (blue) or having both nonvanishing real and imaginary parts (green). }
    \label{f:steps}
\end{figure}

Defining similarly to \eref{e:+} the operation of removing a coordinate from the set $\mathcal A$
  \begin{equ}\label{e:-}
		\AA \ominus {\bm \ell}  = \begin{cases} \AA \setminus \{{\bm \ell}\}\qquad&\text{if }  \ell \in \{j+k,j-k\}\cap \ZN \text{ for }{\bf j}, {\bf k} \in \AA, C_{jk}\neq 0, \mathrm{T}({\bf j})\cdot  \mathrm{T}({\bf k}) =  \mathrm{T}({\bm \ell})\,,\\ \AA & \text{else}\,, \end{cases}
  \end{equ}
  we now proceed to \emph{construct} (sequences of) times $\tau$ and interacting triples $\iota$ such that the transformations $\Phi_\tau^{(\iota)}$ of $q$ implement the operations $\oplus, \ominus$  from \eref{e:+}, \eref{e:-} through the flow of \eref{e:vf}, \ie such that $\AAc(q) \oplus {\bm \ell} = \AAc(\Phi_\tau^{\iota}(q))$ or $\AAc(q) \ominus {\bm \ell} = \AAc(\Phi_\tau^{\iota}(q))$ respectively. To do so we separate the possible interactions between the modes in two types:
\begin{equs}\label{e:types}
	a) &\quad \iota = \jkl\in\II~:~ |j|\neq |k|\neq|\ell|\,,\\
  b)& \quad	\iota = \jkl\in \II~:~ |j|= |k|\neq|\ell|\,.
	\end{equs}
	Note that these two types of interactions are exhaustive, since if $|j|=|k|=|\ell|$, $C_{j\ell} = C_{jk} = C_{k\ell} = 0$.

	The following preparatory lemmas describe the properties of these two types of interactions that we will leverage throughout our proof. The first one shows that for interactions of type a), ordering the indices so that $|j|<|k|<|\ell|$, it is always possible to activate all modes $\jj, \kk, \ll$ or to distribute the amplitude of the $k$-mode to the $j$ and $\ell$-modes reaching, in finite time, a state with $q_\kk=0$. As we show in the proof below, while such a point with $q_\kk=0$ always exists on the orbits of \eref{e:vf}, this point is reachable in finite time for $\iota = \jkl\in \II$ with $|j|< |k|< |\ell|$ only if
	\begin{equ}\label{e:degcond}
		E_{\iota}(q) \neq |k|^2 \EE_{\iota}(q)\,,
	\end{equ}
where $E_{\iota}(q)$ and $\EE_{\iota}(q)$ denote the energy and enstrophy of the coordinates in $\iota \in \II$:
\begin{equ}\label{e:partiales}
	E_{\iota}(q) \eqdef \sum_{\ll \in \iota} |q_\ll|^2\,,\qquad \EE_{\iota}(q) \eqdef \sum_{\ll \in \iota} \frac{|q_\ll|^2}{|\ell|^2}\,.
\end{equ}
\begin{figure}
  \centering
  \begin{subfigure}{.4\textwidth}
\centering
\includegraphics[width=0.8\textwidth]{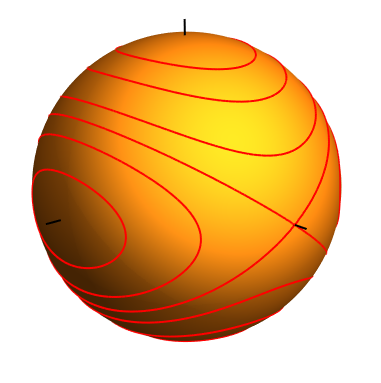}
\caption{ }
\end{subfigure}%
\,
\begin{subfigure}{.4\textwidth}
  \centering
  \includegraphics[width=0.9\textwidth]{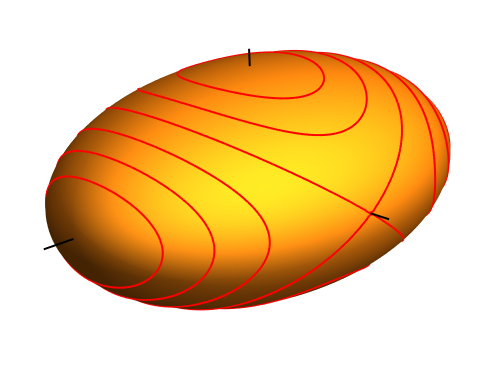}
  \caption{ }
  \end{subfigure}%
\caption{Orbits $\mathcal Q_\iota$ of \eref{e:vf} (in red) corresponding in (A) to various values of the energy $\mathcal E_\iota(q)$ on the sphere of constant enstrophy $E_\iota(q)$ and in (B) to various values of the enstrophy $E_\iota(q)$ on the ellipsoid of constant energy $\mathcal E_\iota(q)$. The axes are, sequentially, $q_\kk, q_\jj, q_\ll$. The orbit with a degenerate point at the pole of the sphere or ellipsoid corresponds to values of $E_\iota , \mathcal E_\iota$ violating \eref{e:degcond}.  }
\label{f:sphere}
\end{figure}

In the following lemma and throughout the section, we abuse notation slightly by defining $\text{sign}(x) = +1$ for $x \in [0,\infty)$ and $-1$ otherwise.
	\begin{lemma}\label{l:midmode}
	  Fix  $\iota = \jkl \in \II$ with $|j|<|k|<|\ell|$. Let \om{$q$ be a nondegenerate point in $\mathcal{Q}_0$ satisfying} \eref{e:degcond} and let $q_\llb = 0$ for at most an index $\llb \in \{\jj,\kk,\ll\}$. \aa{Then the orbit of $V_\iota$ is periodic and} there exist $\tau_-^\iota, \tau_+^\iota \geq 0$ such that
		\begin{enumerate}[label=(\alph*)]
			\item $\pit{\tau_-^\iota}(q) = q'$ with $q_\kk' = 0$, $\mathrm{sign}(q_\jj) = \mathrm{sign}(q_\jj')$ and $\mathrm{sign}(q_\ll) = \mathrm{sign}(q_\ll')$,
			\item $\pit{\tau_+^\iota}(q) = q''$ with $q_\jj'', q_\kk'', q_\ll'' \neq 0$, $\mathrm{sign}(q_\jj) = \mathrm{sign}(q_\jj'')$ and $\mathrm{sign}(q_\ll) = \mathrm{sign}(q_\ll'')$.
		\end{enumerate}
    Furthermore, if $|j|^2 \EE_{\iota}(q) < E_{\iota}(q)< |k|^2 \EE_{\iota}(q)$, there exists $\tau_=^\iota \geq 0$ such that
    \begin{enumerate}[label=(\alph*)]
      \setcounter{enumi}{2}
    \item $\pit{\tau_=^\iota}(q) = q'''$ with $q_\ll''' = 0$, $\mathrm{sign}(q_\jj) = \mathrm{sign}(q_\jj''')$ and $\mathrm{sign}(q_\kk) = \mathrm{sign}(q_\kk''')$.
  \end{enumerate}
	\end{lemma}

\noindent \begin{proof}
	  We consider the intersection between the sphere and the ellipse corresponding to the enstrophy and the energy in the coordinates $\iota = \jkl\in \II$ of interest, resulting in the set
		\begin{equ}
			\QQ_\iota \eqdef \pg{ (q_\jj',q_\kk',q_\ll') \in \mathbb R^3~:~|q_\jj'|^2+|q_\kk'|^2+|q_\ll'|^2 = E_{\iota}(q), \frac{|q_\jj'|^2}{|j|^2}+\frac{|q_\kk'|^2}{|k|^2}+\frac{|q_\ll'|^2}{|\ell|^2} = \EE_{\iota}(q)}\,.
			\end{equ}
			This set is represented in \fref{f:sphere}. We observe that this set has exactly $2$ disjoint simply connected components when $|j|^2 \EE_{\iota}(q) < E_{\iota}(q)< |k|^2 \EE_{\iota}(q)$ and
$|k|^2 \EE_{\iota}(q) < E_{\iota}(q)< |\ell|^2 \EE_{\iota}(q)$. These components are diffeomorphic to $S^1$.
		By continuity the dynamics are limited to one such component of $\QQ_\iota$. Furthermore, $|\dot q|^2$ is uniformly bounded away from $0$ on each such component: the fixed points of \eref{e:vf} must have at least two coordinates vanishing, which cannot be realized on the curves of interest.  Therefore the dynamics on these sets are periodic.

We start by proving part (b) of the lemma. If $q_\jj, q_\kk, q_\ll \neq 0$ the result follows by choosing $\tau_+^\iota =0$. Else, if $q_\llb =0$ for $\llb \in \iota$ the result follows immediately choosing $\tau_+^\iota$ small enough by combining the continuity of the flow $\Phi_t^\iota$ and the fact that $\dot q_\llb = C_{l'l''}q_{\llb'}q_{\llb''}\neq  0$ for $\{\llb',\llb''\} = \iota \setminus \{\llb\}$.

To prove part (a) we consider the cases where $|j|^2 \EE_{\iota}(q) < E_{\iota}(q)< |k|^2 \EE_{\iota}(q)$ and $|k|^2 \EE_{\iota}(q) < E_{\iota}(q)< |\ell|^2 \EE_{\iota}(q)$ separately. In the first case, we see that there is no point $q \in \QQ_\iota$ with $q_\jj = 0$: if that were the case we would have
\begin{equ}
	E_{\iota}(q) = q_\kk^2+q_\ll^2 = |k|^2\pc{\frac{q_\kk^2}{|k|^2}+\frac{q_\ll^2}{|k|^2}} > |k|^2 \EE_\iota(q)\,,
\end{equ}
contradicting our assumption. Consequently the points $(p_\jj,0,p_\ll), (p_\jj,0,-p_\ll)$ with $p_\ll >0$, $\mathrm{sign}(p_\jj)=\mathrm{sign}(q_\jj)$ and
\begin{equ}
	p_\jj^2 + p_\ll^2 = E_\iota(q)\,,\qquad \frac{p_\jj^2}{|j|^2} + \frac{p_\ll^2}{|\ell|^2} = \EE_\iota(q)\,,
\end{equ}
belong to the same connected component as $q$ and by the lower bound on the velocity on this connected component both these points are reachable in finite time from $q$. This also proves part (c) by continuity of the dynamics. The second case where $|k|^2 \EE_{\iota}(q) < E_{\iota}(q)< |\ell|^2 \EE_{\iota}(q)$ can be handled analogously: in this case we have $\QQ_\iota \cap \{q_\ll=0\} = \emptyset$ and we can reach $(p_\jj,0,p_\ll), (-p_\jj,0,p_\ll)$ with $p_\jj >0$, $\mathrm{sign}(p_\ll)=\mathrm{sign}(q_\ll)$ in finite time.
	\end{proof}

	The following lemma considers interactions of type b) in \eref{e:types}. Recalling the definition $j^\perp\eqdef (j_2, -j_1)$
	 we show that interactions with $|j|=|k|\neq|\ell|$ leave component $\ll$ fixed and move $\jj, \kk$ in a circle at constant angular speed.

	\begin{lemma}\label{l:samenorm}
	  Fix an unordered interacting triple $\iota = \jkl$ with $|k| = |j|$ and $q_\ll \neq 0$. For all $\theta$ in $[0,2\pi)$ there exists $t\geq 0$ such that $\pit{t}(q) = q'$ with $(q_\jj',q_\kk') = \sqrt{q_\jj^2+q_\kk^2}(\cos(\theta), \sin(\theta))$ and $q_\ll' = q_\ll\,$.
	\end{lemma}

	\begin{corollary}\label{c:samenorm}
		Fix an (unordered) interacting triple  $\iota = \jkl \in \II$ with $|k| = |j|$ and let $q_\ll, q_\kk \neq 0$. Then there exist $\tau_+^{\iota}, \tau_-^{\iota} \geq 0 $ such that $(\phi_{\tau_+^\iota}^{\iota} (q))_{\mathbf j} > 0$ and $(\phi_{\tau_-^\iota}^{\iota} (q))_{\mathbf j} = 0$\,.
	\end{corollary}

\noindent	\begin{proof}[Proof of \lref{l:samenorm}]
Recall from \eref{e:c} that if $|j|=|k|\neq |\ell|$ we have $C_{jk} = 0$. This implies that, by our choice of $|k| = |j|$, $\dot q_\ll = 0$ and $q'_\ll = q_\ll$. Again by \eref{e:c} and since to have an interacting triple $\ell = j+k$ we must have
\begin{equ}
	{\langle k^\perp,\ell\rangle} = {\langle k^\perp,k+j\rangle} = {\langle k^\perp,j\rangle} =  {\langle (k + j)^\perp - j^\perp,j\rangle} = {\langle \ell^\perp,j\rangle} = -{\langle j^\perp,\ell\rangle}\,,
\end{equ}
	 so that
\begin{equ}
	C_{k\ell} = \frac{\langle k,\ell^\perp\rangle}{4\pi}\bigg(\frac{1}{\lvert k\rvert^2}-\frac{1}{\lvert\ell\rvert^2}\bigg) = -\frac{\langle j,\ell^\perp\rangle}{4\pi}\bigg(\frac{1}{\lvert j\rvert^2}-\frac{1}{\lvert\ell\rvert^2}\bigg)=-C_{j\ell}\,.
\end{equ}
This implies that the dynamics of the vector $\tilde q \eqdef (q_\jj,q_\kk)$ can be written as $\dot{\tilde q} = \tilde C \tilde q^{\perp}$ for $\tilde C \eqdef C_{j\ell}q_\ll\neq 0$, proving the claim.
	\end{proof}


\subsubsection{Existence of invariant measure}\label{sec:EulerInvariant}


As with conservative Lorenz-96, each vector field of the 2D Euler splitting is divergence free and so Lebesgue measure in $\mathbb{R}^n$ is invariant. Consequently, we have

\begin{lemma}\label{lem:EulerInvariant}
Let $\lambda$ denote the Lebesgue measure on $\mathbb R^n$. \jcm{The
measure obtained by conditioning  $\lambda$
to lie on $\mathcal Q \subset Q_0(E,\mathcal{E})$, (or equivalently
conditioned to lie on $Q_0(E,\mathcal{E})$) is
$P_h$-invariant.}
\end{lemma}

\begin{proof}
As in the proof of \pref{prop5.1} we have that Lebesgue measure in $\mathbb R^n$ is $P_h$-invariant. Since the vector fields $V_k$ defined in~\eqref{eq:euler_splitting} are divergence free, the continuity equation\footnote{\aa{As in the proof of \Cref{prop5.1}, the continuity equation is intended here in the weak sense.}} reads
\begin{equs}
	\partial_t\lambda + \diverg\left(V_k\lambda \right)
		= \partial_t\lambda + \nabla\lambda \cdot V_k = 0\,.
\end{equs}
Because each flow $\varphi^{(k)}$ conserves energy $E$ and enstrophy $\mathcal E$, we locally fiber $\mathbb R^n$ using coordinates $(E, \mathcal E, \theta) \in \mathbb R_+\times \mathbb R_+ \times \mathbb R^{n-2}$. In these coordinates, we have $V_k(E, \mathcal E, \theta) = 0\, \partial_E + 0\, \partial_{\mathcal E} +  v_k(E, \mathcal E, \theta) \nabla_\theta$ so by a change of coordinates of the divergence operator the stationary equation becomes
\begin{equ}
	0 = \diverg\left(V_k(x)\lambda(x)\right) = u(E, \mathcal E, \theta)  \diverg_\theta (\lambda(E, \mathcal E, \theta) v_k(E, \mathcal E, \theta))\,,
\end{equ}
where $\diverg_\theta$ denotes the ``angular'' terms of the divergence in  $(E, \mathcal E, \theta)$-coordinates, and $u(E, \mathcal E, \theta)$ result from the change of variables.
Hence, we can factor the solution
$\lambda(E, \mathcal E, \theta) = \bar \lambda(\theta|E, \mathcal E) \cdot \lambda^\perp(E, \mathcal E)$,
where $ \bar \lambda(\theta|E, \mathcal E)$ is the conditional density of Lebesgue measure on a fiber, solving $u(E, \mathcal E, \theta)  \diverg_\theta (\bar \lambda(\theta|E, \mathcal E) v_k(E, \mathcal E, \theta)) = 0$ for any choice of $E/(2N^2)< \mathcal E < E$.
This proves the invariance of $\bar \lambda(\theta|E, \mathcal E)$ under the flow map for any value of the flow times $\tau$.
The stationarity of $\bar \lambda(\theta)$  under $P_h$ follows immediately as in \pref{prop5.1}
\end{proof}


\subsubsection{Spanning}\label{sec:EulerSpanning}


For $j,k,\ell\in\mathbb{Z}^2_N$ with $j+k-\ell=0$ define $M_{jk\ell}$ to be the matrix
\begin{align}\label{100}
	M_{jk\ell}\eqdef\begin{pmatrix}
			\vline & \vline & \vline & \vline \\
			V_{a_ja_ka_\ell} & V_{a_jb_kb_\ell} & V_{b_ja_kb_\ell} & V_{b_jb_ka_\ell} \\
			\vline & \vline & \vline & \vline
		\end{pmatrix}
			&=
			{\scriptsize
			\begin{pmatrix}
				\phantom{-}C_{k\ell}a_ka_\ell & \phantom{-}C_{k\ell}b_kb_\ell & \phantom{-}0 & \phantom{-}0 \\
				\phantom{-}0 & \phantom{-}0 & \phantom{-}C_{k\ell}a_kb_\ell & -C_{k\ell}b_ka_\ell \\
				\phantom{-}C_{j\ell}a_ja_\ell & \phantom{-}0 & \phantom{-}C_{j\ell}b_jb_\ell & \phantom{-}0 \\
				\phantom{-}0 & \phantom{-}C_{j\ell}a_jb_\ell & \phantom{-}0 & -C_{j\ell}b_ja_\ell \\
				-C_{jk}a_ja_k & \phantom{-}0 & \phantom{-}0 & \phantom{-}C_{jk}b_jb_k \\
				\phantom{-}0 & -C_{jk}a_jb_k & -C_{jk}b_ja_k & \phantom{-}0
			\end{pmatrix}
			}
\end{align}
and let $M_{jk\ell}'$ and $M_{jk\ell}''$ be the $4$-by-$4$ and $2$-by-$4$ matrices consisting of the bottom four and bottom two rows of $M_{jk\ell}$, respectively.  Straightforward Gaussian elimination shows that $M$, $M'$, and $M''$ have ranks $4$, $3$, and $2$ whenever $C_{jk}$, $C_{j\ell}$, $C_{k\ell}$, $a_j$, $b_j$, $a_k$, $b_k$, $a_\ell$, and $b_\ell$ are nonzero.

Recalling that a point $q\in\mathbb{R}^n$ is \textit{generic} if all its coordinates are nonzero, we have

\begin{lemma}\label{lem6.1}
The family of vector fields
\begin{align*}
	\mathcal{V} &\coloneqq \big\{V_{a_ja_ka_\ell}, V_{a_jb_kb_\ell}, V_{b_ja_kb_\ell}, V_{b_jb_ka_\ell} : j,k,\ell\in\mathbb{Z}^2_N\ \text{and}\ j+k-\ell=0\big\}
\end{align*}
span $T_q\mathcal{Q}$ at every generic point $q$ in $\mathcal{Q}$.
\end{lemma}

\noindent\begin{proof} Fix a generic point $q$ in $\mathcal{Q}$. The main idea of the proof is to choose an enumeration of $\mathbb{Z}^2_N$ and a subset of vector fields from $\mathcal{V}$ so that the matrix made up of these vector fields evaluated at $q$ is in a convenient form whose rank is readily deduced. Formally, the enumeration is the bijection $F:\mathbb{Z}^2_N\to\{1,\dots,2N(N+1)\}$ given by
\begin{align*}
	F(j) &\coloneqq
		\begin{cases}
			1 & j=(1,0)\,, \\
			5+N & j=(2,0)\,, \\
			j_1+N(2N+1) & j=(j_1,0)\ \text{with}\ j_1>2\,, \\
			j_1+2+N & j=(j_1,1)\ \text{with}\ j_1<3\,, \\
			j_1+3+N & j=(j_1,1)\ \text{with}\ j_1\geq 3\,, \\
			j_1+2-N+(2N+1)j_2 & j=(j_1,j_2)\ \text{with}\ j_2>1\,.
		\end{cases}
\end{align*}
Figure $2$ gives this enumeration in the case $N=4$. Informally, $F$ starts at $(1,0)$, then counts lattice points from left to right along the horizontal line $y=1$ until the point $(2,1)$, which corresponds to $4+N$. It then assigns $5+N$ to $(2,0)$ and continues counting along the line $y=1$. From there it moves up to the lines $y=2$, $y=3$, and so on, counting from left to right along each. Finally, it goes back down to the line $y=0$ and counts the remaining indices from left to right.

\begin{figure}
\centering
\includegraphics[width=0.58\textwidth]{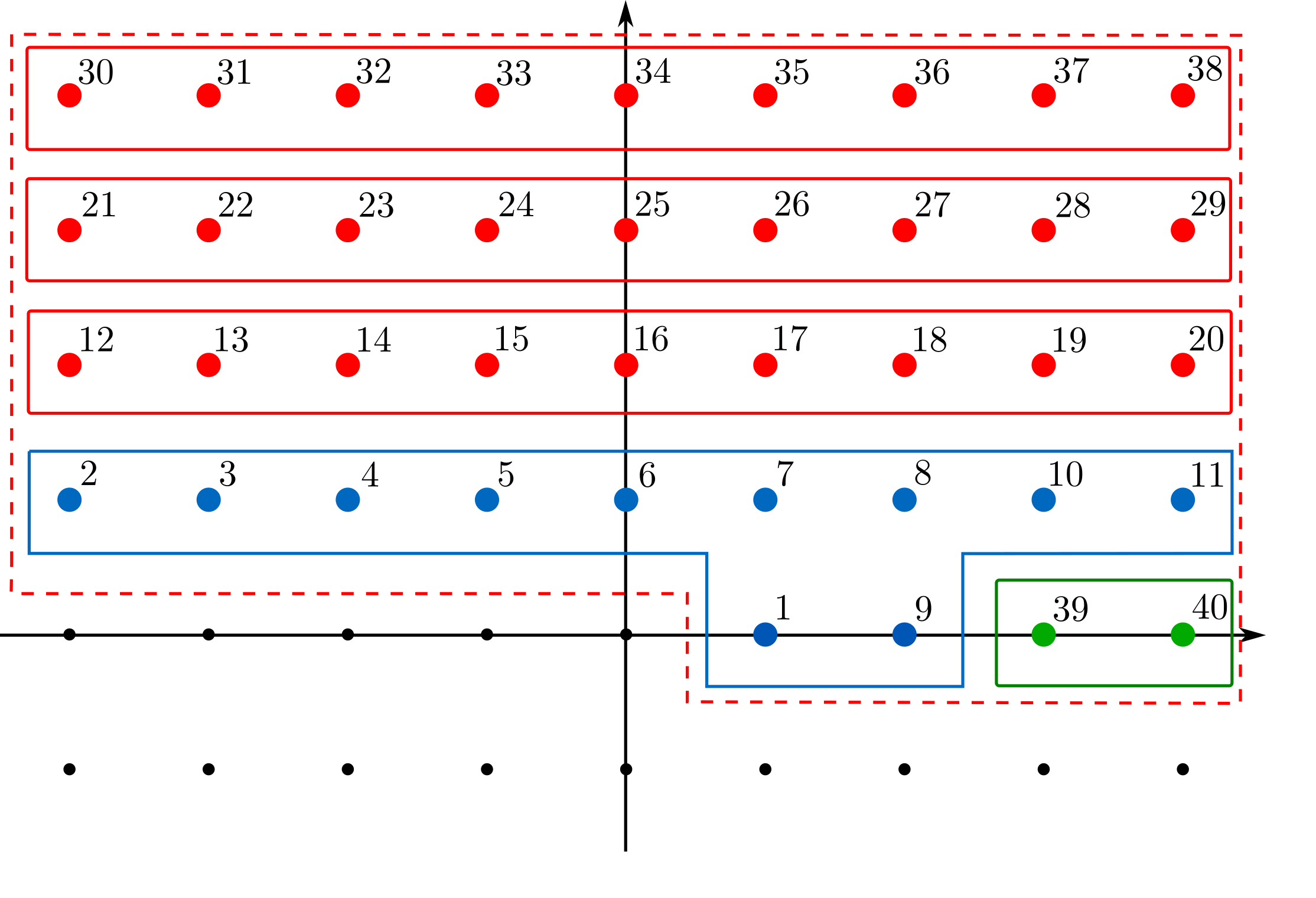}
\caption{Ordering of $\mathbb{Z}^2_N$ when $N=4$\,.}
\label{f:spanning}
\end{figure}

The motivation for $F$ is that all horizontally-adjacent indices $(j_1,j_2)$ and $(j_1+1,j_2)$ form an interacting triple together with $(1,0)$. Fix for the moment an integer $y>1$ and consider the $y$th horizontal line of $\mathbb{Z}^2_N$; that is, the points with second coordinate $y$. These are outlined by red blocks in Figure~\ref{f:spanning}. By the preceding remarks we can choose the vector fields corresponding to the horizontally-adjacent indices \aa{and concatenate them column-wise} to get the block matrix
\begin{align*}
	B_y \coloneqq
		\left(\begin{array}{c|c|c|c}
			\widetilde{M}_{\aa{y}} & * & * & *  \\
  			\hline
  	 		0 & M_{\aa{y, -N+2}}'' & * &  * \\
  			\hline
  	 		0 & 0 & \ddots & * \\
  	 		\hline
  	 		0 & 0 & 0 & M_{\aa{y,N}}''
		\end{array}\right).
\end{align*}
Here, \aa{slightly abusing notation, each $M_{y,i}''$} is the $2$-by-$4$ matrix consisting of the bottom two rows of~\eqref{100} for the \aa{indices $j = (1,0), k= (i-1,y), \ell = (j,y)$}  and
\begin{align*}
	\widetilde{M}_y &\coloneqq
		\begin{pmatrix}
			C_{j\ell}a_ja_\ell & 0 & C_{j\ell}b_jb_\ell & 0 & 0 & 0 \\
			0 & C_{j\ell}a_jb_\ell & 0 & -C_{j\ell}b_ja_\ell & 0 & 0 \\
			-C_{jk}a_ja_k & 0 & 0 & C_{jk}b_jb_k & -C_{j'k'}a_{j'}a_{k'} \\
			0 & -C_{jk}a_jb_k & -C_{jk}b_ja_k & 0 & 0 & -C_{j'k'}a_{j'}b_{k'}
		\end{pmatrix}
\end{align*}
where $j=(1,0), k=(-N,y), \ell=(-N+1,y)$ and $j'=(0,1)$ and $k'=(-N+1,y-1)$. This is $M'$ with two columns from the interacting triple $(0,1), (-N+1,y-1), (-N+1,y)$ adjoined to the end. Note that these adjoined columns contribute entries in the coordinates corresponding to $(0,1)$ and $(-N+1,y-1)$, but these come before all indices in the $y$th row for our ordering. By adding the latter two columns, $\widetilde{M}_y$ has rank $4$ at any generic point. Further,  since each $M_{\aa{y,j}}''$ has rank $2$, each $B_y$ has rank $4+2(2N-1) = 4N+2$. This establishes spanning of the red blocks in Figure~\ref{f:spanning}.

For the blue block we perform a similar procedure to the one above to get
\begin{align*}
	B_1 \coloneqq
		\left(\begin{array}{c|c|c|c|c|c}
			M_{123} & * & * & * & * & * \\
  			\hline
  	 		0 & M_{\aa{1,-N+2}}'' & * &  * & * & * \\
  			\hline
  	 		0 & 0 & \ddots & * & * & * \\
  	 		\hline
  	 		0 & 0 & 0 & \widehat{M} & * & * \\
  	 		\hline
  	 		0 & 0 & 0 & 0 & \ddots & * \\
  	 		\hline
  	 		0 & 0 & 0 & 0 & 0 & M_{\aa{1,N}}''
		\end{array}\right)
\end{align*}
where $M_{123}$ is the matrix from~\eqref{100} for the interacting triple $(1,0), (-N,1), (-N+1,1)$, each $M''$ is as before, and $\widehat{M}$ is the $6$-by-$8$ matrix
\begin{align*}
	\widehat{M} &\coloneqq
		\left(\begin{array}{ccc|cc}
			&
			\begin{matrix}
			& & \\
			& M_{1,N+3,N+4}' & \\
			& &
			\end{matrix}
			&
			&
			\begin{matrix}
			0 &  0 & 0 & 0 \\
			0 & 0 & 0 & 0 \\
			--- & --- & --- & ---
			\end{matrix}
			\\
			&
			\begin{matrix}
			--- & --- & --- & --- \\
			0 & 0 & 0 & 0 \\
			0 & 0 & 0 & 0 \\
			\end{matrix}
			&
			&
  	 		\begin{matrix}
  	 		 & & \\
			& M_{N+2,N+4,N+5}' & \\
			& &
			\end{matrix}
		\end{array}\right)
\end{align*}
located at the rows corresponding to $N+3, N+4$, and $N+5$. The reason for $\widehat{M}$, and for considering the blue block separately, is that $C_{jk}=0$ when $j=(1,0)$ and $k=(0,1)$. The matrix $M$ has rank $6$ at a generic point. Since $M_{123}$ has rank $4$, $\widehat{M}$ has rank $6$, and each of the $2N-3$ remaining $M''$ blocks has rank $2$, the matrix $B_y$ has rank $4+6+2(2N-3) = 4N+4$.

Finally, none of the indices of the green block interact with $(1,0)$ since the $C_{jk}$ are all $0$ in this case. However, by an entirely similar procedure to above, we can use the interactions between $(0,1), (x,0)$, and $(x,1)$ for $x>1$ to get a rank $2(N-2)$ block matrix for the last $N-2$ coordinates of the form
\begin{align*}
	B_{N+1} \coloneqq
		\left(\begin{array}{c|c|c|c}
			\tilde M_{\aa{0,2}}'' & * & * & *  \\
  			\hline
  	 		0 & \tilde M_{\aa{0,3}}'' & * &  * \\
  			\hline
  	 		0 & 0 & \ddots & * \\
  	 		\hline
  	 		0 & 0 & 0 &  \tilde M_{\aa{0,N}}''
		\end{array}\right)
\end{align*}
\aa{where $\tilde M_{\aa{0,x}}'' = M_{(0,1),(x,0),(x,1)}''$ for $M_{jk\ell}''$ consisting of the two bottom rows of \eref{100}.}
Combining the above results we observe that there is an ordering of indices and vector fields such that the matrix whose columns consist of these vector fields has the form
\begin{align*}
	B \coloneqq
		\left(\begin{array}{c|c|c|c}
			B_1 & * & * & *  \\
  			\hline
  	 		0 & B_2 & * &  * \\
  			\hline
  	 		0 & 0 & \ddots & * \\
  	 		\hline
  	 		0 & 0 & 0 &  B_{N+1}
		\end{array}\right).
\end{align*}
Moreover, $B$ has rank
\begin{align*}
	\text{rank}(B) &= \text{rank}(B_1) + \text{rank}(B_{N+1}) + \sum_{y=2}^N \text{rank}(B_y)
		= 4N(N+1)-2
		= n-2
\end{align*}
at every generic point in $\mathcal{Q}$. Now since the dynamics conserve energy and enstrophy, every tangent vector to $\mathcal{Q}$ is perpendicular to the normal vectors for these two quantities which are linearly independent at every generic point. Therefore the maximum dimension of $T_q\mathcal{Q}$ is $n-2$, and by the above argument we have shown the vector fields $\mathcal{V}$ span $T_q\mathcal{Q}$ at $q$.
\end{proof}


\section{Adding forcing and dissipation: Lorenz-96 and 2D Navier-Stokes}\label{sec:ForceAndDissip}


In this section we add dissipation and fixed body forcing to both conservative Lorenz-96 and Galerkin approximations of 2D Euler by introducing a new vector field
\begin{align}
  \label{eq:fluctuationDisiaption}
  V_0(x)= -\nu\Lambda x + F
\end{align}
to the splittings constructed in Sections~\ref{sec:Lorenz} and~\ref{sec:Euler}, where $\nu>0$ is an arbitrary constant, $F$ a fixed nonzero vector with nonnegative entries, and $\Lambda$ a linear operator satisfying
\begin{align}
  \label{eq:disiaptionOperator}
  \Lambda x \cdot x \geq
\alpha \lVert x\rVert^2
\end{align}
for some $\alpha>0$. For the remainder of this section we consider random splittings associated to \om{families of complete, smooth vector fields $\mathcal{V}=\{V_k\}_{k=0}^n$} on $\mathbb{R}^d$ satisfying

\begin{assumption}\label{assump:VectorFields}
$V_0$ is as in~\eqref{eq:fluctuationDisiaption} and the flows of the other $V_k$ conserve Euclidean norm.
\end{assumption}

Fix $h>0$ and let $P_h$ be the transition kernel of a random splitting satisfying Assumption \ref{assump:VectorFields}. When $\Lambda$ is the identity matrix, the addition of $V_0$ to the splitting of conservative Lorenz-96 gives a splitting of the full Lorenz-96 model,~\eqref{lorenz96full}, while for 2D Euler the resulting $V_0$ corresponds to a friction or drag term sometimes called \textit{Ekman damping}. When $\Lambda$ is diagonal with diagonal entry $\lvert k\rvert^2$ in the spots associated to\footnote{Recall that for each index $k\in\mathbb{Z}_N^2$, we have two real coordinates $a_k$ and $b_k$.} $a_k$ and $b_k$, which corresponds to
a Laplacian written in Fourier space, the addition of $V_0$ to the splitting of 2D Euler gives a splitting of 2D Navier-Stokes,~\eqref{eq:2DNS}.

Note that the dissipative part of $V_0$ in~\eqref{eq:fluctuationDisiaption} depends linearly on $x$ whereas the forcing is constant. Thus dissipation dominates forcing for sufficiently large $x$ and, since the remaining vector fields are conservative, the splitting dynamics cannot grow too large. Specifically, letting $\Phi_{h\tau}$ be as in~\eqref{eq:Phi} but with the solution $\phi^{(0)}$ of $\dot x=V_0(x)$
appended to the beginning of each cycle, we have

\begin{lemma}\label{lem:Dissipative}
Under Assumption~\ref{assump:VectorFields} for any initial $x$ and $m>0$,
  \begin{align}
     \|\Phi_{h\tau}^m(x)\|^2 \leq \|x\|^2e^{-\nu\alpha h
    \sum_{k=0}^{m} \tau_{k(n+1)}} + \frac{1}{\nu^2\alpha^2}\|F\|^2\left(  1 -  e^{-\nu\alpha h
    \sum_{k=0}^m \tau_{k(n+1)}}\right).
  \end{align}
\end{lemma}

\noindent \begin{proof} Letting $\varphi=\varphi^{(0)}$, we have
\begin{align*}
	\partial_t\lVert\varphi_t\rVert^2 &= 2\langle F,\varphi_t\rangle - 2\nu\langle \Lambda\varphi_t,\varphi_t\rangle
		\leq \frac{1}{\nu\alpha}\lVert F\rVert^2+\nu\alpha\lVert\varphi_t\rVert^2-2\nu\alpha\lVert\varphi_t\rVert^2
		= \frac{1}{\nu\alpha}\lVert F\rVert^2-\nu\alpha\lVert\varphi_t\rVert^2,
\end{align*}
where the inequality follows from~\eqref{eq:disiaptionOperator} and $2\langle F,\varphi_t\rangle\leq (\nu\alpha)^{-1}\lVert F\rVert^2+\nu\alpha\lVert\varphi_t\rVert^2$. Solving
\begin{align*}
	\dot{y} = \frac{1}{\nu\alpha}\lVert F\rVert^2-\nu\alpha y
\end{align*}
from $y(0)=\lVert x\rVert$ together with the comparison theorem for ODEs \cite{McNabb} then gives
\begin{align*}
	\lVert\varphi_t(x)\rVert^2 &\leq \lVert x\rVert^2 e^{-\nu\alpha t}+\frac{1}{\nu^2\alpha^2}\lVert F\rVert^2\left(1-e^{-\nu\alpha t}\right)
\end{align*}
for all time. Furthermore, since $\varphi^{(k)}$ conserves norm for $1\leq k\leq n$, the above implies
\begin{align*}
	\lVert \Phi_{h\tau}(x)\rVert^2 &= \lVert\varphi^{(n)}_{h\tau_n}\circ\cdots\circ\varphi^{(0)}_{h\tau_0}(x)\rVert^2
		= \lVert \varphi^{(0)}_{h\tau_0}(x)\rVert^2
		\leq \lVert x\rVert^2 e^{-\nu\alpha \tau_0}+\frac{1}{\nu^2\alpha^2}\lVert F\rVert^2\left(1-e^{-\nu\alpha \tau_0}\right).
\end{align*}
The result then follows by straightforward induction on the number of
cycles, namely $m$. \end{proof}

\begin{remark}
The convergence results of \Cref{sec:Convergence} do not directly apply to Lorenz-96 and Galerkin approximations of 2D Navier-Stokes since $\mathcal{V}$-orbits are generally unbounded in both models. However, \Cref{lem:Dissipative} implies that any splitting starting from $x$ whose vector fields satisfy \Cref{assump:VectorFields} will lie inside the ball of radius $\lVert x\rVert^2+(\nu\alpha)^{-2}\lVert F\rVert^2$ centered at the origin for all nonnegative times. In particular, since the splitting vector fields are smooth, a bound analogous to \eqref{eq:Bounded} holds for all $x$ in the ball $B_r(0)$ of radius $r$ centered at the origin in the ambient Euclidean space. Thus all convergence results of Section~\ref{sec:Convergence} hold for these random splittings when $\mathcal{C}^k(\mathcal{X})$ is replaced by $\mathcal{C}^k_r(\mathcal{X})$, the space of $k$-times continuously differentiable functions that vanish outside $B_r(0)$. Intuitively, this says that for any initial condition $x$, the trajectories of a random splitting satisfying \Cref{assump:VectorFields} will converge on average and almost surely as $h\to 0$ to the trajectory of the true dynamics starting from $x$.
\end{remark}

\begin{corollary}\label{cor:LyapunovFunction}
The Euclidean norm is a Lyapunov function for $P_h$. That is, there exist constants $K\geq 0$ and $\gamma\in(0,1)$ such that for all $x\in\mathbb{R}^d$,
\begin{align*}
	\left(P_h\lVert\cdot\rVert\right)(x) &\leq \gamma\lVert x\rVert+K.
\end{align*}
\end{corollary}

\noindent \begin{proof} By Lemma~\ref{lem:Dissipative}, specifically $\lVert\Phi_{ht}(x)\rVert\leq \lVert x\rVert e^{-\frac{1}{2}\nu\alpha t_0}+(\nu\alpha)^{-1}\lVert F\rVert$, we have
\begin{align*}
	\left(P_h\lVert\cdot\rVert\right)(x) &= \int_{\mathbb{R}^{n+1}_+} \lVert\Phi_{ht}(x)\rVert e^{-\sum t_k} dt
		\leq \frac{1}{1+\frac{1}{2}\nu\alpha h}\lVert x\rVert+\frac{1}{\nu\alpha}\lVert F\rVert
\end{align*}
for any $x$. The result follows with $K=(\nu\alpha)^{-1}\lVert F\rVert$ and $\gamma=(1+\tfrac{1}{2}\nu\alpha h)^{-1}$. \end{proof}


\subsection{Ergodicity}


We now present a variation of Theorem~\ref{thrm:Ergodicity}, namely Theorem~\ref{thrm:Ergodicity2}, which simplifies verification of ergodicity in the present setting. Recall from Sections~\ref{sec:Lorenz} and~\ref{sec:Euler} that one of the difficulties in verifying Theorem~\ref{thrm:Ergodicity} was proving controllability, \ie the existence of a distinguished point $x_*$ that could be reached by the splitting dynamics in finite time from any other point. With the addition of dissipation, the fixed point $\nu^{-1}\Lambda^{-1}F$ of $\dot{x}=V_0(x)$ is a natural candidate for $x_*$ and, as we will see, the fact that it is globally attracting obviates several technicalities associated with controllability in the conservative cases discussed above.

\begin{theorem}\label{thrm:Ergodicity2}
Suppose Assumption~\ref{assump:VectorFields} holds and set
$x_*=\nu^{-1}\Lambda^{-1}F$. If there exist $m \geq 0 $ and $t$ in $\mathbb{R}^{mn}_+$ such that the Lie bracket condition holds at $\widetilde{x}\coloneqq\Phi^m_{ht}(x_*)$, then $P_h$ has a unique invariant measure $\mu$ for all $h>0$. Furthermore, there exist $C>0$ and $\gamma$ in $(0,1)$ such that for all $x$ in $\mathbb{R}^d$,
\begin{align}\label{eq:ExpConvergence}
	\lVert P_h^m(x,\cdot)-\mu \rVert &\leq C\gamma^m
\end{align}
where $\lVert\cdot\rVert$ is the norm on probability measures induced by the weighted supremum norm $\lVert f\rVert\coloneqq \sup_x \lvert f(x)\rvert/(1+\lVert x\rVert)$ on bounded measurable functions $f:\mathbb{R}^d\to \mathbb{R}$.
\end{theorem}

\noindent The proof of Theorem~\ref{thrm:Ergodicity2} uses the following lemmas. The first, due to Krylov-Bogolubov, is a standard result from the theory of Markov processes \cite{HairerConvergence}. The second, which follows from \Cref{lem4.1} and \Cref{thrm:submersion}, is from \cite[Theorem~4.4]{Benaim}. For the statement of Lemma~\ref{lem:Krylov}, recall a transition kernel $P$ on $\mathbb{R}^d$ is \textit{Feller} if $Pf$ is continuous whenever $f:\mathbb{R}^d\to \mathbb{R}$ is continuous and bounded. Also, a sequence of probability measures $\{\mu_m\}$ on $\mathbb{R}^d$ is \textit{tight} if for every $\epsilon>0$ there exists a compact subset $K$ of $\mathbb{R}^d$ such that $\mu_m(K)\geq 1-\epsilon$ for all $m$.

\begin{lemma}\label{lem:Krylov}
Let $P$ be a Feller probability transition kernel on $\mathbb{R}^d$. If there exists $x$ in $\mathbb{R}^d$ such that $\{P^m(x,\cdot)\}_{m=0}^\infty$ is tight, then $P$ has an invariant probability measure.
\end{lemma}

\begin{lemma}\label{lem:WeakBracket}
Suppose $\Phi^m_{ht}(x)=\widetilde{x}$ and the Lie bracket condition holds at $\widetilde{x}$. Then there exists a $c>0$, an $\widetilde{m}$, and neighborhoods $U_x$ of $x$ and $\widetilde{U}$ of $\widetilde{x}$ such that for all $y$ in $U_x$ and $B$ in $\mathcal{B}(\mathcal{X})$,
\begin{align*}
	P_h^{\widetilde{m}}(y,B) &\geq c\lambda\left(B\cap \widetilde{U}\right).
\end{align*}
\end{lemma}

\noindent\jcm{The following proof is another instance of the rather classical idea, dating
at least back to the split chains of Nummelin \cite{Nummelin94} and
work of Meyn and Tweedie \cite{MeynTweedie}, that the existence of a  globally
accessible point at which the dynamics is continuous in the right sense implies the
transition densities converge to a unique equilibrium measure. If the
return to the globally
accessible point has finite expectation, then mixing is
exponential. The same basic structure of the SDE version of our system
was leveraged in \cite{EMattingly} to prove exponential mixing (see
also \cite{MattinglyStuart}). In
the closely related PDMP setting, analogous results are found in
\cite{Li_Liu_Cui_2017} in a specific example and
\cite{Benaim_Hurth_Strickler_2018} in a more general context.}

\begin{proof}[Proof of \Cref{thrm:Ergodicity2}]
We first prove existence. Continuity of $\Phi_{ht}$ immediately implies $P_h$ is Feller. Furthermore, Lemma~\ref{lem:Dissipative} implies that random splitting starting from any $x$ is constrained to lie in a compact subset of $\mathbb{R}^d$, namely the closed ball of radius $\lVert x\rVert^2+(\nu\alpha)^{-2}\lVert F\rVert^2$ centered at the origin. Thus, for any $x$, the sequence $\{P_h^m(x,\cdot)\}_{m=0}^\infty$ is tight and existence follows from Lemma~\ref{lem:Krylov}.

Next we prove uniqueness. The hypothesis and Lemma~\ref{lem:WeakBracket} together imply the existence of $c>0$, $\widetilde{m}$, and neighborhoods $U_*$ of $x_*$ and $\widetilde{U}$ of $\widetilde{x}$ such that
\begin{align}\label{eq:InequalityLem7.1}
	P_h^{\widetilde{m}}(x,B) &\geq c\lambda\left(B\cap \widetilde{U}\right)
\end{align}
for all $x\in U_*$ and Borel sets $B$. Also, positive-definiteness of $\Lambda$ implies
\begin{align*}
	\lVert\phi^{(0)}_t(x) - x_*\rVert &\leq e^{-\alpha t} \lVert x-x_*\rVert
\end{align*}
for any $x\in\mathbb{R}^d$ and $t\geq 0$. In particular, for any open ball $B_r$ of radius $r$ centered at the origin, there exists $T_0>0$ such that $\varphi^{(0)}_{ht}(B_r)$ is properly contained in $U_*$ whenever $ht>T_0$. And since $\varphi^{(0)}_{ht}(B_r)$ is properly contained in $U_*$ and the $\varphi^{(k)}$ are continuous, there exist $T_k>0$ such that $\Phi_{ht}=\varphi^{(n)}_{ht_n}\circ\cdots\circ\varphi^{(0)}_{ht_0}(x)\in U_*$ for all $x\in B_r$ and $ht_k\in (0,T_k)$. So, for any $x\in B_r$,
\begin{align*}
	P_h(x, U_*) &\geq \int_0^{T_n}\cdots\int_0^{T_1}\int_{T_0}^\infty\1_{U_*}\left(\Phi_{ht}(x)\right) e^{-\sum t_k} dt
		= \frac{1}{T_0}\prod_{k=1}^n \left(1-e^{-T_k}\right)
		> 0
\end{align*}
and hence $\inf _{x\in B_r} P_h(x, U_*)>0$.

As in the proof of Theorem~\ref{thrm:Ergodicity}, suppose toward a contradiction that $\mu_1$ and $\mu_2$ are distinct $P_h$-ergodic probability measures and that $A_1$ and $A_2$ are disjoint measurable sets partitioning $\mathbb{R}^d$ with $\mu_i(B)=\mu_i(B\cap A_i)$ for all Borel sets $B$. Fix $x_i$ in the support of $\mu_i$, let $r$ be sufficiently large that $x_1,x_2\in B_r$, and set $\kappa\coloneqq \inf_{x\in B_r} P_h(x,U_*)>0$. Then by~\eqref{eq:InequalityLem7.1} for any Borel set $B$,
\begin{equation}\label{eq:Minorization}
\begin{aligned}
	\mu_i(B) &= \mu_i P_h^{\widetilde{m}+1}(B)
		= \int_{\mathbb{R}^d}\int_{\mathbb{R}^d} P_h^{\widetilde{m}}(y,B)P_h(x,dy)\mu_i(dx) \\
		&\geq \int_{B_r}\int_{U_*} P_h^{\widetilde{m}}(y,B)P_h(x,dy)\mu_i(dx)
		\geq \kappa c\lambda\left(B\cap \widetilde{U}\right)\mu_i\left(B_r\right).
\end{aligned}
\end{equation}
In particular, $\mu_i(B)=0$ implies $\lambda(B\cap \widetilde{U})=0$ since $c$, $\kappa$, and $\mu_i(B_r)$ are all strictly positive (the latter because $B_r$ is an open set containing both $x_1$ and $x_2$ which were chosen to be in the supports of $\mu_1$ and $\mu_2$, respectively). But $\mu_1(A_2\cap \widetilde{U})=\mu_2(A_1\cap \widetilde{U})=0$ and so we obtain the contradiction
\begin{align*}
	0 &< \lambda\left(\widetilde{U}\right)
		= \lambda\left(A_1\cap \widetilde{U}\right)+\lambda\left(A_2\cap \widetilde{U}\right)
		= 0,
\end{align*}
which concludes the proof of uniqueness.

Finally, for the exponential convergence statement~\eqref{eq:ExpConvergence}, we have from~\eqref{eq:Minorization} that for any $r>0$,
\begin{align*}
	\inf_{x\in B_r} P_h^{\widetilde{m}+1}(x,B) &\geq \kappa c\lambda\left(B\cap\widetilde{U}\right)
\end{align*}
for all Borel sets $B$. That is, the transition probabilities $P_h^{\widetilde{m}+1}(x,\cdot)$ are minorized uniformly over $B_r$ by the probability measure $\widetilde{\lambda}\coloneqq \lambda(\widetilde{U})^{-1}\lambda(\cdot\cap\widetilde{U})$. Exponential convergence then follows from Corollary~\ref{cor:LyapunovFunction} upon taking $r>2K/(1-\gamma)$. See for example Theorem $1.2$ in \cite{yetMH}. \end{proof}

\begin{corollary}\label{cor:Lorenz96}
Consider the random splitting of Lorenz-96 associated
to the vector fields $\{V_k\}_{k=0}^n$, where $V_0(x)=-\nu x+F$ and
$\{V_k\}_{k=1}^n$ are the splitting vector fields of conservative
Lorenz-96 from Section~\ref{sec:Lorenz}. If $x_*\coloneqq -\nu F$ is
not a fixed point of conservative Lorenz-96, \ie $\nu^2\sum_{k=1}^n
(F_k^2+F_{k+1}^2)F_{k-1}^2\neq 0$, then the random splitting has a
unique\jcm{, and hence ergodic,} invariant measure on $\mathbb R^n$ and the dynamics converge to this measure at an exponential rate in the sense of~\eqref{eq:ExpConvergence}.
\end{corollary}

\noindent\begin{proof} The determinant of the $n$-by-$n$ matrix
\begin{align*}
	\begin{pmatrix}
	 	\vline & \vline & 	& \vline \\
	 	V_0(x) & V_1(x) & \cdots & V_{n-1}(x) \\
	 	\vline & \vline & & \vline
	 \end{pmatrix} &=
	\begin{pmatrix}
	 	-\nu x_1+F_1 & \phantom{-}x_2x_n & \phantom{-}0 & \hdots & \phantom{-}0 \\
	 	-\nu x_2+F_2 & -x_1x_n & \phantom{-}x_3x_1 & \hdots & \phantom{-}0 \\
	 	 -\nu x_3+F_3 & \phantom{-}0 & -x_2x_1 & \cdots & \phantom{-}\vdots \\
	 	\phantom{-}\vdots & \phantom{-}\vdots & \phantom{-}\vdots & \ddots & \phantom{-}x_nx_{n-2} \\
	 	-\nu x_n+F_n &\phantom{-}0 & \phantom{-}0 &  & -x_{n-1}x_{n-2}
	 \end{pmatrix}
\end{align*}
is
\begin{align*}
	x_1x_{n-1}x_n\left(\prod_{x=2}^{n-2}x_k^2\right)\left(\nu^2\lVert x\rVert^2-\langle F,x\rangle\right).
\end{align*}
So the $\{V_k\}_{k=0}^n$ span $\mathbb{R}^n$ at every $x$ with nonzero coordinates and satisfying $\nu^2\lVert x\rVert^2\neq \langle F,x\rangle$. In particular, since $x_*$ is not a fixed point of conservative Lorenz-96, we showed in the proof of Proposition~\ref{prop5.1} that $x_*$ can be moved via the splitting dynamics to some $\widetilde{x}$ with nonzero coordinates. Finally, by rotating slightly more on the last step if necessary, we can also guarantee $\nu^2\lVert \widetilde{x}\rVert^2\neq \langle F,\widetilde{x}\rangle$. Thus the Lie bracket condition holds at $\widetilde{x}$ and the result follows by Theorem~\ref{thrm:Ergodicity2}. \end{proof}

\begin{corollary}\label{cor:2DNS}
Fix $N\geq 2$ and set $n=4N(N+1)$. Consider the random splitting of
the $\Nth$ Galerkin approximation of 2D Navier-Stokes associated to
$\{V_k\}_{k=0}^n$, where $V_0(x)=-\nu\Lambda x+F$ with $\Lambda$ the
$n$-by-$n$ diagonal matrix corresponding to the Laplacian discussed at
the beginning of this section, and $\{V_k\}_{k=1}^n$ the splitting
vector fields of 2D Euler from Section~\ref{sec:Euler}. If $F$ is \om{nondegenerate in the sense of
  \Cref{def:nondegenerate}}, then the random splitting has a unique,
and hence ergodic, invariant measure and the dynamics converge to this measure at an exponential rate in the sense of~\eqref{eq:ExpConvergence}.
\end{corollary}

\noindent\begin{proof}
Recall in this case $V_0(x)=-\nu\Lambda x+F$ where $\Lambda$ is the diagonal matrix with diagonal entry $\lvert k\rvert^2$ in the slots corresponding to the coordinates $a_k$ and $b_k$. Fix $j,k,\ell\in\mathbb{Z}^2_N$ with $j+k-\ell=0$ and let $W$ be one of the vector fields $V_{a_ja_ka_\ell}$, $V_{a_jb_kb_\ell}$, $V_{b_ja_kb_\ell}$, or $V_{b_jb_ka_\ell}$. Letting e.g. $(x_j,x_k,x_\ell)=(a_j,a_k,a_\ell)$ when $W=V_{a_ja_ka_\ell}$ and similarly for the other cases, direct computation yields
\begin{equation}\label{eq:V0Brackets}
\begin{aligned}
	&[V_0,W]_j(x) = C_{k\ell}\left(F_kx_\ell+F_\ell x_k+\nu(\lvert j\rvert^2-\lvert k\rvert^2-\lvert \ell\rvert^2)x_kx_\ell\right), \\
	&[V_0,W]_k(x) = C_{j\ell}\left(F_jx_\ell+F_\ell x_j+\nu(\lvert k\rvert^2-\lvert j\rvert^2-\lvert \ell\rvert^2)x_jx_\ell\right), \\
	&[V_0,W]_\ell(x) = -C_{jk}\left(F_jx_k+F_k x_j+\nu(\lvert \ell\rvert^2-\lvert j\rvert^2-\lvert k\rvert^2)x_jx_k\right),
\end{aligned}
\end{equation}
where $[V_0,W]_j(x)$ is the component of $[V_0,W]$ corresponding to the component $x_j$ of $x$, and similarly for $[V_0,W]_k$ and $[V_0,W]_\ell$. As in the 2D Euler case, Gaussian elimination shows that the $6$-by-$6$ matrix (see~\eqref{100} for an explicit form of the middle $4$ columns)
\begin{align}\label{1000}
	\begin{pmatrix}
			\vline & \vline & \vline & \vline & \vline & \vline \\
			V_0 & V_{a_ja_ka_\ell} & V_{a_jb_kb_\ell} & V_{b_ja_kb_\ell} & V_{b_jb_ka_\ell} & [V_0,W] \\
			\vline & \vline & \vline & \vline & \vline & \vline
		\end{pmatrix}
\end{align}
is rank $6$ at every \om{generic\footnote{Recall a \textit{generic point} is one with all coordinates nonzero; see \Cref{def:generic}.}} point $q$ in $\mathbb{R}^n$. Thus $V_0$ and $[V_0,W]$ add two new directions to the splitting vector fields of 2D Euler and by an entirely similar argument to the spanning argument in Section~\ref{sec:EulerSpanning} we have that the Lie bracket condition holds at every such $q$. Furthermore, since $F$ \om{is nondegenerate} the controllability argument of Section~\ref{sec:EulerControl} implies $x_*$ can be evolved via the split dynamics to a generic point. The result then follows by Theorem~\ref{thrm:Ergodicity2}.
\end{proof}

\begin{remark}
A very similar argument to the one above proves unique ergodicity for Ekman damping as well, \ie when $\Lambda$ is the identity matrix on $\mathbb{R}^n$. In this case \eqref{eq:V0Brackets} becomes
\begin{align*}
	&[V_0,W]_j(x) = C_{k\ell}\left(F_kx_\ell+F_\ell x_k-\nu x_kx_\ell\right), \\
	&[V_0,W]_k(x) = C_{j\ell}\left(F_jx_\ell+F_\ell x_j-\nu x_jx_\ell\right), \\
	&[V_0,W]_\ell(x) = -C_{jk}\left(F_jx_k+F_k x_j-\nu x_jx_k\right),
\end{align*}
and the rest of the argument goes through unchanged.
\end{remark}

\noindent\textbf{Acknowledgments:} All authors thank the National Science Foundation  grant NSF-DMS-1613337 for partial support during this project. AA also gratefully acknowledges the partial support of NSF-CCF-1934964, and OM also thanks NSF-DMS-2038056 for partial support during this project. JCM thanks David
Herzog  and Brendan Williamson for discussions at the start of these
investigations. JCM thanks the hospitality and support of the Institute
for Advanced Study, where this manuscript was completed. We also thank
the referees for their insightful comments which improved both the form and the content of this paper.\\[-4pt]

\noindent \textbf{Data availability statement:} Data sharing not applicable to this article as no datasets were generated or analysed during the current study.



\bibliographystyle{abbrv}
\bibliography{refs}


\appendix



\section{Convergence lemmas}



\subsection{Semigroups, norms, and bounds}


In this subsection we elaborate on the semigroup framework of Section~\ref{sec:Convergence}. The notation and results are used extensively in the proofs of Lemmas~\ref{lem3.1} and~\ref{lem3.2}, which are given in subsections \ref{subsec:ProofLem3.1} and \ref{subsec:ProofLem3.2}, respectively.

Fix a $\mathcal{V}$-orbit $\mathcal{X}$. The $\mathcal{C}^2$ assumption implies the $V_k$, which act on functions $f$ via $V_kf(x)=Df(x)V_k(x)$, are linear operators from $\mathcal{C}^2(\mathcal{X})$ to $\mathcal{C}^1(\mathcal{X})$ and from $\mathcal{C}^1(\mathcal{X})$ to $\mathcal{C}(\mathcal{X})$. It also implies the semigroups $\{S_t\}_{t\geq 0}$ and $\{\widetilde{S}^{(k)}_t\}_{t\geq 0}$ defined in~\eqref{3.20} and~\eqref{3.21} are linear operators on $\mathcal{C}^k(\mathcal{X})$ for $k\leq 2$. Our aim now is to obtain bounds on norms of compositions of these random semigroups. For $i\leq j$ define $\Phi^{(i,j)}_{h\tau}\coloneqq \varphi^{(j)}_{h\tau_j}\circ\cdots\circ\varphi^{(i)}_{h\tau_i}$ and $\widetilde{S}^{(i,j)}_{h\tau}\coloneqq \widetilde{S}^{(i)}_{h\tau}\cdots\widetilde{S}^{(j)}_{h\tau}$. Note $\widetilde{S}^{(i,j)}_{h\tau}$ acts on functions $f$ via
\begin{align*}
	\widetilde{S}^{(i,j)}_{h\tau}f(x) &= f\left(\Phi^{(i,j)}_{h\tau}(x)\right)
		= f\left(\varphi^{(j)}_{h\tau_j}\circ\cdots\circ\varphi^{(i)}_{h\tau_i}(x)\right).
\end{align*}
So for any $f\in\mathcal{C}(\mathcal{X})$ with $\lVert f\rVert_\infty=1$, we have
\begin{align*}
	\lVert \widetilde{S}^{(i,j)}_{h\tau}f\rVert_\infty &= \lVert f(\Phi^{(i,j)}_{h\tau})\rVert_\infty
		= 1
\end{align*}
and hence $\lVert \widetilde{S}^m_{h\tau}\rVert_{0\to 0}=1$. Next, let $\varphi=\varphi^{(k)}$ for arbitrary $k$. Then
\begin{align*}
	\varphi_t(x) &= x + \int_0^t V(\varphi_s(x)) ds
\end{align*}
and so
\begin{align*}
	D\varphi_t(x) &= I + \int_0^t DV(\varphi_s(x))D\varphi_s(x) ds
\end{align*}
and
\begin{align*}
	D^2\varphi_t(x) &= \int_0^t D^2V(\varphi_s(x))\left(D\varphi_s(x),D\varphi_s(x)\right)+DV(\varphi_s(x))D^2\varphi_s(x) ds.
\end{align*}
In particular, $\lVert D\varphi_t(x)\rVert \leq 1 + C_*\int_0^t \lVert D\varphi_s(x)\rVert ds$ for all $x$ in $\mathcal{X}$ and Gr\"onwall's inequality implies
\begin{align}\label{eq:Gron1}
	\sup_{x\in\mathcal{X}}\lVert D\varphi_t(x)\rVert &\leq e^{C_*t},
\end{align}
where here and throughout $C_*$ is the constant from~\eqref{eq:Bounded} corresponding to $\mathcal{X}$. Similarly, since $\lVert D^2V\left(D\varphi,D\varphi\right)\rVert\leq \lVert D^2V\rVert\lVert D\varphi\rVert^2\leq C_*\lVert D\varphi\rVert^2$,
\begin{align*}
	\lVert D^2\varphi_t(x)\rVert \leq C_*\int_0^t \lVert D\varphi_s(x)\rVert^2+\lVert D^2\varphi_s(x)\rVert ds
		\leq C_*te^{2C_*t}+C_*\int_0^t\lVert D^2\varphi_s(x)\rVert ds
\end{align*}
and Gr\"onwall implies
\begin{align}\label{eq:Gron2}
	\sup_{x\in\mathcal{X}}\lVert D^2\varphi_t(x)\rVert \leq C_*te^{3C_*t}.
\end{align}
Note~\eqref{eq:Gron1} and~\eqref{eq:Gron2} hold uniformly over all $\varphi^{(k)}$. Thus, for $f\in C^1(\mathcal{X})$ with $\lVert f\rVert_1=1$,
\begin{align*}
	\left\lVert D\left(\widetilde{S}^{(i,j)}_{h\tau}f\right)\right\rVert &= \left\lVert Df\left(\Phi^{(i,j)}_{h\tau}\right)D\Phi^{(i,j)}_{h\tau}\right\rVert
		\leq \prod_{k=i}^{j} \lVert D\varphi^{(k)}_{h\tau_k}\rVert
		\leq e^{C_*h\sum_{k=i}^{j}\tau_k},
\end{align*}
where the first inequality follows from submultiplicity and the second from~\eqref{eq:Gron1}. Similarly,
\begin{align*}
	D^2\Phi^{(i,j)}_{h\tau} &= \sum_{k=i}^j D\varphi^{(j)}_{h\tau_j}\cdots D\varphi^{(k+1)}_{h\tau_{k+1}}D^2\varphi^{(k)}_{h\tau_k}\left(D\Phi^{(i,k-1)}_{h\tau}, D\Phi^{(i,k-1)}_{h\tau}\right)
\end{align*}
together with~\eqref{eq:Gron1} and~\eqref{eq:Gron2} gives
\begin{align*}
	\left\lVert D^2\Phi^{(i,j)}_{h\tau}\right\rVert &\leq \sum_{k=i}^j \left\lVert D\varphi^{(j)}_{h\tau_j}\right\rVert\cdots \left\lVert D\varphi^{(k+1)}_{h\tau_{k+1}}\right\rVert\left\lVert D^2\varphi^{(k)}\right\rVert\left\lVert D\Phi^{(i,k-1)}_{h\tau}\right\rVert^2 \\
		&\leq C_*\sum_{k=i}^j h\tau_k e^{C_*h\sum_{k+1}^j\tau_\ell}e^{3C_*h\tau_k}e^{2C_*h\sum_1^{k-1}\tau_\ell}
		\leq C_*he^{3C_*h\sum_{k=i}^j\tau_k}\sum_{k=i}^j \tau_k.
\end{align*}
Therefore
\begin{align*}
	\left\lVert D^2\left(\widetilde{S}^{(i,j)}_{h\tau}f\right)\right\rVert &= \left\lVert D^2f\left(\Phi^{(i,j)}_{h\tau}\right)\left(D\Phi^{(i,j)}_{h\tau},D\Phi^{(i,j)}_{h\tau}\right)+Df\left(\Phi^{(i,j)}_{h\tau}\right)D^2\Phi^{(i,j)}_{h\tau}\right\rVert \\
		&\leq \left\lVert D\Phi^{(i,j)}_{h\tau}\right\rVert^2 + \left\lVert D^2\Phi^{(i,j)}_{h\tau}\right\rVert
		\leq e^{2C_*h\sum_{k=i}^{j}\tau_k} + \left\lVert D^2\Phi^{(i,j)}_{h\tau}\right\rVert \\
		&\leq e^{2C_*h\sum_{k=i}^{j}\tau_k} + C_*he^{3C_*h\sum_{k=i}^j\tau_k}\sum_{k=i}^j \tau_k \\
		&\leq \left(1+C_*h\sum_{k=i}^j\tau_k\right)e^{3C_*h\sum_{k=i}^j \tau_k}.
\end{align*}
The above computations prove
\begin{lemma}\label{lem:Bounds}
For any $h>0$ and $i\leq j$, we have $\lVert \widetilde{S}^{(i,j)}_{h\tau}\rVert_{0\to 0}=1$ as well as
\begin{align*}
	\left\lVert \widetilde{S}^{(i,j)}_{h\tau}\right\rVert_{1\to 1}\leq e^{C_*h\sum_{k=i}^{j}\tau_k}
	\quad\text{and}\quad
	\left\lVert \widetilde{S}^{(i,j)}_{h\tau}\right\rVert_{2\to 2}\leq \left(1+C_*h\sum_{k=i}^j\tau_k\right)e^{3C_*h\sum_{k=i}^j \tau_k}.
\end{align*}
In particular, $\left\lVert \widetilde{S}^{(i,j)}_{h\tau}\right\rVert_{\ell\to \ell}\leq \left(1+C_*h\sum_{k=i}^j\tau_k\right)e^{3C_*h\sum_{k=i}^j \tau_k}$ for all $\ell\leq 2$.
\end{lemma}

\noindent Note that under the $\mathcal{C}^2$ assumption $\widetilde{S}^{(i,j)}_{h\tau}$ can also be regarded as a linear operator from $\mathcal{C}^2(
\mathcal{X})$ to $\mathcal{C}^1(
\mathcal{X})$. So since $\{f\in\mathcal{C}^2(\mathcal{X}) : \lVert f\rVert_2=1\}$ is a subset of $\{f\in\mathcal{C}^1(\mathcal{X}) : \lVert f\rVert_1=1\}$, we have
\begin{align}\label{eq:2to1}
	\left\lVert\widetilde{S}^{(i,j)}_{h\tau}\right\rVert_{2\to 1} &= \sup_{\lVert f\rVert_2=1}\left\lVert\widetilde{S}^{(i,j)}_{h\tau}f\right\rVert_1
		\leq \sup_{\lVert f\rVert_1=1}\left\lVert\widetilde{S}^{(i,j)}_{h\tau}f\right\rVert_1
		= \left\lVert\widetilde{S}^{(i,j)}_{h\tau}\right\rVert_{1\to 1}
		\leq e^{C_*h\sum_{k=i}^{j}\tau_k}.
\end{align}
We also have the following corollary of Lemma~\ref{lem:Bounds}.

\begin{cor}\label{cor:PolyBounds}
Fix $i\leq j$ and set $m\coloneqq j-i+1$. For all $\ell\leq 2$ and polynomial $p:\mathbb{R}^m_+\to\mathbb{R}$ there exists $h_*>0$ such that for all $h<h_*$,
\begin{align}
	\mathbb{E}\lVert p(\tau_i,\dots,\tau_j)\widetilde{S}^{(i,j)}_{h\tau}\rVert_{k\to k} &< \infty.
\end{align}
\end{cor}

\noindent\begin{proof} Writing $t=(t_i,\dots,t_j)$ and $dt=dt_i\cdots dt_j$, we have
\begin{align*}
	\mathbb{E}\lVert p(\tau_i,\dots,\tau_j)\widetilde{S}^{(i,j)}_{h\tau}\rVert_{\ell\to \ell} &= \int_{\mathbb{R}^m_+} \lvert p(t)\rvert\left\lVert \widetilde{S}^{(i,j)}_{ht}\right\rVert_{\ell\to \ell} e^{-\sum t_k} dt \\
		&\leq \int_{\mathbb{R}^m_+} \lvert p(t)\rvert\left(1+C_*h\sum_{k=i}^j t_k\right)e^{(3C_*h-1)\sum_{k=i}^j t_k} dt
\end{align*}
which is finite for all $h < h_*\coloneqq (3C_*)^{-1}$. \end{proof}


\subsection{Proof of Lemma~\ref{lem3.1}}
\label{subsec:ProofLem3.1}


We highlight the steps of the proof with italicized font. \\

\noindent\textit{Variation of constants.} We begin by differentiating $\widetilde{S}_{h\tau}$ in $h$:
\begin{align*}
	\partial_h\widetilde{S}_{h\tau} &= \sum_{k=1}^n \tau_k e^{h\tau_1}\cdots e^{h\tau_{k-1}}V_k e^{h\tau_k}\cdots e^{h\tau_n}
		= \sum_{k=1}^n \tau_k \widetilde{S}^{(1,k-1)}_{h\tau}V_k\widetilde{S}^{(k,n)}_{h\tau}.
\end{align*}
Next, commute $\widetilde{S}^{(1,k-1)}_{h\tau}$ and $V_k$ via $[\widetilde{S}^{(1,k-1)}_{h\tau}, V_k]\coloneqq\widetilde{S}^{(1,k-1)}_{h\tau}V_k-V_k\widetilde{S}^{(1,k-1)}_{h\tau}$ to get
\begin{align*}
		\partial_h\widetilde{S}_{h\tau} &= \sum_{k=1}^n \tau_kV_k\widetilde{S}_{h\tau}+\sum_{k=1}^n \tau_k[\widetilde{S}^{(1,k-1)}_{h\tau}, V_k]\widetilde{S}^{(k,n)}_{h\tau}
		= V\widetilde{S}_{h\tau}+(V_\tau-V)\widetilde{S}_{h\tau}+E_{h\tau}
\end{align*}
where $V_\tau\coloneqq \sum_{k=1}^n \tau_kV_k$ and $E_{h\tau}\coloneqq \sum_{k=1}^n \tau_k[\widetilde{S}^{(1,k-1)}_{h\tau}, V_k]\widetilde{S}^{(k,n)}_{h\tau}$. So, by variation of constants,
\begin{align}\label{A.1}
	\widetilde{S}_{h\tau}-S_h &= \int_0^h S_{h-r}(V_\tau-V)\widetilde{S}_{r\tau} dr+\int_0^hS_{h-r}E_{r\tau} dr.
\end{align}
Call $S_{h-r}(V_\tau-V)\widetilde{S}_{r\tau}$ \textit{error term 1} and $S_{h-r}E_{r\tau}$ \textit{error term 2}. These terms will be treated separately in what follows. First however, we invoke variation of constants again to get an expression for $[\widetilde{S}^{(1,k-1)}_{r\tau}, V_k]$ that will be used to control error term 2. Differentiating in $r$ gives
\begin{align*}
	\partial_r[\widetilde{S}^{(1,k-1)}_{r\tau}, V_k] &= \sum_{j=1}^{k-1}\tau_j[\widetilde{S}_{r\tau}^{(1,j-1)}V_j\widetilde{S}_{r\tau}^{(j,k-1)},V_k] \\
		&= \sum_{j=1}^{k-1}\tau_j\bigg([V_j\widetilde{S}_{r\tau}^{(1,k-1)},V_k]+\big[[\widetilde{S}_{r\tau}^{(1,j-1)},V_j]\widetilde{S}_{r\tau}^{(j,k-1)},V_k\big]\bigg) \\
		&= \sum_{j=1}^{k-1} \tau_jV_j[\widetilde{S}_{r\tau}^{(1,k-1)},V_k] + \sum_{j=1}^{k-1} \tau_j\bigg([V_j,V_k]\widetilde{S}_{r\tau}^{(1,k-1)}+\big[[\widetilde{S}_{r\tau}^{(1,j-1)}, V_j]\widetilde{S}_{r\tau}^{(j,k-1)},V_k\big]\bigg).
\end{align*}
The second equality follows from commuting $\widetilde{S}^{(1,j-1)}_{h\tau}$ and $V_j$ as before, and the third follows from the identity $[XY,Z]=X[Y,Z]+[X,Z]Y$. So, by variation of constants,
\begin{equation}\label{eq:VarConst}
\begin{aligned}
[\widetilde{S}_{r\tau}^{(1,k-1)},V_k] &= \sum_{j=1}^{k-1}\int_0^r\tau_j e^{(r-s)\sum_{j=1}^{k-1}\tau_jV_j}[V_j,V_k]\widetilde{S}^{(1,k-1)}_{s\tau}ds \\
		&\qquad+\sum_{j=1}^{k-1}\int_0^r \tau_j e^{(r-s)\sum_{j=1}^{k-1}\tau_jV_j}\big[[\widetilde{S}^{(1,j-1)}_{s\tau},V_j]\widetilde{S}^{(j,k-1)}_{s\tau},V_k\big]ds.
\end{aligned}
\end{equation}
Note $\lVert e^{(r-s)\sum_{j=1}^{k-1}\tau_jV_j}\rVert_{0\to 0}=1$. So, by Corollary~\ref{cor:PolyBounds} the integrands above satisfy
\begin{align*}
	\mathbb{E}\lVert \tau_j e^{(r-s)\sum_{j=1}^{k-1}\tau_jV_j}[V_j,V_k]\widetilde{S}^{(1,k-1)}_{s\tau}\rVert_{2\to 0} &\leq \lVert [V_j,V_k]\rVert_{2\to 0}\mathbb{E}\lVert \tau_j\widetilde{S}^{(1,k-1)}_{s\tau}\rVert_{2\to 2}
		< C
\end{align*}
and
\begin{align*}
	\mathbb{E}\big\lVert \tau_j e^{(r-s)\sum_{j=1}^{k-1}\tau_jV_j}\big[[\widetilde{S}^{(1,j-1)}_{s\tau},V_j]\widetilde{S}^{(j,k-1)}_{s\tau},V_k\big]\big\rVert_{2\to 0} &\leq \mathbb{E}\big\lVert \tau_j \big[[\widetilde{S}^{(1,j-1)}_{s\tau},V_j]\widetilde{S}^{(j,k-1)}_{s\tau},V_k\big]\big\rVert_{2\to 0}
		< C
\end{align*}
for some $C$. Therefore
\begin{align}\label{A.2}
	\mathbb{E}\lVert [\widetilde{S}_{r\tau}^{(1,k-1)},V_k]\rVert_{2\to 0} &\leq 2\sum_{j=1}^{k-1}\int_0^r C ds
		\leq Cr
\end{align}
for some new constant $C$ (we will often absorb arbitrary constants into existing ones). \\

\noindent\textit{Error term 1.} Rewrite error term 1 as
\begin{equation}\label{A.3}
\begin{aligned}
S_{h-r}(V_\tau-V)\widetilde{S}_{r\tau} &= \sum_{k=1}^n (\tau_k-1)S_{h-r}V_k\widetilde{S}_{r\tau} \\
		&= \sum_{k=1}^n (\tau_k-1)S_{h-r}V_k\widetilde{S}^{(1,k-1)}_{r\tau}\widetilde{S}^{(k+1,n)}_{r\tau} \\
		&\qquad+\sum_{k=1}^n (\tau_k-1)S_{h-r}V_k\widetilde{S}^{(1,k-1)}_{r\tau}(e^{r\tau_kV_k}-I)\widetilde{S}^{(k+1,n)}_{r\tau} \\
		&\eqqcolon \mathcal{A}_1+\mathcal{A}_2
\end{aligned}
\end{equation}
where $\mathcal{A}_1$ and $\mathcal{A}_2$ are the first and second sums in the preceding expression. The second equality is obtained by adding and subtracting the identity $I$ as follows:
\begin{align*}
	\widetilde{S}_{r\tau} &= \widetilde{S}^{(1,k-1)}_{r\tau}\big(e^{r\tau_kV_k}-I+I\big)\widetilde{S}^{(k+1,n)}_{r\tau}
		= \widetilde{S}^{(1,k-1)}_{r\tau}\widetilde{S}^{(k+1,n)}_{r\tau}+\widetilde{S}^{(1,k-1)}_{r\tau}(e^{r\tau_kV_k}-I)\widetilde{S}^{(k+1,n)}_{r\tau}.
\end{align*}
Notice $\widetilde{S}^{(1,k-1)}_{r\tau}\widetilde{S}^{(k+1,n)}_{r\tau}$ does not depend on $\tau_k$. So, since the $\tau_i$ are independent with mean $1$,
\begin{align}\label{A.4}
	\mathbb{E}(\mathcal{A}_1) &= \sum_{k=1}^n S_{t-r}V_k\mathbb{E}(\tau_k-1)\mathbb{E}\big(\widetilde{S}^{(1,k-1)}_{r\tau}\widetilde{S}^{(k+1,n)}_{r\tau}\big)
		= 0.
\end{align}
For the second sum, Taylor expanding $r\mapsto e^{r\tau_kV_k}$ about $r=0$ with remainder gives
\begin{align*}
	e^{r\tau_kV_k}-I &= r\tau_kV_ke^{r_*\tau_kV_k}
\end{align*}
for some $r_*\in[0,r]$. Therefore
\begin{align*}
	\mathcal{A}_2 &= r\sum_{k=1}^n \tau_k(\tau_k-1)S_{h-r}V_k\widetilde{S}^{(1,k-1)}_{r\tau}V_ke^{r_*\tau_kV_k}\widetilde{S}^{(k+1,n)}_{r\tau}
\end{align*}
and by Lemma~\ref{lem:Bounds} and Corollary~\ref{cor:PolyBounds},
\begin{align}\label{A.5}
	\lVert\mathbb{E}(\mathcal{A}_2)\rVert_{2\to 0} &\leq Cr\sum_{k=1}^n \mathbb{E}\lVert\widetilde{S}^{(1,k-1)}_{r\tau}\rVert_{1\to 1}\mathbb{E}\lVert\tau_k(\tau_k-1)\widetilde{S}^{(k,n)}_{r\tau}\rVert_{2\to 2}
		\leq Cr
\end{align}
for some $C>0$. Combining Equations~\eqref{A.3},~\eqref{A.4}, and~\eqref{A.5} gives
\begin{align}\label{A.6}
	\lVert\mathbb{E}(S_{h-r}(V_\tau-V)\widetilde{S}_{r\tau})\rVert_{2\to 0} &\leq Cr.
\end{align}

\noindent\textit{Error term 2.} Recall error term 2 is $S_{h-r}E_{r\tau}\coloneqq \sum_{k=1}^n \tau_kS_{h-r}[\widetilde{S}^{(1,k-1)}_{r\tau}, V_k]\widetilde{S}^{(k,n)}_{r\tau}$. So, we have that
\begin{align*}
	\lVert S_{h-r}E_{r\tau}\rVert_{2\to 0} &\leq \sum_{k=1}^n \tau_k\lVert S_{h-r}\rVert_{0\to 0}\lVert [\widetilde{S}^{(1,k-1)}_{r\tau},V_k]\rVert_{2\to 0}\lVert \tau_k\widetilde{S}^{(k,n)}_{r\tau}\rVert_{2\to 2}\,.
\end{align*}
Note $[\widetilde{S}^{(1,k-1)}_{r\tau},V_k]$ is independent of $\tau_k$. So, by~\eqref{A.2} Corollary~\ref{cor:PolyBounds},
\begin{align}\label{A.7}
	\lVert \mathbb{E}( S_{h-r}E_{r\tau})\rVert_{2\to 0} &\leq Cr
\end{align}
for some $C>0$. \\

\noindent\textit{Final step.} Combining~\eqref{A.1},~\eqref{A.6}, and~\eqref{A.7} and absorbing constants into $C$, we have
\begin{align*}
	\lVert P_h-S_h\rVert_{2\to 0} &= \lVert \mathbb{E}(\widetilde{S}_{h\tau}-S_h)\rVert_{2\to 0} \\
		&\leq \int_0^h \lVert \mathbb{E}(S_{h-r}(V_\tau-V)\widetilde{S}_{r\tau})\rVert_{2\to 0} dr+\int_0^h\lVert \mathbb{E}(S_{h-r}E_{r\tau})\rVert_{2\to 0}  dr \\
		&\leq C\int_0^h r dr
		= \tfrac{1}{2}Ch^2. \QED
\end{align*}


\subsection{Concentration of the sum of exponential random variables}


The proof of Lemma~\ref{lem3.2} will itself use two lemmas.

\begin{lemma}\label{lemA1}
Let $\{\tau_k\}_{k=1}^\infty$ be iid exponential with mean $1$. For any $m\in\mathbb{N}$, $K>0$ and $\beta>1$,
\begin{align}\label{A.141}
	\mathbb{P}\left(\sum_{k=1}^m\tau_k>Km^\beta\right) &\leq 2^me^{-\frac{1}{2}Km^\beta}.
\end{align}
\end{lemma}

\noindent\begin{proof}
Note if $\tau\sim\text{Exp}(1)$ then $\mathbb{E}(e^{\tau/2})=2$. So, by Markov's inequality and independence,
\begin{align*}
	\mathbb{P}\left(\sum_{k=1}^m\tau_k>Km^\beta\right) &= \mathbb{P}\left(e^{\frac{1}{2}\sum_{k=1}^m\tau_k}>e^{\frac{1}{2}Km^\beta}\right)
		\leq e^{-\frac{1}{2}Km^\beta}\left(\mathbb{E}\left[e^{\frac{1}{2}\tau}\right]\right)^m
		= 2^me^{-\frac{1}{2}Km^\beta}. \qedhere
\end{align*}
\end{proof}

\begin{lemma}\label{lemA2}
Let $\{\tau_k\}_{k=1}^\infty$ be iid exponential with mean $1$. For any $m\in\mathbb{N}$ and $K\in (0,1)$,
\begin{align}
	\mathbb{P}\left(\bigg\lvert\sum_{k=1}^m \tau_k-1\bigg\rvert > Km\right) &< 2e^{-\frac{1}{2}K^2m}.
\end{align}
\end{lemma}

\noindent\begin{proof}
Fix $m$. For any $\gamma\in(0,1)$,
\begin{align*}
	\mathbb{P}\left(\bigg\lvert\sum_{k=1}^m \tau_k-1\bigg\rvert >
  Km\right) 
		&= \mathbb{P}\left(\sum_{k=1}^m \tau_k > (1+K)m\right)+\mathbb{P}\left(-\sum_{k=1}^m \tau_k > -(1-K)m\right) \\
		&= \mathbb{P}\left(e^{\gamma\sum_{k=1}^m \tau_k} > e^{(1+K)\gamma m}\right)+\mathbb{P}\left(e^{-\gamma\sum_{k=1}^m \tau_k} > e^{-(1-K)\gamma m}\right) \\
		&\leq e^{-(1+K)\gamma m}\left(\mathbb{E}\left[e^{\gamma \tau}\right]\right)^m + e^{(1-K)\gamma m}\left(\mathbb{E}\left[e^{-\gamma \tau}\right]\right)^m \\
		&= e^{-(1+K)\gamma m}\left(1-\gamma\right)^{-m} + e^{(1-K)\gamma m}\left(1+\gamma\right)^{-m} \\
		&= \exp\left(-\gamma m\left[1+K+\frac{\log(1-\gamma)}{\gamma}\right]\right)+\exp\left(\gamma m\left[1-K-\frac{\log(1+\gamma)}{\gamma}\right]\right).
\end{align*}
The inequality is Markov's inequality and the equality immediately after the inequality follows from independence together with $\mathbb{E}[\exp(\alpha \tau)]=(1-\alpha)^{-1}$ for any $\alpha\in (-1,1)$. The other steps are all algebraic manipulations. By Taylor's theorem with remainder there exists $\gamma_1\in(-\gamma,0)$ such that
\begin{align*}
	\frac{1}{\gamma}\log(1-\gamma) &= -1-\frac{\gamma}{2(1-\gamma_1)^2}
		> -1-\frac{\gamma}{2},
\end{align*}
where the inequality follows since $\gamma_1<0$. Therefore
\begin{align*}
	\exp\left(-\gamma m\left[1+K+\frac{\log(1-\gamma)}{\gamma}\right]\right) &\leq \exp\left(-\gamma m\left[K-\frac{\gamma}{2}\right]\right).
\end{align*}
Similarly,
\begin{align*}
	\exp\left(\gamma m\left[1-K-\frac{\log(1+\gamma)}{\gamma}\right]\right) &\leq \exp\left(-\gamma m\left[K-\frac{\gamma}{2}\right]\right).
\end{align*}
So combining with the first computation of this proof and taking $\gamma=K$ gives
\begin{align*}
	\mathbb{P}\left(\bigg\lvert\sum_{k=1}^m \tau_k-1\bigg\rvert > Km\right) &\leq 2\exp\left(-\gamma m\left[K-\frac{\gamma}{2}\right]\right)
		= 2e^{-\frac{1}{2}K^2m}. \qedhere
\end{align*}
\end{proof}


\subsection{Proof of Lemma~\ref{lem3.2}}
\label{subsec:ProofLem3.2}


Fix $t>0$. The argument is similar to that of Lemma~\ref{lem3.1}. \\

\noindent\textit{Variation of constants}. Fix $m\in\mathbb{N}$. Since $\widetilde{S}^m_{h\tau}=\exp(h\tau_1V_1)\cdots\exp(h\tau_{mn}V_{mn})$,
\begin{align*}
	\partial_h\widetilde{S}^m_{h\tau} &= \sum_{k=1}^{mn} \tau_k\widetilde{S}^{(1,k-1)}_{h\tau}V_k\widetilde{S}^{(k,mn)}_{h\tau}
		= \sum_{k=1}^{mn} \tau_kV_k\widetilde{S}^m_{h\tau}+\tau_k[\widetilde{S}^{(1,k-1)}{h\tau}, V_k]\widetilde{S}^{(k,mn)}_{h\tau} \\
		&= mV\widetilde{S}^m_{h\tau}+\sum_{k=1}^{mn}(\tau_k-1)V_k\widetilde{S}^m_{h\tau}+\sum_{k=1}^{mn}\tau_k[\widetilde{S}^{(1,k-1)}_{h\tau}, V_k]\widetilde{S}^{(k,mn)}_{h\tau},
\end{align*}
where the second equality is obtained by commuting $\widetilde{S}^{(1,k-1)}_{h\tau}$ and $V_k$, and the third by replacing $\tau_k$ with $\tau_k-1+1$. So, setting $E_{h\tau}^{(m)}\coloneqq \sum_{k=1}^{mn}\tau_k[\widetilde{S}^{(1,k-1)}_{h\tau}, V_k]\widetilde{S}^{(k,mn)}_{h\tau}$, variation of constants implies
\begin{align*}
	\widetilde{S}^m_{h\tau}-S_{hm} &= \int_0^h S_{m(h-r)}\left(\sum_{k=1}^{mn}(\tau_k-1)V_k\right)\widetilde{S}^m_{r\tau} dr + \int_0^h S_{m(h-r)}E_{r\tau}^{(m)} dr.
\end{align*}
Therefore, since $\lVert S_{m(h-r)}\rVert_{0\to 0}=1$,
\begin{align*}
	\lVert\widetilde{S}^m_{h\tau}-S_{hm}\rVert_{2\to 0} &\leq \int_0^h \bigg\lVert\sum_{k=1}^{mn}(\tau_k-1)V_k\bigg\rVert_{1\to 0}\left\lVert\widetilde{S}^m_{r\tau}\right\rVert_{2\to 1} dr + \int_0^h \lVert E_{r\tau}^{(m)}\rVert_{2\to 0} dr.
\end{align*}
Let $I_1(h)$ and $I_2(h)$ denote the first and second integrals, respectively. Then for any $\epsilon>0$,
\begin{align}\label{A.13}
	\mathbb{P}\left(\lVert\widetilde{S}^m_{h\tau}-S_{hm}\rVert_{2\to 0} > \frac{\epsilon}{m}\right) &\leq \mathbb{P}\left(I_1(h) > \frac{\epsilon}{2m}\right)+\mathbb{P}\left(I_2(h) > \frac{\epsilon}{2m}\right).
\end{align}
We consider the two probabilities on the right, called the \textit{first} and \textit{second probabilities}, separately. \\

\noindent\textit{First probability}. Note $\sum_{k=1}^{mn}(\tau_k-1)V_k=\sum_{k=1}^n\sum_{j=1}^m (\tau_j^{(k)}-1)V_k$ where $\tau^{(k)}_j\coloneqq \tau_{(j-1)n+k}$. So
\begin{align*}
	\bigg\lVert\sum_{k=1}^{mn}(\tau_k-1)V_k\bigg\rVert_{1\to 0} &\leq C_*\sum_{k=1}^n\bigg\lvert\sum_{j=1}^m \tau_j^{(k)}-1\bigg\rvert,
\end{align*}
and together with Lemma~\ref{lem:Bounds} and Equation~\eqref{eq:2to1},
\begin{align*}
	I_1(h) &\leq C_*\sum_{k=1}^n\bigg\lvert\sum_{j=1}^m \tau_j^{(k)}-1\bigg\rvert\int_0^h \prod_{k=1}^n e^{C_*r\sum_{j=1}^m \tau_j^{(k)}} dr\,.
\end{align*}
Therefore
\begin{align*}
	\mathbb{P}\left(I_1(h) > \frac{\epsilon}{2m}\right) &\leq \mathbb{P}\left(C_*\sum_{k=1}^n\bigg\lvert\sum_{j=1}^m \tau_j^{(k)}-1\bigg\rvert\int_0^h \prod_{k=1}^n e^{C_*r\sum_{j=1}^m \tau_j^{(k)}} dr > \frac{\epsilon}{2m}\right) \\
		&\leq \sum_{k=1}^n\mathbb{P}\left(\bigg\lvert\sum_{j=1}^m \tau_j^{(k)}-1\bigg\rvert\int_0^h \prod_{k=1}^n e^{C_*r\sum_{j=1}^m \tau_j^{(k)}} dr > \frac{\epsilon}{2C_*mn}\right).
\end{align*}
The second inequality follows from a union bound together with the fact that for any nonnegative random variables $X_k$ and constant $c$, $\{\sum_{k=1}^n X_k>c\}\subseteq\cup_{k=1}^n \{X_k>c/n\}$. Set
\begin{align*}
	A(h) \coloneqq \bigcap_{k=1}^n\left\{h\sum_{j=1}^m \tau_j^{(k)} \leq \alpha\right\}
		\quad\text{and}\quad
		B_k(h) &\coloneqq \left\{\bigg\lvert\sum_{j=1}^m \tau_j^{(k)}-1\bigg\rvert\int_0^h \prod_{k=1}^n e^{C_*r\sum_{j=1}^m \tau_j^{(k)}} dr > \frac{\epsilon}{2C_*mn}\right\}
\end{align*}
for arbitrary $\alpha>0$ and note that
\begin{align*}
	A(h)\cap B_k(h) &\subseteq \left\{\bigg\lvert\sum_{j=1}^m \tau_j^{(k)}-1\bigg\rvert he^{C_*n\alpha} > \frac{\epsilon}{2C_*mn}\right\}
		\eqqcolon B(h).
\end{align*}
Therefore
\begin{align*}
	\mathbb{P}\left(I_1(h) > \frac{\epsilon}{m}\right) &\leq \sum_{k=1}^n\mathbb{P}\left(B_k(h)\cap A(h)\right)+\mathbb{P}\left(B_k(h)\cap A(h)^c\right)
		\leq n\big[\mathbb{P}\left(B(h)\right)+\mathbb{P}\left(A(h)^c\right)\big].
\end{align*}
Set $h=t/m^2$. By Lemma~\ref{lemA2} for all $\epsilon>0$ such that $K\coloneqq\epsilon(2C_*tn)^{-1}e^{-C_*n\alpha}<1$,
\begin{align*}
	\mathbb{P}\left(B(h)\right) &= \mathbb{P}\left(\bigg\lvert\sum_{j=1}^m \tau_j^{(k)}-1\bigg\rvert > \frac{\epsilon m}{2C_*tne^{C_*n\alpha}}\right)
		\leq 2e^{-\frac{1}{2}K^2m}.
\end{align*}
And by Lemma~\ref{lemA1},
\begin{align*}
	\mathbb{P}\left(A(h)^c\right) &= \mathbb{P}\left(\bigcup_{k=1}^n\left\{\sum_{j=1}^m\tau_j^{(k)} > \frac{\alpha}{h}\right\}\right)
		\leq n\mathbb{P}\left(\sum_{j=1}^m \tau_j > \frac{\alpha m^2}{t}\right)
		\leq n2^me^{-\frac{1}{2}K'm^2}
\end{align*}
where $K'\coloneqq \alpha/t$. Therefore
\begin{align}\label{eq:I1Bound}
	\mathbb{P}\left(I_1(h) > \frac{\epsilon}{2m}\right) &\leq 2e^{-\frac{1}{2}K^2m}+2^m ne^{-\frac{1}{2}K'm^2}
		\leq 2^mCe^{-\frac{1}{2}Cm^2}
\end{align}
for some positive constant $C$ independent of $m$. \\

\noindent\textit{Second probability}. Recall $E_{r\tau}^{(m)}\coloneqq \sum_{k=1}^{mn}\tau_k[\widetilde{S}^{(1,k-1)}_{r\tau}, V_k]\widetilde{S}^{(k,mn)}_{r\tau}$. Also, from Equation~\eqref{eq:VarConst},
\begin{align*}
[\widetilde{S}_{r\tau}^{(1,k-1)},V_k]\widetilde{S}^{(k,mn)}_{r\tau} &= \sum_{j=1}^{k-1}\int_0^r\tau_j e^{(r-s)\sum_{j=1}^{k-1}\tau_jV_j}[V_j,V_k]\widetilde{S}^{(1,k-1)}_{s\tau}\widetilde{S}^{(k,mn)}_{r\tau}ds \\
		&\qquad+\sum_{j=1}^{k-1}\int_0^r \tau_j e^{(r-s)\sum_{j=1}^{k-1}\tau_jV_j}\big[[\widetilde{S}^{(1,j-1)}_{s\tau},V_j]\widetilde{S}^{(j,k-1)}_{s\tau},V_k\big]\widetilde{S}^{(k,mn)}_{r\tau}ds.
\end{align*}
Lemma~\ref{lem:Bounds} together with $\lVert [V_j,V_k]\rVert_{2\to 0} \leq \lVert V_j\rVert_{1\to 0}\lVert V_k\rVert_{2\to 1}+\lVert V_k\rVert_{1\to 0}\lVert V_j\rVert_{2\to 1}\leq 2C_*^2$ give
\begin{align*}
	\left\lVert [V_j,V_k]\widetilde{S}^{(1,k-1)}_{s\tau}\widetilde{S}^{(k,mn)}_{r\tau}\right\rVert_{2\to 0} &\leq 2C_*^2\left(1+C_*r\sum_{j=1}^{mn}\tau_j\right)e^{3C_*r\sum_1^{mn}\tau_j}.
\end{align*}
Also,
\begin{multline*}
	\big[[\widetilde{S}^{(1,j-1)}_{s\tau},V_j]\widetilde{S}^{(j,k-1)}_{s\tau},V_k\big] =\widetilde{S}^{(1,j-1)}_{s\tau}V_j\widetilde{S}^{(j,k-1)}_{s\tau}V_k-V_k\widetilde{S}^{(1,j-1)}_{s\tau}V_j\widetilde{S}^{(j,k-1)}_{s\tau} \\
	-V_j\widetilde{S}^{(1,k-1)}_{s\tau}V_k+V_kV_j\widetilde{S}^{(1,k-1)}_{s\tau}
\end{multline*}
together with Lemma~\ref{lem:Bounds} gives
\begin{align*}
	\left\lVert\big[[\widetilde{S}^{(1,j-1)}_{s\tau},V_j]\widetilde{S}^{(j,k-1)}_{s\tau},V_k\big]\widetilde{S}^{(k,mn)}_{r\tau}\right\rVert_{2\to 0} &\leq 4C_*^2\left(1+C_*r\sum_{j=1}^{mn}\tau_j\right)e^{3C_*r\sum_1^{mn}\tau_j}.
\end{align*}
Therefore for any $0\leq r\leq h$,
\begin{align*}
	\left\lVert E_{r\tau}^{(m)}\right\rVert_{2\to 0} &\leq \sum_{k=1}^{mn}\sum_{j=1}^{k-1}\tau_k\tau_j\int_0^r\left\lVert [V_j,V_k]\widetilde{S}^{(1,k-1)}_{s\tau}\widetilde{S}^{(k,mn)}_{r\tau}\right\rVert_{2\to 0} \\
		&\qquad\qquad\qquad\qquad\qquad +\left\lVert\big[[\widetilde{S}^{(1,j-1)}_{s\tau},V_j]\widetilde{S}^{(j,k-1)}_{s\tau},V_k\big]\widetilde{S}^{(k,mn)}_{r\tau}\right\rVert_{2\to 0} ds \\
		&\leq 6C_*^2r\bigg(1+C_*r\sum_{\ell=1}^{mn}\tau_\ell\bigg)e^{3C_*r\sum_1^{mn}\tau_\ell}\sum_{k=1}^{mn}\sum_{j=1}^{k-1}\tau_k\tau_j \\
		&\leq Ch\bigg(1+Ch\sum_{\ell=1}^{mn}\tau_\ell\bigg)e^{Ch\sum_1^{mn}\tau_\ell}\bigg(\sum_{k=1}^{mn}\tau_k\bigg)^2
\end{align*}
for some $C>0$. So, we have that
\begin{align*}
	I_2(h) &= \int_0^h\lVert E_{r\tau}^{(m)}\rVert_{2\to 0} dr
		\leq Ch^2\bigg(1+Ch\sum_{\ell=1}^{mn}\tau_\ell\bigg)e^{Ch\sum_1^{mn}\tau_\ell}\bigg(\sum_{k=1}^{mn}\tau_k\bigg)^2\,.
\end{align*}
For arbitrary $\alpha>0$, set
\begin{align*}
	A(h) &\coloneqq \left\{h\sum_{k=1}^{mn}\tau_k\leq \alpha\right\}
	\quad\text{and}\quad
	B(h)\coloneqq \left\{Ch^2\bigg(1+Ch\sum_{\ell=1}^{mn}\tau_\ell\bigg)e^{Ch\sum_1^{mn}\tau_\ell}\bigg(\sum_{k=1}^{mn}\tau_k\bigg)^2 > \frac{\epsilon}{2m}\right\}\,.
\end{align*}
Then taking $h=t/m^2$ as before,
\begin{equation}\label{eq:I2Bound}
\begin{aligned}
	\mathbb{P}\left(I_2(h) > \frac{\epsilon}{2m}\right) &= \mathbb{P}\left(A(h)\cap B(h)\right)+\mathbb{P}\left(A(h)^c\cap B(h)\right) \\
		&\leq \mathbb{P}\left(Ch^2\left(1+C\alpha\right)e^{C\alpha}\bigg(\sum_{k=1}^{mn}\tau_k\bigg)^2 > \frac{\epsilon}{2m}\right)+\mathbb{P}\left(h\sum_{k=1}^{mn}\tau_k>\alpha\right) \\
		&= \mathbb{P}\left(\sum_{k=1}^{mn}\tau_k > Km^{\frac{3}{2}}\right)+\mathbb{P}\left(\sum_{k=1}^{mn}\tau_k>\frac{\alpha m^2}{t}\right) \\
		&\leq n\left[\mathbb{P}\left(\sum_{k=1}^m\tau_k > K'm^{\frac{3}{2}}\right)+\mathbb{P}\left(\sum_{k=1}^m\tau_k>\frac{\alpha m^2}{nt}\right)\right] \\
		&\leq n\left(2^me^{-\frac{1}{2}K'm^{3/2}}+2^me^{-\frac{1}{2}K''m^2}\right) \leq 2^mC'e^{-\frac{1}{2}C'm^{3/2}}
\end{aligned}
\end{equation}
for some $C'>0$ where $K=(\epsilon(2t^2C(1+C\alpha)e^{C\alpha})^{-1})^{1/2}$, $K'=Kn^{-1}$, $K''=\alpha(nt)^{-1}$, and the second-to-last last inequality follows from Lemma~\ref{lemA1}. Combining~\eqref{A.13},~\eqref{eq:I1Bound}, and~\eqref{eq:I2Bound}  and taking $h=t/m^2$ we therefore have that for all $\epsilon$ sufficiently small,
\begin{align*}
	\mathbb{P}\left(\lVert\widetilde{S}^m_{t\tau/m^2}-S_{t/m}\rVert_{2\to 0} > \tfrac{\epsilon}{m}\right) &\leq 2^mC''e^{-\frac{1}{2}C''m^{3/2}}
\end{align*}
for some constant $C''>0$ independent of $m$. So, we have that
\begin{align*}
	\sum_{m=1}^\infty\mathbb{P}\left(\lVert\widetilde{S}^m_{t\tau/m^2}-S_{t/m}\rVert_{2\to 0} > \tfrac{\epsilon}{m}\right) &\leq \sum_{m=1}^\infty 2^mC''e^{-\frac{1}{2}C''m^{3/2}}
		< \infty. \QED
\end{align*}


\section{Controllability lemmas}


Combining the partial results obtained above we show the existence of transformations implementing the steps listed at the beginning of the section:

\begin{lemma}\label{l:contr1}
\om{If $q^{(0)}$ in $\mathcal{Q}_0$} is nondegenerate, then there exists $M_1$ and a sequence of transition times and interaction triples $\{ \iota(m),\tau(m)\}_{m = 1}^{M_1}$ such that $\Phi_{\tau(M_1)}^{\iota(M_1)}\circ\dots\circ \Phi_{\tau(1)}^{\iota(1)} (q^{(0)}) = q^{(1)}$ as in \eref{e:q0}.
\end{lemma}
\begin{proof}
If \eref{e:q0} is satisfied by $q^{(0)}$ we simply set $M_1 = 0$, $q^{(1)} = q^{(0)}$. If not, by \om{nondegeneracy} there exists a sequence of triples $\{\iota(m)\}_{m=1}^{M}$ with  $\iota(m) = \jj(m)\kk(m)\ll(m)$ such that $\AA_0 := \AA(q^{(0)})$ and $ \AA_m = \AA_{m-1} \oplus \ll(m)$ with $\{(0,1,+),(1,0,+),(j^*,-)\}\subset \AA_{M}$.
	We notice that all steps of this procedure satisfy, upon possibly reordering the indices within each triple, either the conditions of \lref{l:midmode} (b) or of \lref{l:samenorm}, so we sequentially choose $\tau(m) = \tau_+^{\iota(m)}$ from those lemmas.

	 To activate coordinate $(1,1,-)$ -- if this was not already done in the previous procedure -- we start with component $b_{j^*} \neq 0$ for $|j^*|\neq 1$ and consider a nearest neighbors path $\{\ell(n)\}_{n=1}^{M'}$ in $\ZN$ connecting $j^*$ to $(1,1)$ without performing any step on the axes. It is easy to see that such path can be realized through repeated application of \lref{l:midmode} (b)  by choosing for the $n$-th step the triples $\iota(n) = (0,1,+)(\ell(n),-)(\ell(n)\pm(0,1),-)$ or $\iota(n) = (1,0,+)(\ell(n),-)(\ell(n)\pm(1,0),-)$ for vertical and horizontal steps respectively.

   Finally, coordinates $(1,0,-)$ and $(0,1,-)$ can be activated by applying \lref{c:samenorm} to the triples $(1,0,-)(0,1,+)(1,1,-)$  and $(1,0,+)(0,1,-)(1,1,-)$ respectively, while $(1,1,+)$ is activated by  (b) by interchanging the type of modes $(1,1,-)$ and $(1,0,+)$ (or $(0,1,+)$)  in $\iota(M')$ from the previous paragraph to $(1,1,+)$ and $(1,0,-)$ (or $(0,1,-)$).
\end{proof}

\begin{lemma}\label{l:contr2}
\om{Let $q^{(1)}$ be a nondegenerate point in $\mathcal{Q}_0$ satisfying} \eref{e:q0}. Then there exists $M_2$ and a sequence of interacting triples and transition times  $\{\iota(m),\tau(m)\}_{m = 1}^{M_2}$ such that $\Phi_{\tau(M_2)}^{\iota(M_2)}\circ\dots\circ \Phi_{\tau(1)}^{\iota(1)} (q^{(1)}) = q^{(2)}$ \om{is a nondegenerate point in $\mathcal{Q}_0$} satisfying \eref{e:q20} and \eref{e:q21}.
\end{lemma}
\noindent \begin{proof}
In this part of the proof, we only consider interactions involving triples $\iota(m)$ of the form
\begin{equ}\label{e:nnjump}
 \Big\{ (0,1)(l,h)(l,h\pm 1)\text{ or }
 (1,0)(l,h)(l\pm 1,h) : \text{ $|l|,|h|\leq N,~|(l,h)|\neq 1$} \Big\}\,.
\end{equ}
By \lref{l:midmode} (a), if $|j|<|k|<|\ell|$ and $(0,1),(l,h) \in \AAc(q)$ there exists $\tau(m) = \tau_-^{\iota(m)}$ such that defining $\AA_m = \AAc(\phi^{\iota(m)}_{\tau(m)}(q))$ we have  $(l,h) \not \in \AA_m$ and $(0,1) \in \AA_m$ (and similarly for $(1,0)$)
\footnote{Note that the same result can trivially be obtained if $(l,h)\not \in \AAc(q)$ setting $\tau_-^{\iota(m)}=0$}. Note that while a triple as above satisfies by assumption that $|j|<|k|<|\ell|$ and at least two of its coordinates are nonvanishing, it does not, in general, satisfy \eref{e:degcond}. However, assuming that $q$ does not satisfy \eref{e:degcond}, by \lref{l:samenorm} and setting $\iota' = (1,0)(0,1)(1,1)$,
there exists $\tau^{\iota'}$ such that $  |q_{(1,0)}| \neq |(\Phi_{\tau^{\iota'}}^{\iota'}(q))_{(1,0)}| >0$.
Since none of the coordinates in $\ZN\setminus \{(1,0)(0,1)(1,1)\}$ are affected by this operation, $(\Phi_{\tau^{\iota'}}^{\iota'}(q))$ satisfies \eref{e:degcond} and \lref{l:midmode} can be applied to this state.

To conclude the proof we identify a sequence of triples $\iota(m) = (j(m),k(m),\ell(m))\in \II$ of the form \eref{e:nnjump} such that for $\AA_0 = \AA(q^{(1)}) \subseteq \ZNp$
\begin{equ}
(((\AA_0 \ominus k(1)) \ominus k(2)) \ominus \dots) \ominus k(M_2) = \{(1,0,\chi),(0,1,\chi), (1,1\chi), (N,N\chi), (-N,N\chi)\,, \chi \in \{+,-\}\}\,.
	\end{equ}
	A possible such sequence is given by triples of the form
        \begin{align*}
         \Big\{ (1,0,+)(l,h, \chi)(l+1,h,\chi)~ : ~ (l,h) \in
          \{(0,2),\dots, (0,N)\} \,,\chi\in \{+,-\}\Big\}
        \end{align*}
to remove the vertical column of $\ZN$ (which cannot interact with $(0,1)$),  followed by
{\small
\begin{equ}
\Big\{	 \Big((0,1,+)(l,h,\chi)(l,h+1,\chi)~:~ (l,h) \in
\big\{(l,0),\dots, (l,N): |l|  \in (1,\dots,N-1) \big\}\setminus\{(1,1)\}\Big)\,,\chi\in \{+,-\}\Big\}\,,
\end{equ}
}where importantly the set of transitions for each $l$ is ordered. The above transformation zeroes all coefficients except those in the set $\{(1,1),(0,1),(1,0)\}\cup\{(l,N)~:~l \in (-N,\dots, N)\}$. We further remove the coefficients from $\{(l,N)~:~l \in (-N+1,\dots, N-1)\}$ by sequentially applying Lemma~\ref{l:midmode} to the ordered sequence of interacting triples
\begin{equ}\Big((1,0,+)(l,h,\chi)(l+1,h,\chi) ~:~ (l,h) \in \{(0,N),\dots, (N-1,N)\,,\chi\in \{+,-\}\}\Big)\,,\end{equ}
and then
\begin{equ}\Big((1,0,+)(l,h,\chi)(l-1,h,\chi)~:~ (l,h) \in \{(-1,N),\dots, (-N+1,N)\}\,,\chi\in \{+,-\}\Big)\,.\end{equ}
It is easy to check that each transition in the above construction sequentially satisfies the assumptions of \lref{l:midmode} (a), and that once a mode has been removed from $\AA$ it will not interact again in this procedure.
The fact that \eref{e:q21} holds follows from \eref{e:q0} and that in an interacting triple $\iota = \jkl$ with $|j|<|k|<|l|$ both modes $\jj$ and $\ll$ are in $\mathcal A$ at the end of the interaction by $\tau_-^\iota$.
\end{proof}

\begin{lemma}\label{l:contr3}
		Let \om{$q^{(2)}$ be a nondegenerate point in $\mathcal{Q}_0$ satisfying} \eref{e:q20} and \eref{e:q21}. Then there exists $M_3$ and a sequence of  interacting triples and transition times $\{\iota(m),\tau(m)\}_{m = 1}^{M_3}$ such that $\Phi_{\tau(M_3)}^{\iota(M_3)}\circ\dots\circ \Phi_{\tau(1)}^{\iota(1)} (q^{(2)}) = q^{(3)}$ \om{is a nondegenerate point in $\mathcal{Q}_0$} satisfying \eref{e:q30} and \eref{e:q31}.
\end{lemma}

Since it may not be possible to ``transfer'' the content of \emph{e.g.}, mode $(-N,N)$ to $(-N+1,N)$ through one single interaction with mode $(1,0)$ -- and therefore it won't be possible to transfer the amplitude of mode $(-N,N)$ to $(N,N)$ in one single ``pass'' -- we proceed to prove that, through a sequence of interactions, we can transfer a \emph{finite} and $q_{(-N,N)}$-independent amount of energy from mode $(-N,N)$ to $(N,N)$. Therefore, the transfer of amplitude from mode $(-N,N)$ to $(N,N)$ may be accomplished by repeating this sequence of interactions sufficiently many times.

The following corollary of \lref{l:samenorm} will be instrumental for the proof of \lref{l:contr3}:
\begin{corollary}\label{c:S4}
	Let $q_{(1,1)}, \qb{{(1,1)}} \neq 0$ then for any $q,q'$ with $q_\jj=q_\jj'$ for all $|j|>1$ there exist a sequence $\{\iota(m), \tau(m)\}_{m=1}^4$ such that $\Phi_{\tau(4)}^{\iota(4)} \circ\dots\circ \Phi_{\tau(1)}^{\iota(1)} (q) = q'$.
\end{corollary}

\noindent \begin{proof}[Proof of \lref{l:contr3}]
The desired result follows upon showing that for any $i \in \{-N,\dots,N\}$, setting $\ll = (-i,N,\chi), \ll' = {(i,N,\chi')}$ for $\chi, \chi' \in \{-,+\}$ there exists $M_{\ll, \ll'}$ and a sequence of triples and interaction times $\{\iota(m)$,  $\tau(m)\}_{m = 1}^{M_{\ll, \ll'}}$ such that for any $q$ satisfying $\bigcup_{|i'| < i}\{(i',N,+),{(i',N,-)}\} \cap \AAc(q) = \emptyset$
and $q' = \Phi_{\tau(M_{\ll, \ll'})}^{\iota(M_{\ll, \ll'})}\circ \dots \circ \Phi_{\tau(1)}^{\iota(1)}(q)$ we have
\begin{equ}\label{e:qcond}
	q_\jj' = \begin{cases} q_\jj \qquad & \text{for } \jj \in \ZN \setminus \{\ll,\ll'\}\,,\\
	0&\text{for } \jj = \ll \text{ if } \ll \neq \ll'\,,\end{cases}
\end{equ}
and for $\kk \in \{\ll, \ll'\}$,  $\text{sign}(q_\kk) = \text{sign}(q_\kk')$ holds if $q_\kk' \neq 0$ (recalling our choice of notation $\text{sign}(0)=+1$).
Indeed, if $\text{sign}(\qb{{(N,N)}})\geq 0$ we sequentially apply the above result to the pairs \begin{equ}(\ll,\ll') = ((N,N,+), (-N,N,+)), ((-N,N,+), (N,N,-)), ((-N,N,-), (N,N,-))\,.\end{equ}
Otherwise, when
$\text{sign}(\qb{{(N,N)}})=-1$
we first apply the above result to
$\ll={(N,N,-)}$, $\ll'={(-N,N,-)}$ and then proceed as in the previous case.

We prove the result above by induction on $i \in \{0,\dots,N\}$. The proof for $i \leq 0$ is analogous.

{\bf Base case ($i=0: (0,N,\chi)\to (0,N,\chi')$)}: If $\ll = \ll'$ there is nothing to show. We proceed to consider the case $\ll = (0,N,+)$, $\ll'={(0,N,-)}$, as the converse follows by analogous arguments. In this case, for a sufficiently small $\epsilon>0$ we consider the interactions $\iota = (1,0,+)(0,N,+)(1,N,+)$ and $\iota'= {(1,0,-)}{(0,N,-)}(1,N,+)$, running the corresponding flow maps by a small amount of time $\tau(\epsilon)$, $\tau'(\epsilon)$ such that $(\Phi_{\tau'(\epsilon)}^{\iota'} \circ \Phi_{\tau(\epsilon)}^\iota(q)_{{(0,N,-)}})^2 = \qb{{(0,N)}}^2+\epsilon$. We then apply \Cref{c:S4} to the coordinates $(1,0,+),{(1,0,-)}$ to return them in the initial configuration. Note that the existence of a uniform $\epsilon>0$ such that the transitions above can be performed in a single pair of interactions (and therefore the finiteness of the total number of interactions required to perform the desired transformation) follows from the fact that $\qb{{(0,N)}}$ is nondecreasing and the continuity of the dynamics together with \lref{l:midmode}.

{\bf Induction step ($i>0: (-i,N,\chi)\to (i,N,\chi')$)}:  We consider two possibilities for $q$:
a) there exists $q''$ with $|a_{(1,0)}''| \in [|a_{(1,0)}|/2,|a_{(1,0)}| ]$, $q_{(-i,N,\chi)}''=0$ and for $\iota'' = (1,0,+)(-i+1,N,\chi)(-i,N,\chi)$
	\begin{equs}
		E_{\iota''}(q) &= E_{\iota''}(q'')\,,\qquad
		\EE_{\iota''}(q) &= \EE_{\iota''}(q'')\,,
	\end{equs}
or b) such $q''$ does not exist.

In case a) the state $q''$ can be reached by letting $\iota = (1,0,+)(-i+1,N,\chi)(-i,N,\chi)$ interact for a finite amount of time $\tau$ from \lref{l:midmode} (c). Then, by the induction assumption there is a sequence of triples and interaction times allowing to reach a state $q'''$ with $q_{(-i+1,N,\chi)}''' = 0$, $q_{(i-1,N,\chi')}''' = q_{(-i+1,N,\chi)}''$ and $q_\jj''' = q_\jj''$ for all other $j \in \ZN$. The desired state can then be reached by application of \lref{l:midmode} (a) to the triple $\iota = (1,0,+)(i-1,N,\chi')(i,N,\chi')$\,. We proceed to check that the final state satisfies \eref{e:qcond}. Because modes $j \not \in \{(-i,N), \dots, (i,N), (1,0)\}$ did not interact in the procedure above for such $\jj$ we must have that  $q_\jj = q_\jj'$. The fact that for $j \in \{(-i,N), \dots, (i-1,N)\}$ $q_\jj' = 0$ follows by construction and the induction assumption. It remains to check that $|\qa{(1,0)}'| = |\qa{(1,0)}|$. Since the only modes affected by the above transformation are $(-i,N,\chi),(i,N\chi'),(1,0,+)$, this follows directly by   conservation of energy and enstrophy:
\begin{equs}
	(q_{(-i,N,\chi)})^2 + (q_{(i,N,\chi')})^2 + ({q_{(1,0,+)}})^2 & = (q_{(i,N,\chi')}')^2  + (q_{(1,0,+)}')^2\,,\\
			\frac{(q_{(-i,N,\chi)})^2}{N^2+i^2} + \frac{(q_{(i,N,\chi')})^2}{N^2+i^2} + ({q_{(1,0,+)}})^2 & = \frac{(q_{(-i,N,\chi)}')^2}{N^2+i^2} +  (q_{(1,0,+)}')^2\,.
\end{equs}

In case b) we proceed to show that case a) can be reached with a finite number of interactions. More specifically if condition a) is not satisfied we let the triple $\iota'' = (-i,N,\chi)(-i+1,N, \chi)(1,0, +)$ for $\chi \in \{+,-\}$ interact as described by \lref{l:midmode}
 for a time $\tau''$ to reach \om{a nondegenerate point $q''$ in $\mathcal{Q}_0$} with $q_\jj'' = q_\jj$ for $\jj \not \in \{(-i,N, \chi),(-i+1,N,\chi),(1,0,+)\}$, $\qa{(1,0)}'' = \qa{(1,0)}/2$ and $q_{(-i,N,\chi)}'',q_{(-i+1,N,\chi)}''$ satisfying the conservation laws
\begin{equs}
	(q_{(-i,N, \chi)})^2  + ({q_{(1,0,+)}})^2 & = (q_{(-i,N, \chi)}'')^2 + (q_{(-i+1,N,\chi)}'')^2 + (q_{(1,0,+)}/2)^2\,,\\
       \frac{(q_{(-i,N, \chi)})^2}{N^2+i^2}  + ({q_{(1,0,+)}})^2 & = \frac{(q_{(-i,N,\chi)}'')^2}{N^2+i^2} +\frac{(q_{(-i+1,N,\chi)}'')^2}{N^2+(i-1)^2}+  (q_{(1,0,+)}/2)^2\,,
\end{equs}
so that
$
	(q_{(-i,N,\chi)}'')^2 = (q_{(-i,N,\chi)})^2  - C_{N,i} (q_{(1,0)})^2
$
for $C_{N,i} = \frac 3 4  \frac{N^2+i^2}{i^2-(i-1)^2} ( N^2+(i-1)^2-1)$. We see that a positive, $q_{(1,0,+)}$-dependent amplitude is removed from $(q_{(-i,N,\chi)})^2$. Again applying the induction step and \lref{l:midmode} (a) to transfer, respectively, the amplitude from ${(-i+1,N,\chi)}$ to ${(i-1,N, \chi')}$ and from ${(i-1,N,\chi')}$ to ${(i,N,\chi')}$ we reach the state $q'$ with $q_\jj = q_\jj'$ for modes $j \not \in \{(-i,N), \dots, (i,N), (1,0)\}$ (since these modes either vanish in both cases or they did not interact). Further, by conservation of energy and enstrophy, we have that
\begin{equs}
	(q_{(-i,N,\chi)})^2  + (q_{(i,N,\chi')})^2+ ({q_{(1,0,+)}})^2 & = (q_{(-i,N,\chi)}'')^2 + (q_{(i,N,\chi')}'')^2 + (q_{(1,0,+)}'')^2\,,\\
			\frac{(q_{(-i,N,\chi)})^2}{N^2+i^2}  + \frac{(q_{(i,N,\chi')})^2}{N^2+i^2} + ({q_{(1,0,+)}})^2 & = \frac{(q_{(-i,N,\chi)}'')^2}{N^2+i^2} +\frac{(q_{(i,N,\chi')}'')^2}{N^2+i^2}+  (q_{(1,0,+)}'')^2\,,
\end{equs}
so that $|q_{(1,0,+)}''| = |q_{(1,0,+)}|$. This shows that the amplitude $C_{N,i} (q_{(1,0,+)})^2$ subtracted to $q_{(-i,N,\chi)}$ is constant at each cycle, showing by boundedness of $q_{(-i,N,\chi)}$ that with a finite number of iterations as the one described above we can reach state a), concluding the proof.
\end{proof}

\begin{lemma}\label{l:contr4}
		Let \om{$q^{(3)}$ be a nondegenerate point in $\mathcal{Q}_0$ satisfying} \eref{e:q30} and \eref{e:q31}. Then there exists $M_4$ and a sequence of interacting triples and transition times $\{\iota(m),\tau(m)\}_{m = 1}^{M_4}$ such that $\Phi_{\tau(M_4)}^{\iota(M_4)}\circ\dots\circ \Phi_{\tau(1)}^{\iota(1)} (q^{(3)}) = q^*$ \om{is a nondegenerate point in $\mathcal{Q}_0$} satisfying \eref{e:qstar}.
\end{lemma}
\noindent\begin{proof}
	We start the proof by applying \Cref{c:S4} to transform the state $q^{(3)}$ into $q = \Phi_{\tau(1)}(q^{(3)})$ satisfying $q_\jj^{(3)} = q_\jj$ for all $|j|>1$ and $\qa{(0,1)}=\qb{{(0,1)}}=\qb{{(1,0})}=\qa{(1,0)}>0$. Throughout this proof, we refer to states $q$ such that $q_{(i,i',\chi)} = q_{(i',i,\chi)}$ for all $i,i' \in (0,\dots, N)$, $\chi \in \{+,-\}$  as \emph{symmetric}.

We then proceed to transfer the amplitude from $\qa{(1,1)}$ to $\qb{(2,1)}, \qb{(1,2)}$ by transforming $q$ into another symmetric state $q'$ with ${(2,1,-)}, {(1,2,-)} \in \AAc(q')$ and $(1,1,+) \not \in \AAc(q')$. This can be done by letting triples $\iota(2) = {(1,0,-)}(1,1,+){(2,1,-)}\in \II$ and $\iota(3) = {(0,1,-)}(1,1,+){(1,2,-)}\in \II$
interact, and choosing the interaction times $\tau, \tau'(\tau)$ such that $\Phi_{\tau'(\tau)}^{\iota(3)}\circ \Phi_{\tau}^{\iota(2)}(q)_{(1,1 ,+)}=0$. Further, we note that the difference $\qb{{(1,2)}}'-\qb{{(2,1)}}'$
is negative for $\tau=0$, positive for $\tau'(\tau)=0$ and is continuous in $\tau$, so there must exist $\tau^*$ such that $\qb{{(1,2)}}'=\qb{{(2,1)}}'$.
To show that $q'$ is symmetric it only remains to show that $\qb{{(1,0)}}' = \qb{{(0,1)}}'$. This follows from the conservation laws:
{\small
\begin{equ}
 	B_{(1,0)(1,1)} \pc{(\qb{{(1,0)}}')^2-(\qb{{(1,0)}})^2} = B_{(2,1)(1,1)} (\qb{{(2,1)}}')^2 =	B_{(1,2)(1,1)} (\qb{{(1,2)}}')^2= B_{(0,1)(1,1)} \pc{(\qb{{(0,1)}}')^2-(\qb{{(0,1)}})^2}
\end{equ}
}where
\begin{equ}
	B_{jk} \eqdef \frac 1 {|j|^2} - \frac 1 {|k|^2}\,.
\end{equ}

Next, we let the triples $\iota(4) = (1,0,-)(0,1,+)(1,1,-)$ and $\iota(5) = (0,1,-)(1,0,+)(1,1,-)$ interact. By \lref{l:samenorm} there exists an interaction time such that the initial state $q'$ is mapped to  $q''$ with $\qb{{(1,0)}}'' = \qb{{(0,1)}}'' = 0$ and $\qa{{(1,0)}}'' = \qa{{(0,1)}}'' > 0$, so that ${(1,0,-)}, {(0,1,-)}\not \in \AAc(q'')$.

We then proceed to transfer the amplitude from modes ${(1,2,-)}$ and ${(2,1,-)}$ to ${(2,2,-)}$. This is done letting triples $\iota(6) = (1,0,+){(1,2,-)}{(2,2,-)}$ and $\iota(7) = (0,1,+){(2,1,-)}{(2,2,-)}$ interact until the modes ${(2,1,-)},{(1,2,-)}$ are depleted, as proved in \lref{l:midmode}. The symmetry of the final state $q'''$ is  again a consequence of the conservation laws:
{\small
\begin{equ}
 	B_{{(1,0)}{(2,2)}} \pc{(\qa{{(1,0)}}''')^2-(\qa{{(1,0)}}'')^2} = B_{{(2,1)}{(2,2)}} (\qb{{(2,1)}}'')^2 =	B_{{(1,2)}{(2,2)}} (\qb{{(1,2)}}'')^2= B_{{(0,1)}{(2,2)}} \pc{(\qa{{(0,1)}}''')^2-(\qa{{(0,1)}}'')^2}\,.
\end{equ}
}Summarizing, we have reached a symmetric state $q''' = \Phi_{\tau(7)}^{\iota(7)}\circ\dots\circ \Phi_{\tau(2)}^{\iota(2)}(q)$ with \begin{equ}
\AA(q''') = \{(1,0,+), (0,1,+), {(2,2,-)}, {(1,1,-)}, {(N,N,-)}\}\,.
\end{equ}

	The desired result then follows immediately if we can show that we can transfer the amplitude of mode $(i-1,i-1,-)$ to $(i,i,-)$ for $i\in (2,\dots, N)$ while preserving the fact that $\qa{(1,0)}' = \qa{(0,1)}'$. We show this by considering, sequentially, the interaction triples
  \begin{align*}
    &\iota(4i) = (1,0,+)(i-1,i-1,-)(i,i-1,-)\,,\qquad &\iota(4i+1)=(0,1,+)(i-1,i-1,-)(i-1,i,-)\,,\\
    &\iota(4i+2)=(0,1,+)(i,i-1,-)(i,i,-)\,,\qquad &\iota(4i+3)=(1,0,+)(i-1,i,-)(i,i,-)\,.\qquad\;\;\;\;\;
  \end{align*}
    More specifically, we consider the family of endpoints
	\begin{equ}
		q''(t) = \Phi_{\tau_-^{\iota(4i+3)}}^{\iota(4i+3)}\circ \Phi_{\tau_-^{\iota(4i+2)}}^{\iota(4i+2)}\circ \Phi_{\tau_-^{\iota(4i+1)}}^{\iota(4i+1)}\circ\Phi_{t}^{\iota(4i)}(q')\,,
	\end{equ}
	where $\tau_-^{\iota}$ is defined in \lref{l:midmode} (a). By construction, this sequence implies that $\qa{(i-1,i-1)}''= \qa{(i-1,i)}''= \qa{(i,i-1)}''=0$ and $\qa{(i,i)}''\neq 0$. It remains to prove that $\qa{(1,0)}'' = \qa{(0,1)}''$.
  As a composition of continuous functions, $q''(t)$ is continuous in $t$ and therefore so is $\Delta q(t) = \qa{(1,0)}''(t) - \qa{(0,1)}''(t)$. Further, since by symmetry $\qa{(1,0)}''(0) = \qa{(0,1)}''(\tau_-^{\iota(4i)})$, we must have $\mathrm{sign}( \Delta q(0)) = -\mathrm{sign}(\Delta q(\tau_-^{\iota_1})) $.
  This implies the existence of $\tau(4i) \in [0,\tau_-^{\iota_1}]$ with $\Delta q(0)=0$, concluding the proof.
\end{proof}

\end{document}